\documentclass[reqno]{amsart} 
\usepackage{amsthm, amssymb, amsfonts}
\usepackage{lmodern}
\usepackage{color} 
\usepackage[T5]{fontenc}
\usepackage{amscd,amssymb}
\usepackage[v2,cmtip]{xy}
\usepackage{mathrsfs}
\theoremstyle{plain}
\newtheorem{thm}{Theorem}[section]
\newtheorem{prop}[thm]{Proposition}
\newtheorem{lem}[thm]{Lemma}
\newtheorem{con}[thm]{Conjecture}
\newtheorem{corl}[thm]{Corollary}
\theoremstyle{definition}
\newtheorem{defn}[thm]{Definition}

\newtheorem{nota}[thm]{Notation}

\theoremstyle{plain}
\newtheorem{thms}{Theorem}[subsection]
\newtheorem{props}[thms]{Proposition}
\newtheorem{lems}[thms]{Lemma}
\newtheorem{corls}[thms]{Corollary}
\theoremstyle{definition}

\allowdisplaybreaks

\begin{document}

\title[The squaring operation and the Singer algebraic transfer]
{The squaring operation and\\ the Singer algebraic transfer}

 \author{Nguy\~\ecircumflex n Sum}

\address{Department of Mathematics and Application, S\`ai G\`on University,
273 An D\uhorn \ohorn ng V\uhorn \ohorn ng, District 5, H\`\ocircumflex\ Ch\'i Minh city, Viet Nam}

\email{nguyensum@qnu.edu.vn}





\begin{abstract}
Let $P_k$ be the graded polynomial algebra $\mathbb F_2[x_1,x_2,\ldots ,x_k]$, with the degree of each $x_i$ being 1, regarded as a module over the mod-2 Steenrod algebra $\mathcal A$, and let $GL_k$ be the general linear group over the prime field $\mathbb F_2$ which acts regularly on $P_k$. We study the algebraic transfer constructed by Singer using the technique of the \textit{hit problem}. This transfer is a homomorphism from the homology of the mod-2 Steenrod algebra, $\text{Tor}^{\mathcal A}_{k,k+d} (\mathbb F_2,\mathbb F_2)$, to the subspace of $\mathbb F_2{\otimes}_{\mathcal A}P_k$ consisting of all the $GL_k$-invariant classes of degree $d$. 

In this paper, we extend a result of H\uhorn ng on the relation between the Singer algebraic transfer and the classical squaring operation on the cohomology of the Steenrod algebra. Using this result, we show that Singer's conjecture for the algebraic transfer is true in the case $k=5$ and the degree $5(2^{s} -1)$ with $s$ an arbitrary positive integer.
\end{abstract}

\maketitle

\section{Introduction}\label{s1} 
\setcounter{equation}{0}

Let $V_k$ be an elementary abelian 2-group of rank $k$ and let $BV_k$ be the classifying space of $V_k$.  Then, 
$$P_k:= H^*(BV_k) \cong \mathbb F_2[x_1,x_2,\ldots ,x_k],$$ a polynomial algebra in  $k$ generators $x_1, x_2, \ldots , x_k$, each of degree 1. Here the cohomology is taken with coefficients in the prime field $\mathbb F_2$ of two elements. 

Being the cohomology of a topological space,  $P_k$ is a module over the mod-2 Steenrod algebra, $\mathcal A$.  The action of $\mathcal A$ on $P_k$ is determined by the elementary properties of the Steenrod squares $Sq^i$ and subject to the Cartan formula (see Steenrod and Epstein~\cite{st}).

A polynomial $g$ in $P_k$ is called \textit{hit} if it can be written as a finite sum $g = \sum_{j > 0}Sq^{j}(g_j)$ for suitable polynomials $g_j \in P_k$. That means $g$  belongs to  $\mathcal{A}^+P_k$, where $\mathcal{A}^+$ is the augmentation ideal of $\mathcal A$. 

The \textit{hit problem} is to find a minimal generating set for $P_k$ regarded as a module over the  mod-2 Steenrod algebra. Equivalently, we want to find a vector space basis for $\mathbb F_2 \otimes_{\mathcal A} P_k$ in each degree $d$. Such a basis may be represented by a list of monomials of degree $d$.

The hit problem was first studied by Peterson~\cite{pe}, Wood~\cite{wo}, Singer~\cite {si1}, and Priddy \cite{pr}, who showed its relation to several classical problems in the homotopy theory. Then, this problem  was investigated by Carlisle and Wood~\cite{cw}, Crabb and Hubbuck~\cite{ch}, Janfada and Wood~\cite{jw1}, Kameko~\cite{ka}, Mothebe \cite{mo}, Nam~\cite{na}, Ph\'uc and Sum \cite{sp,sp2}, Silverman~\cite{sl}, Silverman and Singer~\cite{ss}, Singer~\cite{si2}, Walker and Wood~\cite{wa1,wa2,wa3}, Wood~\cite{wo,wo2} and others.

The vector space $\mathbb F_2 \otimes_{\mathcal A} P_k$ was explicitly calculated by Peterson~\cite{pe} for $k=1, 2,$ by Kameko~\cite{ka} for $k=3$, and recently by the present author \cite{su3,su1}  for $k=4$, unknown in general. 

\smallskip
Let $GL_k$ be the general linear group over the field $\mathbb F_2$. Since $V_k$ is an $\mathbb F_2$-vector space of dimension $k$, this group acts naturally on $V_k$ and therefore on the cohomology of $BV_k$. The two actions of $\mathcal A$ and $GL_k$ upon $P_k$ commute with each other. Hence, there is an inherited action of $GL_k$ on $\mathbb F_2\otimes_{\mathcal A}P_k$. 

For a non-negative integer $d$, denote by $(P_k)_d$ the subspace of $P_k$ consisting of all the homogeneous polynomials of degree $d$ in $P_k$ and by $(\mathbb F_2\otimes_{\mathcal A}P_k)_d$ the subspace of $\mathbb F_2\otimes_{\mathcal A}P_k$ consisting of all the classes represented by the elements in $(P_k)_d$. 
In \cite{si1}, Singer defined the algebraic transfer,  which is a homomorphism
$$\varphi_k :\text{Tor}^{\mathcal A}_{k,k+d} (\mathbb F_2,\mathbb F_2) \longrightarrow  (\mathbb F_2\otimes_{\mathcal A}P_k)_d^{GL_k}$$
from the homology of the Steenrod algebra, $\text{Tor}^{\mathcal A}_{k,k+d} (\mathbb F_2,\mathbb F_2)$, to the subspace of $(\mathbb F_2{\otimes}_{\mathcal A}P_k)_d$ consisting of all the $GL_k$-invariant classes. 
The Singer algebraic transfer is a useful tool in describing the homology groups of the Steenrod algebra. It was studied by many authors (see Boardman ~\cite{bo}, Bruner-H\`a-H\uhorn ng ~\cite{br}, H\`a ~\cite{ha}, H\uhorn ng ~\cite{hu2, hu3}, Ch\ohorn n-H\`a~ \cite{cha,cha1,cha2}, Minami ~\cite{mi}, Nam~ \cite{na2}, H\uhorn ng-Qu\`ynh~ \cite{hq}, Qu\`ynh~ \cite{qh}, the present author \cite{su5} and others).

It was shown that the algebraic transfer is an isomorphism for $k=1,2$ by Singer in \cite{si1} and for $k=3$ by Boardman in \cite{bo}. However, for any $k\geqslant 4$,  $\varphi_k$ is not a monomorphism in infinitely many degrees (see Singer \cite{si1}, H\uhorn ng \cite{hu3}.) Singer made the following conjecture.
\begin{con}[Singer \cite{si1}] \label{gtsing} The algebraic transfer $\varphi_k$ is an epimorphism for any $k \geqslant 0$.
\end{con}

 The conjecture is true for $k\leqslant 3$. Based on the results in \cite{su3,su1}, we are verifying this conjecture for $k=4$. We hope that it is also true in this case. However, for $k \geqslant 5$, the conjecture is still open.

There is a classical operator, known as Kameko's squaring operation
$$\widetilde{Sq}^0_*: (\mathbb F_2 \otimes_{\mathcal A} P_k)_{2d+k} \longrightarrow (\mathbb F_2 \otimes_{\mathcal A} P_k)_d,$$ 
which is induced by an $\mathbb F_2$-linear map $\phi: P_k \to P_k$ given by
$$
\phi(x) = 
\begin{cases}y, &\text{if }x=x_1x_2\ldots x_ky^2,\\  
0, & \text{otherwise,} \end{cases}
$$
for any monomial $x \in P_k$. Note that $\phi$ is not an $\mathcal A$-homomorphism. However, 
$\phi Sq^{2i} = Sq^{i}\phi$ and $\phi Sq^{2i+1} = 0$
for any non-negative integer $i$.
Since $\widetilde{Sq}^0_*$ is a homomorphism of $GL_k$-modules, it induces a homomorphism which is also denoted by
$\widetilde{Sq}^0_*: (\mathbb F_2 \otimes_{\mathcal A} P_k)^{GL_k}_{2d+k} \longrightarrow (\mathbb F_2 \otimes_{\mathcal A} P_k)^{GL_k}_d.$
 It was recognized by Boardman \cite{bo} for k = 3 and by Minami \cite{mi} for general $k$ that Kameko's squaring operation commutes with the dual of the classical squaring operation on the cohomology of the Steenrod algebra, $Sq^0: \text{Ext}_{\mathcal A}^{k,d+k}(\mathbb F_2,\mathbb F_2) \to \text{Ext}_{\mathcal A}^{k,2d+2k}(\mathbb F_2,\mathbb F_2)$, through the Singer algebraic transfer. This means that the following diagram is commutative:
$$
\xymatrix{\text{Tor}^{\mathcal A}_{k,2d+2k}(\mathbb F_2,\mathbb F_2)\ar[rr]^{\varphi_k} \ar[d]^{{Sq}^0_*}&&  (\mathbb F_2{\otimes}_{\mathcal A} P_k)_{2d+k}^{GL_k} \ar[d]^{\widetilde{Sq}_*^0}\\
\text{Tor}^{\mathcal A}_{k,d+k}(\mathbb F_2,\mathbb F_2) \ar[rr]^{\varphi_k} && (\mathbb F_2{\otimes}_{\mathcal A} P_k)_{d}^{GL_k}.}
$$

For a positive integer $n$, by $\mu(n)$ one means the smallest number $r$ for which it is possible to write $n = \sum_{1\leqslant i\leqslant r}(2^{u_i}-1),$ where $u_i >0$. 

\begin{thm}[Kameko~\cite{ka}]\label{dlk} 
Let $d$ be a non-negative integer. If $\mu(2d+k)=k$, then 
$$\widetilde{Sq}^0_*: (\mathbb F_2 \otimes_{\mathcal A} P_k)_{2d+k}\longrightarrow (\mathbb F_2 \otimes_{\mathcal A} P_k)_d$$
is an isomorphism of $GL_k$-modules. 
\end{thm}

From the result of Carlisle and Wood \cite{cw} on the boundedness conjecture, H\uhorn ng observes in \cite{hu3} that, for any degree $d$, there exists a non-negative integer $t$ such that
$$(\widetilde{Sq}^0_*)^{s-t}: (\mathbb F_2 \otimes_{\mathcal A} P_k)_{k(2^s-1) + 2^sd} \longrightarrow (\mathbb F_2 \otimes_{\mathcal A} P_k)_{k(2^t-1) + 2^td}$$
is an isomorphism of $GL_k$-modules for every $s \geqslant t$. However, this result does not
confirm how large $t$ should be.

Denote by  $\alpha(n)$ the number of ones in dyadic expansion of a positive integer $n$ and by $\zeta(n)$ the greatest integer $u$ such that $n$ is divisible by $2^u$. That means $n = 2^{\zeta(n)}m$ with $m$ an odd integer.  We set 
$$t(k,d) = \max\{0,k- \alpha(d+k) -\zeta(d+k)\}.$$ 
The following is one of our main results.
\begin{thm}\label{cthm} Let $d$ be an arbitrary non-negative integer. Then
$$(\widetilde{Sq}^0_*)^{s-t}: (\mathbb F_2 \otimes_{\mathcal A} P_k)_{k(2^s-1) + 2^sd} \longrightarrow (\mathbb F_2 \otimes_{\mathcal A} P_k)_{k(2^t-1) + 2^td}$$
is an isomorphism of $GL_k$-modules for every $s \geqslant t$ if and only if $t \geqslant  t(k,d)$.
\end{thm}

It is easy to see that $t(k,d) \leqslant k-2$ for every $d$ and $k \geqslant 2$. Hence, one gets the following.
\begin{corl}[H\uhorn ng \cite{hu3}]\label{chq} Let $d$ be an arbitrary non-negative integer. If $k \geqslant 2$, then
$$(\widetilde{Sq}^0_*)^{s-k+2}: (\mathbb F_2 \otimes_{\mathcal A} P_k)_{k(2^s-1) + 2^sd} \longrightarrow (\mathbb F_2 \otimes_{\mathcal A} P_k)_{k(2^{k-2}-1) + 2^{k-2}d}$$
is an isomorphism of $GL_k$-modules for every $s \geqslant k-2$.
\end{corl}

Corollary \ref{chq} shows that the number $t = k -2$ commonly serves for every degree $d$. H\uhorn ng stated in \cite{hu3} that $t = k - 2$ is the minimum number for this purpose and proved it for $k=5$. It is easy to see that for $d = 2^k-k+1$, we have $t(k,d) = k-2$. So, his statement is true for all $k \geqslant 2$.
 
An application of Theorem \ref{cthm} is the case $k = 5$ and  $d=0$. Then we have $t(5,0) = 3$. So, Theorem \ref{cthm} implies that
$$(\widetilde{Sq}^0_*)^{s-3}: (\mathbb F_2 \otimes_{\mathcal A} P_5)_{5(2^s-1)} \longrightarrow (\mathbb F_2 \otimes_{\mathcal A} P_5)_{35}$$
is an isomorphism of $GL_5$-modules for every $s \geqslant 3$. Hence, by computing the space $(\mathbb F_2 \otimes_{\mathcal A} P_5)_{5(2^s-1)}$ for $s=1,2,3$, we obtain the following.


\begin{thm}\label{dlbs} For $s$ an arbitrary positive integer, we have
$$\dim(\mathbb F_2{\otimes}_{\mathcal A}P_5)_{5(2^s-1)} = \begin{cases}46, &\text{ if $s=1$},\\ 432, &\text{ if $s=2$},\\ 1117  &\text{ if $s \geqslant 3$}. 
\end{cases}$$ 
\end{thm}

 This theorem has been proved in \cite{su5} for $s = 2$. In \cite{hu3}, H\uhorn ng also proved this theorem for $s=2,3$ by using a computer program of S. Shpectorov written in GAP. However, the detailed proof was unpublished at the time of the writing.

 Theorem \ref{dlbs} is proved by determining the admissible monomials of degree $5(2^s-1)$ in $P_5$. The computations are based on some results in \cite{ka}, \cite{si1} and \cite{su1} on the admissible monomials and the hit monomials (see Section \ref{s2}). 

To verify Conjecture \ref{gtsing} for $k = 5$ and the degree $d= 5(2^s-1)$, we need the following. 

\begin{thm}\label{pthm3} For any positive integer $s$, we have 
$$\dim(\mathbb F_2 \otimes_{\mathcal A} P_5)_{5(2^s-1)}^{GL_5} = \begin{cases}0, &\text{if } s=1,\\ 2, &\text{if } s=2,\\ 1, &\text{if } s \geqslant 3.\end{cases}$$ 
\end{thm}

This result confirms the one of H\uhorn ng in \cite{hu3}, which was also proved by using a computer calculation. 

From the results of  Chen \cite{che}, Lin \cite{wl} and Tangora \cite{ta}, we have
$$\text{Tor}_{5,5.2^s}^{\mathcal A}(\mathbb F_2,\mathbb F_2) =\begin{cases}0, &\text{if } s=1,\\ \langle (h_0^4h_4)^*,(h_1d_0)^*\rangle, &\text{if } s=2,\\ \langle (h_{s-1}d_{s-2})^*\rangle, &\text{if } s \geqslant 3, \end{cases}$$
and $h_{s-1}d_{s -2} \ne 0$, where $h_{s-1}$ denote the Adams element in $\text{\rm Ext}_{\mathcal A}^{1,2^{s-1}}(\mathbb F_2, \mathbb F_2)$ and $d_{s-2} \in \text{\rm Ext}_{\mathcal A}^{4,2^{s+2}+2^{s-1}}(\mathbb F_2, \mathbb F_2)$ for $s \geqslant 2$. 
In \cite{su5}, we have proved that 
$$\varphi_5: \text{Tor}_{5,20}^{\mathcal A}(\mathbb F_2,\mathbb F_2) \longrightarrow (\mathbb F_2{\otimes}_{\mathcal A}P_5)_{15}^{GL_5}$$
is an isomorphism. Combining the results of  H\`a \cite{ha} and Singer \cite{si1}, we have $\varphi_5((h_{s-1}d_{s -2})^*) \ne 0$.  Hence, Theorem \ref{pthm3} implies that 
$$\varphi_5: \text{Tor}_{5,5.2^s}^{\mathcal A}(\mathbb F_2,\mathbb F_2) \longrightarrow (\mathbb F_2{\otimes}_{\mathcal A}P_5)_{5(2^s-1)}^{GL_5}$$
 is also an isomorphism for $s \geqslant 3$. So, we get the following.

\begin{corl}\label{pthm1} Singer's conjecture is true for $k=5$ and the degree $5(2^{s} - 1)$,  with $s$ an arbitrary positive integer.
\end{corl}

For $d=2$, we have $t(5,2) = 2$ and $5(2^s-1) + 2^sd = 7.2^s-5$. Hence, by Theorem \ref{cthm},
$(\widetilde{Sq}^0_*)^{s-2}: (\mathbb F_2 \otimes_{\mathcal A} P_5)_{7.2^s-5} \longrightarrow (\mathbb F_2 \otimes_{\mathcal A} P_5)_{23}$
is an isomorphism of $GL_5$-modules for every $s \geqslant 2$. So, by an explicit computation of $(\mathbb F_2 \otimes_{\mathcal A} P_5)_{7.2^s-5}^{GL_5}$ for $s=1,2$, T\'in proved in \cite{tin} the following.

\begin{thm}[T\'in \cite{tin}]\label{pthm2} Let  $s$ be a positive integer. Then,
 $(\mathbb F_2 \otimes_{\mathcal A} P_5)_{7.2^s-5}^{GL_5} = 0$.
\end{thm}

The theorem has been proved by Singer \cite{si1} for $s=1$.
In \cite{hu3}, H\uhorn ng  also proved this theorem for by using a computer calculation. However, the detailed proof was also unpublished. 

From the results of Tangora \cite{ta}, Lin \cite{wl} and Chen \cite{che}, we can see that for any $s \geqslant 1$,
$\dim\text{Tor}_{5,7.2^s}^{\mathcal A}(\mathbb F_2,\mathbb F_2) = 1$. 
Hence, by Theorem \ref{pthm2}, the homomorphism
$\varphi_5$ is an epimorphism in the degree $7.2^s-5$. However, it is not a monomorphism. This result confirms the one of H\uhorn ng.

\begin{corl}[H\uhorn ng \cite{hu3}]  There are infinitely many degrees in which $\varphi_5$ is not a monomorphism.
\end{corl}

This paper is organized as follows. In Section \ref{s2}, we recall  some needed information on the admissible monomials in $P_k$ and Singer's criterion on the hit monomials.  Theorem \ref{cthm} is proved in Section \ref{s3}. In Section \ref{s4}, we explicitly determine a system of $\mathcal A$-generators for  $P_5$ in degree $5(2^s-1)$. Theorem \ref{pthm3} is proved in Section \ref{s5} by using the results in Section \ref{s4}. 

Theorems \ref{cthm} and \ref{pthm2} have already been announced in \cite{tsu}.

\section{Preliminaries}\label{s2}
\setcounter{equation}{0}

In this section, we recall some needed information from Kameko~\cite{ka} and Singer \cite{si2}, which will be used in the next sections.
\begin{nota} 
Let $\alpha_i(a)$ denote the $i$-th coefficient  in dyadic expansion of a non-negative integer $a$. That means
$a= \alpha_0(a)2^0+\alpha_1(a)2^1+\alpha_2(a)2^2+ \ldots ,$ for $ \alpha_i(a) =0$ or 1 with $i\geqslant 0$. 

For a set of integers $\mathbb J = \{j_1,j_2,\ldots , j_s\}$ with $1\leqslant j_u \leqslant k$, $1 \leqslant u \leqslant s$, we define the monomial $X_{\mathbb J} \in P_k$ by setting
$X_{\mathbb J} = \prod_{j\ne j_u, \forall u}x_j = x_1\ldots \hat x_{j_1}\ldots \hat x_{j_s}\ldots x_k .$

Let $x=x_1^{\nu_1(x)}x_2^{\nu_2(x)}\ldots x_k^{\nu_k(x)} \in P_k$. Set 
$$\mathbb J_t(x) = \{j : 1 \leqslant j \leqslant k, \ \alpha_t(\nu_j(x)) =0\},$$
for $t\geqslant 0$. Then, we have
$x = \prod_{t\geqslant 0}X_{\mathbb J_t(x)}^{2^t}.$ 
\end{nota}
\begin{defn}
For a monomial  $x$ in $P_k$,  define two sequences associated with $x$ by
\begin{align*} 
\omega(x)=(\omega_1(x),\omega_2(x),\ldots , \omega_i(x), \ldots),\ \
\sigma(x) = (\nu_1(x),\nu_2(x),\ldots ,\nu_k(x)),
\end{align*} 
where
$\omega_i(x) = \sum_{1\leqslant j \leqslant k} \alpha_{i-1}(\nu_j(x))= \deg X_{\mathbb J_{i-1}(x)},\ i \geqslant 1.$
The sequence $\omega(x)$ is called  the weight vector of $x$. 

The sequence $\omega=(\omega_1,\omega_2,\ldots , \omega_i, \ldots)$ of non-negative integers is called  the weight vector if $\omega_i = 0$ for $i \gg 0$.
\end{defn}

The sets of the weight vectors and the exponent vectors are given the left lexicographical order. 

For a  weight vector $\omega$,  we define $\deg \omega = \sum_{i > 0}2^{i-1}\omega_i$.  
Denote by   $P_k(\omega)$ the subspace of $P_k$ spanned by all monomials $y$ such that
$\deg y = \deg \omega$, $\omega(y) \leqslant \omega$, and by $P_k^-(\omega)$ the subspace of $P_k$ spanned by all monomials $y \in P_k(\omega)$  such that $\omega(y) < \omega$. 

\begin{defn}\label{dfn2} Let $\omega$ be a weight vector and $f, g$ two polynomials  of the same degree in $P_k$. 

i) $f \equiv g$ if and only if $f + g \in \mathcal A^+P_k$. If $f \equiv 0$ then $f$ is called {\it hit}.

ii) $f \equiv_{\omega} g$ if and only if $f + g \in \mathcal A^+P_k+P_k^-(\omega)$. 
\end{defn}

Obviously, the relations $\equiv$ and $\equiv_{\omega}$ are equivalence ones. Denote by $QP_k(\omega)$ the quotient of $P_k(\omega)$ by the equivalence relation $\equiv_\omega$. Then, we have 
$$QP_k(\omega)= P_k(\omega)/ ((\mathcal A^+P_k\cap P_k(\omega))+P_k^-(\omega)).$$  

For a  polynomial $f \in  P_k$, we denote by $[f]$ the class in $\mathbb F_2\otimes_{\mathcal A}P_k = P_k/\mathcal A^+P_k$ represented by $f$. If  $\omega$ is a weight vector, then denote by $[f]_\omega$ the class in the space $P_k/(\mathcal A^+P_k+P_k^-(\omega))$ represented by $f$ . 

Denote by $|S|$ the cardinal of a set $S$. If $S$ is a subset of a vector space, then we denote by $\langle S\rangle$ the subspace spanned by $S$.

It is easy to see that
$$ QP_k(\omega) \cong QP_k^\omega := \langle\{[x] \in QP_k : x \text{ \rm  is admissible and } \omega(x) = \omega\}\rangle. $$
So, we get
$$(\mathbb F_2{\otimes}_{\mathcal A}P_k)_n = \bigoplus_{\deg \omega = n}QP_k^\omega \cong \bigoplus_{\deg \omega = n}QP_k(\omega).$$
Hence, we can identify the vector space $QP_k(\omega)$ with $QP_k^\omega \subset QP_k$. 

We note that the weight vector of a monomial is invariant under the permutation of the generators $x_i$, hence $QP_k(\omega)$ is an $\Sigma_k$-module, where $\Sigma_k \subset GL_k$ is the symmetric group. 

Note that $V_k \cong \langle x_1, x_2,\ldots, x_k\rangle \subset P_k$. For $1 \leqslant i \leqslant k$, define the $\mathbb F_2$-linear map $\rho_i:V_k \to V_k$, which is determined by $\rho_i(x_i) = x_{i+1}, \rho_i(x_{i+1}) = x_i$, $\rho_i(x_j) = x_j$ for $j \ne i, i+1,\ 1 \leqslant i < k$, and $\rho_k(x_1) = x_1+x_2$,  $\rho_k(x_j) = x_j$ for $j > 1$.   The general linear group $GL_k\cong GL(V_k)$ is generated by $\rho_i,\ 1\leqslant i \leqslant k,$ and the symmetric group $\Sigma_k$ is generated by $\rho_i,\ 1 \leqslant i < k$. The $\mathbb F_2$-linear map $\rho_i$ induces a homomorphism of $\mathcal A$-algebras which is also denoted by $\rho_i: P_k \to P_k$. So, an element $[f]_\omega \in QP_k(\omega)$ is an $GL_k$-invariant if and only if $\rho_i(f) \equiv_\omega f$ for $1 \leqslant i\leqslant k$.  It is an  $\Sigma_k$-invariant if and only if $\rho_i(f) \equiv_\omega f$ for $1 \leqslant i < k$.

\begin{lem}\label{bdm} Let $\omega$ be a weight vector. Then, $QP_k(\omega)$ is an $GL_k$-module. 
\end{lem}

\begin{proof} Since $QP_k(\omega)$ is an $\Sigma_k$-module, we prove the lemma by proving the fact that if $x$ is a monomial in $P_k$, then $\rho_k(x) \in P_{k}(\omega(x))$.

 If $\nu_1(x) = 0$, then $x = \rho_k(x)$ and $\omega(\rho_k(x)) = \omega(x).$ Suppose $\nu_1(x) >0$ and $\nu_1(x) = 2^{t_1} + \ldots + 2^{t_b}$, where $0 \leqslant t_1 <  \ldots < t_b,\ b\geqslant 1$. 

Since $x = \prod_{t\geqslant 0}X_{\mathbb J_t(x)}^{2^t} \in P_k$ and $\rho_k$ is a homomorphism of algebras, we have
$$\rho_k(x) = \prod_{t\geqslant 0}(\rho_k(X_{\mathbb J_t(x)}))^{2^t}= \Bigg(\prod_{u=1}^b\big((x_1+x_2)X_{\mathbb J_{t_u}(x)\cup 1}\big)^{2^{t_u}} \Bigg)\Bigg(\prod_{t \ne t_1, t_2,\ldots, t_b}X_{\mathbb J_{t}(x)}^{2^t}\Bigg).$$

Then, $\rho_k(x)$ is a sum of monomials of the form
$$ \bar y = \Bigg(\prod_{j=1}^c\big(x_2X_{\mathbb J_{t_{u_j}}(x)\cup 1}\big)^{2^{t_u}} \Bigg)\Bigg(\prod_{t \ne t_{u_1}, \ldots, t_{u_c}}X_{\mathbb J_{t}(x)}^{2^t}\Bigg),$$
where $0 \leqslant c \leqslant b$. If $c= 0$, then $\bar y = x$ and $\omega(\bar y) = \omega(x).$ Suppose $c >0$.
 
If $2 \in \mathbb J_{t_{u_j}}(x)$ for all $j$, $1 \leqslant j \leqslant c$, then $\omega (\bar y) = \omega (x)$  and $\bar y \in P_{k}(\omega(x))$. Suppose there is an index $j$ such that $2 \not\in \mathbb J_{t_{u_j}}(x)$. Let $j_0$ be the smallest index such that $2\not\in \mathbb J_{t_{u_{j_0}}}(x)$. Then, we have 
$$ \omega_i(\bar y) =  \begin{cases}\omega _i(x), &\text{if }  i \leqslant t_{u_{j_0}},\\ \omega _i(x)-2, &\text{if } i  = t_{u_{j_0}}+1.\end{cases}$$
Hence $\omega (\bar y) < \omega(x)$ and $\bar y \in P_{k}(\omega(x))$. The lemma is proved.
\end{proof}

\begin{defn}\label{defn3} 
Let $x, y$ be monomials of the same degree in $P_k$. We say that $x <y$ if and only if one of the following holds:  

i) $\omega (x) < \omega(y)$;

ii) $\omega (x) = \omega(y)$ and $\sigma(x) < \sigma(y).$
\end{defn}

\begin{defn}
A monomial $x$ is said to be inadmissible if there exist monomials $y_1,y_2,\ldots, y_t$ such that $y_j<x$ for $j=1,2,\ldots , t$ and $x + \sum_{j=1}^ty_j \in \mathcal A^+P_k.$ 
A monomial $x$ is said to be admissible if it is not inadmissible.
\end{defn}

Obviously, the set of all the admissible monomials of degree $n$ in $P_k$ is a minimal set of $\mathcal{A}$-generators for $P_k$ in degree $n$. 

\begin{defn} 
 A monomial $x$ in $P_k$ is said to be strictly inadmissible if and only if there exist monomials $y_1,y_2,\ldots, y_t$ such that $y_j<x,$ for $j=1,2,\ldots , t$ and 
$$x = \sum_{j=1}^t y_j + \sum_{u=1}^{2^s-1}Sq^u(q_u)$$ 
with $s = \max\{i : \omega_i(x) > 0\}$ and suitable polynomials $q_u \in P_k$.
\end{defn}

It is easy to see that if $x$ is strictly inadmissible, then it is inadmissible.

\begin{thm}[See Kameko \cite{ka}]\label{dlcb1}  
 Let $x, y, w$ be monomials in $P_k$ such that $\omega_i(x) = 0$ for $i > r>0$, $\omega_s(w) \ne 0$ and   $\omega_i(w) = 0$ for $i > s>0$.

{\rm i)}  If  $w$ is inadmissible, then  $xw^{2^r}$ is also inadmissible.

{\rm ii)}  If $w$ is strictly inadmissible, then $wy^{2^{s}}$ is also strictly inadmissible.
\end{thm} 

Now, we recall a result of Singer \cite{si2} on the hit monomials in $P_k$. 

\begin{defn}\label{spi}  A monomial $z$ in $P_k$   is called a spike if $\nu_j(z)=2^{t_j}-1$ for $t_j$ a non-negative integer and $j=1,2, \ldots , k$. If $z$ is a spike with $t_1>t_2>\ldots >t_{r-1}\geqslant t_r>0$ and $t_j=0$ for $j>r,$ then it is called the minimal spike.
\end{defn}

In \cite{si2}, Singer showed that if $\mu(n) \leqslant k$, then there exists uniquely a minimal spike of degree $n$ in $P_k$. 


The following is a criterion for the hit monomials in $P_k$.

\begin{thm}[See Singer~\cite{si2}]\label{dlsig} Suppose $x \in P_k$ is a monomial of degree $n$, where $\mu(n) \leqslant k$. Let $z$ be the minimal spike of degree $n$. If $\omega(x) < \omega(z)$, then $x$ is hit.
\end{thm}

This result implies the one of Wood, which originally is a conjecture of Peterson~\cite{pe}.
 
\begin{thm}[See Wood~\cite{wo}]\label{dlmd1} 
If $\mu(n) > k$, then $(\mathbb F_2 \otimes_{\mathcal A}P_k)_n = 0$.
\end{thm} 

Now, we recall some notations and definitions in \cite{su1}, which will be used in the next sections. We set 
\begin{align*} 
P_k^0 &=\langle\{x=x_1^{a_1}x_2^{a_2}\ldots x_k^{a_k} \ : \ a_1a_2\ldots a_k=0\}\rangle,
\\ P_k^+ &= \langle\{x=x_1^{a_1}x_2^{a_2}\ldots x_k^{a_k} \ : \ a_1a_2\ldots a_k>0\}\rangle. 
\end{align*}

It is easy to see that $P_k^0$ and $P_k^+$ are the $\mathcal{A}$-submodules of $P_k$. Furthermore, we have the following.

\begin{prop}\label{2.7} We have a direct summand decomposition of the $\mathbb F_2$-vector spaces
$\mathbb F_2\otimes_{\mathcal A}P_k =QP_k^0 \oplus  QP_k^+.$
Here $QP_k^0 = \mathbb F_2\otimes_{\mathcal A}P_k^0$ and  $QP_k^+ = \mathbb F_2\otimes_{\mathcal A}P_k^+$.
\end{prop}

\begin{defn}For $1 \leqslant i \leqslant k$, define the homomorphism $f_i: P_{k-1} \to P_k$ of algebras by substituting
$$f_i(x_j) = \begin{cases} x_j, &\text{ if } 1 \leqslant j <i,\\
x_{j+1}, &\text{ if } i \leqslant j <k.
\end{cases}$$
\end{defn}
Obviously, we have the following.
 \begin{prop}\label{2.8} If $B$ is a minimal set of generators for $\mathcal A$-module $P_{k-1}$ in degree $n$, then $f(B):=\bigcup_{1 \leqslant i \leqslant k}f_i(B)$ is also a minimal set of generators for $\mathcal A$-module $P_{k}^0$ in degree $n$.
\end{prop}

\begin{defn}
For any $1 \leqslant i < j \leqslant k$, we define the homomorphism $p_{(i;j)}: P_k \to P_{k-1}$ of algebras by substituting
$$p_{(i;j)}(x_u) =\begin{cases} x_u, &\text{ if } 1 \leqslant u < i,\\
x_{j-1}, &\text{ if }  u = i,\\  
x_{u-1},&\text{ if } i< u \leqslant k.
\end{cases}$$
Then, $p_{(i;j)}$ is a homomorphism of $\mathcal A$-modules.  In particular, $p_{(i;j)}(f_i(y)) = y$ for any $y \in P_{k-1}$. 
\end{defn}

\begin{lem}[see \cite{sp}]\label{bdm1} Let $x$ be a monomial in $P_k$. Then, $p_{(i;j)}(x) \in P_{k-1}(\omega(x))$. 
\end{lem}

Lemma \ref{bdm1} implies that if $\omega$ is a weight vector and $x \in P_k(\omega)$, then $p_{(i;j)}(x) \in P_{k-1}(\omega)$.
Moreover, $p_{(i;j)}$ passes to a homomorphism from $QP_k(\omega)$ to $QP_{k-1}(\omega)$. 

For a subset $B \subset P_k$ and a weight vector $\omega$, we denote $[B] = \{[f] : f \in B\}$ and $[B]_\omega = \{[f]_\omega : f \in B\}$. From Theorem \ref{dlsig}, we see that if $\omega$ is the weight vector of a minimal spike in $P_k$, then $[B]_\omega = [B].$

From now on, we denote by $B_{k}(n)$ the set of all admissible monomials of degree $n$  in $P_k$, $B_{k}^0(n) = B_{k}(n)\cap P_k^0$, $B_{k}^+(n) = B_{k}(n)\cap P_k^+$. For a weight vector $\omega$ of degree $n$, we set $B_k(\omega) = B_{k}(n)\cap P_k(\omega)$, $B_k^+(\omega) = B_{k}^+(n)\cap P_k(\omega)$. 
Then, $[B_k(\omega)]_\omega$ and $[B_k^+(\omega)]_\omega$, are respectively the basses of the $\mathbb F_2$-vector spaces $QP_k(\omega)$ and $QP_k^+(\omega) := QP_k(\omega)\cap QP_k^+$.

 For any monomials $z, z_1, z_2, \ldots, z_m$ in $P_k(\omega)$ with $m\geqslant 1$, and for a subgroup $G\subset GL_k$, we denote 
 $ [G(z_1, z_2, \ldots, z_m)]_\omega$ the $G$-submodule of $QP_k(\omega)$ generated by the set $\{[z_1]_\omega, [z_2]_\omega , \ldots , [z_m]_\omega\}$. 

\section{Proof of Theorem \ref{cthm}}\label{s3}
\setcounter{equation}{0}

To make the paper self-contained, we give here a proof for the following lemma, which is an elementary property of the $\mu$-function.
\begin{lem}\label{bd1} Let $n$ be a positive integer. Then, $\mu(n) = s$ if and only if there exists uniquely a sequence of integers $v_1 > v_2 >\ldots > v_{s-1}\geqslant v_s>0$ such that
\begin{equation} \label{ct1.1}n =  2^{v_1} + 2^{v_2}+ \ldots + 2^{v_{s-1}}+ 2^{v_{s}} - s= \sum_{i=1}^s(2^{v_i}-1).
\end{equation}
\end{lem}
 \begin{proof}  Assume that $\mu(n) = s$. Set
$\beta(n) = \min\{u\in \mathbb N: \alpha(n+u) \leqslant u\}.$ We prove $\mu(n) = \beta(n).$

Suppose $\beta(n) = t$. Then $\alpha(n+t) = r \leqslant t$, and
$ n = 2^{c_1} + 2^{c_2} + \ldots +2^{c_{r}} -t,$
where $c_1> c_2 >  \ldots >  c_{r} \geqslant 0$. 

If  $c_r \leqslant t-r$ then 
$$\alpha(n+t-1) = \alpha(2^{c_1} + 2^{c_2} + \ldots +2^{c_{r-1}}+ 2^{c_r}-1) = r-1 + c_r \leqslant t-1.$$
Hence $\beta(n) \leqslant t-1$. This contradicts the fact that $\beta(n) = t$.
So, $c_r > t-r$. 

If $r=t$ then $c_t =c_r > t-r = 0$. Set $v_i= c_i, \ i= 1,2,\ldots , t$. We obtain
$$ n = 2^{v_1} + 2^{v_2} + \ldots + 2^{v_{t-1}} + 2^{v_t} - t = \sum_{i=1}^t(2^{v_i}-1),$$
where $ v_1>v_2 > \ldots> v_{t-2} > v_{t-1} > v_t>0$. Hence, $\mu(n) \leqslant t = \beta(n)$.

Suppose $r<t$. Obviously, 
$$2^{c_r} = 2^{c_r-1} + \ldots + 2^{c_r -t+r+1}+2^{c_r -t+r}+2^{c_r -t+r}.$$
Set 
\begin{align*} 
&v_i = c_i, \  i=1,2,\ldots, r-1,\\
&v_{r+\ell} =c_r-\ell- 1> 0, \ \ell = 0,1, \ldots ,t-r-2,\\
&v_{t-1}= v_t = c_r -t+r > 0.
\end{align*}
Then, we get
$$n = 2^{v_1} + 2^{v_2} + \ldots + 2^{v_{t-1}} + 2^{v_t} - t = \sum_{i=1}^t(2^{v_i}-1), $$
with $v_1>v_2 > \ldots >v_{t-2} > v_{t-1} = v_t>0$. Hence $\mu(n) \leqslant t = \beta(n)$.
 
Since $\mu(n) = s$, $n = \sum_{i=1}^s(2^{h_i}-1)$ with $h_i$ positive integers. Then, $\alpha(n+s) = \alpha(\sum_{i=1}^s2^{h_i})\leqslant s$. So, we get $\mu(n) = s \geqslant \beta(n)$. Hence, $t =\beta(n) =\mu(n)  = s$. Thus, $n$ is of the form (\ref{ct1.1}). 

Now, assume that $n$ is of the form (\ref{ct1.1}). Then, $\mu(n) \leqslant s$.
We prove $\mu(n) = s$ by induction on $s$.

If  $s=1$, then $\mu(n) =1$, since $\mu(n) > 0$. If $s = 2$, then $\alpha(n+1) = \alpha (2^{v_1} + 2^{v_2}-1)  = 1 + v_2 >1$. Hence,  $\mu(n) = \beta (n) \geqslant 2$. So, $\mu(n) =2$.

Suppose $s > 2$. By the inductive hypothesis, $\mu(n+1-2^{v_1}) = s-1$.

It is well-known that there exists uniquely an integer $d$ such that $2^d \leqslant n+1 < 2^{d+1}$. Since $v_1>v_2>\ldots > v_{s-1} \geqslant v_s> 0$, we have
$$2^{v_1} \leqslant n+1 < 2^{v_1}+ 2^{v_1-1}+\ldots + 2^{v_1-s+2}+2^{v_1-s+2} = 2^{v_1+1}.$$ 
So, we get $v_1 = d.$ Set $\mu(n) = t \leqslant s$.  There exists $u_1 > u_2 > \ldots > u_{t-1}\geqslant u_t > 0$ such that $n = 2^{u_1}+2^{u_2} +  \ldots + 2^{u_t}-t$. Then, $u_1 = d = v_1$ and $\alpha(n+1-2^d +t-1) \leqslant t-1$. Hence,  $t-1 \geqslant \beta(n+1-2^d) =  \mu(n+1-2^d) = s-1$. This implies $ t\geqslant s$ and $\mu(n) = t = s$. 

By induction on $i$, we get $u_i = v_i$ for $1\leqslant i \leqslant s$. The lemma is proved.
\end{proof}

From this lemma we easily obtain the following.
\begin{corl}[See Kameko \cite{ka}]\label{mmd1} Let $n,k $ be  positive integers. Then

\medskip
{\rm i)} $\mu(n) > k $ if and only if $\alpha(n+k) > k$.

{\rm ii)} If $n > \mu(n)$, then $n-\mu(n)$ is even and $\mu\big(\frac{n-\mu(n)}2\big) \leqslant \mu(n)$.

 {\rm iii)} $\mu\big(2n + \mu(n)\big) = \mu(n).$
\end{corl}

Suppose that $d$ is a non-negative integer such that $\mu(2d + k) =s < k$. By Lemma \ref{bd1}, there exists a sequence of integers $v_1 > v_2 >\ldots > v_{s-1}\geqslant v_s>0$ such that
$2d+k =  \sum_{i=1}^s(2^{v_i}-1)$. Set $z = x_1^{2^{v_1}-1}x_2^{2^{v_2}-1}\ldots x_s^{2^{v_s}-1} \in (P_k)_{2d+k}$. Since $z$ is a spike and $s<k$, we have $[z] \ne 0$ and $\widetilde{Sq}^0_*([z]) = 0$. So, one gets the following.
\begin{corl}\label{mmd2} 
Let $d$ be an arbitrary non-negative integer. If $\mu(2d+k)<k$, then 
$$(\widetilde{Sq}^0_*)_{(k,d)}:= \widetilde{Sq}^0_* : (\mathbb F_2 \otimes_{\mathcal A} P_k)_{2d+k}\longrightarrow (\mathbb F_2 \otimes_{\mathcal A} P_k)_d$$
is not a monomorphism. 
\end{corl} 

Now, we are ready to prove Theorem \ref{cthm}.
\begin{proof}[Proof of Theorem \ref{cthm}] Set $q = \alpha(d+k)$, $r = \zeta(d+k)$ and $m = k(2^t-1) + 2^td$. From the proof of Lemma \ref{bd1} and Corollary \ref{mmd1}, we see that if $q > k$, then 
$$\mu(k(2^s-1) + 2^sd) \geqslant \mu(d)>k$$ 
for any $s \geqslant 0 = t(k,d)$. By Theorem \ref{dlmd1}, 
$$(\mathbb F_2 \otimes_{\mathcal A}P_k)_{k(2^s-1) + 2^sd} =0,\ (\mathbb F_2 \otimes_{\mathcal A}P_k)_{d} =0.$$ 
So, the theorem holds. 

Assume that $q \leqslant k$. Using Theorem \ref{dlk}, Corollaries \ref{mmd1} and  \ref{mmd2}, we see that the homomorphism
$$(\widetilde{Sq}^0_*)^{s-t}: (\mathbb F_2 \otimes_{\mathcal A} P_k)_{k(2^s-1) + 2^sd} \longrightarrow (\mathbb F_2 \otimes_{\mathcal A} P_k)_{k(2^t-1) + 2^td}$$
is an isomorphism of $GL_k$-modules for every $s \geqslant t$ if and only if $\mu(2m+k) = k$.  

Since $\alpha(d+k) = q$ and $\zeta(d+k) = r$, there exists a sequence of integers $c_1 > c_2 > \ldots > c_{q-1} > c_q = r\geqslant 0$ such that
$$ d+k = 2^{c_1} + 2^{c_2} + \ldots + 2^{c_q}.$$

If $q=k$, then 
\begin{align*} 2m +k &= k(2^{t+1}-1) + 2^{t+1}d = k(2^{t+1}-1)  + 2^{t+1}(2^{c_1} + 2^{c_2} + \ldots + 2^{c_k}-k)\notag\\
&= 2^{c_1+t+1} + 2^{c_2+t+1} +\ldots + 2^{c_{k}+t+1}  - k.
\end{align*}
By Lemma \ref{bd1}, $\mu(2m+k) = k$ for any $t \geqslant 0 = t(k,d)$. Hence, the theorem holds.

Suppose that $q < k$. Then, we have
\begin{align}\label{ctmthm}2m +k &= 2^{c_1+t+1} + 2^{c_2+t+1} +\ldots + 2^{c_{q-1}+t+1} + 2^{r+t+1} - k\notag\\
&= 2^{c_1+t+1} + 2^{c_2+t+1} +\ldots + 2^{c_{q-1}+t+1}\notag\\
&\quad  + 2^{r+t}+ 2^{r+t-1} + \ldots + 2^{r+t- (k-q-1)} +  2^{r+t- (k-q-1)} - k.
\end{align}

If $q+ r \geqslant  k$, then $r+t- (k-q-1) = q+r - k + 1 +t > 0$ for any $t \geqslant 0 = t(k,d)$. By Lemma \ref{bd1}, $\mu(2m+k) = k$. So, the theorem is true. 

If $q+r < k$, then  from Lemma \ref{bd1} and the relation (\ref{ctmthm}), we see that $\mu(2m+k) = k$ if and only if $t \geqslant k-q-r = t(k,d)$. The theorem is completely proved.
\end{proof}

\section{$\mathcal A$-generators for $P_5$ in degree $5(2^s-1)$}\label{s4}
\setcounter{equation}{0}

In this section, we prove Theorem \ref{dlbs} by determining the admissible monomials of $P_5$ in degree $5(2^s-1)$.
From Theorem \ref{cthm}, we see that
$$(\widetilde{Sq}^0_*)^{s-3}: (\mathbb F_2 \otimes_{\mathcal A} P_5)_{5(2^s-1)} \longrightarrow (\mathbb F_2 \otimes_{\mathcal A} P_k)_{35}$$
is an isomorphism of $GL_5$-modules for every $s \geqslant 3$. So, we need only to determine $\mathcal A$-generators for $P_5$ in degree $5(2^s-1)$ for $s=1,2,3$.

\subsection{The Cases $s=1, 2$}\

\medskip
By using a result in \cite{su1}, we can easily obtain the following.
\begin{props}\label{dl10} There exist exactly $46$ admissible monomials of degree $5$ in $P_5$. Consequently $\dim (\mathbb F_2 \otimes_{\mathcal A} P_5)_5 = 46$.
\end{props} 

For $s=2$, we have $5(2^s - 1) = 15$. The space $(\mathbb F_2 \otimes_{\mathcal A} P_5)_{15}$ has been computed in \cite{su5}.

\begin{props}[See \cite{su5}]\label{dl11} There exist exactly $432$ admissible monomials of degree $15$ in $P_5$. Consequently $\dim (\mathbb F_2 \otimes_{\mathcal A} P_5)_{15} = 432.$ 
\end{props} 
The admissible monomials of degrees $5$ and $15$ are explicitly determined as in Subsections \ref{ss61} and \ref{ss62}.

\subsection{The admissible monomials of degree $16$ in $P_5$}\label{subs51}\

\medskip
To determine the admissible monomials of $P_5$ in degree $5(2^s-1)$ for $s=3$, we need to determine the admissible monomials of degree 16 in $P_5$. 

\begin{lems}\label{bdday61}
If $x$  is an admissible monomial of degree $16$ in $P_5$, then $\omega(x)$ is one of the following sequences:
$$(2,1,1,1),\ (2,1,3),\ (2,3,2),\ (4,2,2), \ (4,4,1). $$
\end{lems}
\begin{proof}  Observe that $z = x_1^{15}x_2$ is the minimal spike of degree $16$ in $P_5$ and  $\omega(z) =  (2,1,1,1)$. Since $[x] \ne 0$, by Theorem \ref{dlsig}, either $\omega_1(x) =2$ or $\omega_1(x) =2$. If $\omega_1(x) =2$, then $x = x_ix_jy^2$ with $y$ a monomial of degree $7$ in $P_5$. Since $x$ is admissible, by Theorem \ref{dlcb1}, $y$ is admissible. A routine computation shows that either $\omega(y) =(1,1,1)$ or $\omega(y) =(1,3)$ or $\omega(y) =(3,2)$. If $\omega_1(x) =4$, then $x = X_jy_1^2$ with $y_1$ an admissible monomial of degree $6$ in $P_5$. It is easy to see that either $\omega(y_1) =(2,2)$ or $\omega(y_1) =(4,1)$. The lemma is proved.
\end{proof}

We have $(\mathbb F_2 \otimes_{\mathcal A} P_5)_{16}  = (QP_5^0)_{16}\bigoplus (QP_5^+)_{16}$. By Lemma \ref{bdday61}, 
\begin{align*} (QP_5^+)_{16}  &\cong QP_5^+(2,1,1,1) \bigoplus QP_5^+(2,1,3)\\ 
&\quad\bigoplus QP_5^+(2,3,2)\bigoplus QP_5^+(4,2,2)\bigoplus QP_5^+(4,4,1).
\end{align*}

From a result in \cite{su1}, we easily obtain $\dim(QP_5^0)_{16} = 255$.
\begin{props}\label{md611} $(QP_5^+)_{16}$ is the $\mathbb F_2$-vector space of dimension $188$ with a basis consisting of all the classes represented by the monomials $a_t = a_{16,t},\ 1 \leqslant t \leqslant 188$, which are determined as in Subsection \ref{ss64}.
\end{props}

To prove this proposition, we need the following lemma which is easily proved by a direct computation.

\begin{lems}\label{inad16} The following monomials are strictly inadmissible:

\medskip
{\rm i)} $x_j^2x_\ell x_t$,   $j < \ell < t$; $x_j^2x_\ell x_t^2x_u^3$, $j < \ell$; $x_j^2x_\ell x_tx_u^2x_v^2$, $j < \ell$;
$x_j^3x_\ell^{12} x_t$, $x_j^3x_\ell^4 x_t^9$, $x_j^3x_\ell^5 x_t^8$, $x_j^3x_\ell^4 x_tx_u^8$,  $j < \ell < t$;  $x_j^2x_\ell^3 x_t^3$.

{\rm ii)} $x_j^3x_\ell^4 x_t^4x_u^5$,\ $x_j^3x_\ell^4 x_tx_u^4x_v^4,\ j < \ell < t.$

{\rm iii)} $x_jx_\ell^6 x_t^3x_u^6$, $x_jx_\ell^6 x_tx_u^2x_v^6$, $j < \ell < t$; $x_jx_\ell^2 x_t^2x_u^5x_v^6, u \geqslant 4$; $x_jx_\ell^2 x_t^3x_u^4x_v^6, t \geqslant 3$; $x_jx_\ell^2 x_t^6x_u^7$, $x_jx_\ell^2 x_t^2x_u^4x_v^7$; $x_1^3x_2^4 x_3x_4^2x_5^6$, $x_1^3x_2^4 x_3x_4^6x_5^2$.

{\rm iv)} $x_j^2x_\ell x_t^3x_u^3x_v^3,\  j < \ell;\ x_j^2x_\ell x_tx_ux_v^3,\ j < \ell< t < u;\ x_1^3x_2^4x_3^3x_4^3x_5^3.$

Here $(j,\ell, t,u,v)$ is a permutation of $(1,2,3,4,5)$.
\end{lems}

\begin{proof}[Proof of Proposition \ref{md611}] From Lemma \ref{bdday61}, we have
$$B_5^+(16) = B_5^+(2,1,1,1)\cup B_5^+(2,1,3)\cup B_5^+(2,3,2)\cup B_5^+(4,2,2)\cup B_5^+(4,4,1).$$
We prove $|B_5^+(2,1,1,1)| = 4$, $|B_5^+(2,1,3)| = 5$, $|B_5^+(2,3,2)| = 20$, $|B_5^+(4,2,2)| = 110$ and $|B_5^+(4,4,1)| = 49$. For simplicity, we prove $|B_5^+(2,3,2)|  = 20$ by showing that $B_5^+(2,3,2) = \{a_t = a_{16,t}: 10 \leqslant t \leqslant 29\}$. The others are proved by the similar computations.

Let $x$ be an admissible monomial in $P_5^+$ such that $\omega(x) = (2,3,2):= \omega$. Then $x = x_jx_\ell y^2$ with $1 \leqslant j < \ell \leqslant 5$ and $y$ a monomial of degree 7 in $P_5$. Since $x$ is admissible, by Theorem \ref{dlcb1}, $y \in B_5(3,2)$. 

Let $z \in B_5(3,2)$ such that $x_jx_\ell z^2 \in P_5^+$. By a direct computation, we see that if $x_jx_\ell z^2\ne a_t$, for all $t,\ 10 \leqslant t \leqslant 29$, then there is a monomial $w$ which is given in Lemma \ref{inad16} such that $x_jx_\ell z^2= wz_1^{2^{u}}$ with suitable monomial $z_1 \in P_5$, and $u = \max\{j \in \mathbb Z : \omega_j(w) >0\}$. By Theorem \ref{dlcb1}, $x_jx_\ell z^2$ is inadmissible. Since $x = x_jx_\ell y^2$ with $y \in B_5(3,2)$ and $x$ admissible, one obtain $x= a_t$ for some $t,\ 10 \leqslant t \leqslant 29$. Hence, $QP_5^+(\omega)$ is spanned by the set $\{[a_t]: 10 \leqslant t \leqslant 29\}$. 

We now prove the set $\{[a_t]: 10 \leqslant t \leqslant 29\}$ is linearly independent in $QP_5(\omega)$. Suppose there is a linear relation
$\mathcal S = \sum_{t = 10}^{29}\gamma_ta_t \equiv_\omega 0$, where $\gamma_t \in \mathbb F_2$.

From a result in \cite{su1}, $\dim QP_4^+(2,3,2) = 4$, with the basis $\{[w_u]: 1 \leqslant u \leqslant 4\}$, where
$$w_1 = x_1x_2^3x_3^6x_4^6,\ w_2 = x_1^3x_2x_3^6x_4^6,\ w_3 = x_1^3x_2^5x_3^2x_4^6,\ w_4 = x_1^3x_2^5x_3^6x_4^2.$$

By a direct computation using Lemma \ref{bdm1}, we get
\begin{align*}
p_{(1;2)}(\mathcal S) &\equiv_\omega \gamma_{13}w_{2} +  \gamma_{14}w_{3} +   \gamma_{15}w_{4} \equiv_\omega 0,\\
p_{(1;3)}(\mathcal S) &\equiv_\omega  \gamma_{10}w_{1} +  \gamma_{18}w_{3} +  \gamma_{19}w_{4}  \equiv_\omega 0.
\end{align*}
The above relations imply $\gamma_t = 0$ for $t = 10, 13, 14, 15, 18, 19$. Then,
\begin{align*}
p_{(1;4)}(\mathcal S) &\equiv_\omega \gamma_{11}w_{1} + \gamma_{16}w_{2} + \gamma_{21}w_{4}  \equiv_\omega 0,\\ 
p_{(1;5)}(\mathcal S) &\equiv_\omega \gamma_{12}w_{1} + \gamma_{17}w_{2} + \gamma_{20}w_{3} \equiv_\omega 0,\\  
p_{(2;3)}(\mathcal S) &\equiv_\omega \gamma_{24}w_{3} + \gamma_{25}w_{4} \equiv_\omega 0,\\ 
p_{(2;4)}(\mathcal S) &\equiv_\omega \gamma_{11}w_{1} + \gamma_{22}w_{2} + \gamma_{27}w_{4} \equiv_\omega 0,\\  
p_{(2;5)}(\mathcal S) &\equiv_\omega \gamma_{12}w_{1} + \gamma_{23}w_{2} + \gamma_{26}w_{3} \equiv_\omega 0.
\end{align*}
From the last equalities, we obtain $\gamma_t = 0$ for $t \ne 28, 29$. Then,
\begin{align*}
p_{(3;4)}(\mathcal S) \equiv_\omega \gamma_{29}w_{4}  \equiv_\omega 0,\ p_{(3;5)}(\mathcal S) \equiv_\omega \gamma_{28}w_{3}  \equiv_\omega 0.
\end{align*}
So, $\gamma_t =0$ for all $t, \ 10 \leqslant t \leqslant 29$, completing the proof.
\end{proof} 
By combining the above results, one gets the following.
\begin{corls} There exist exactly $443$ admissible monomials of degree $16$ in $P_5$. Consequently $\dim (\mathbb F_2 \otimes_{\mathcal A} P_5)_{16} = 443.$ 
\end{corls}

 \subsection{The Case $s=3$}\

\medskip
For $s=3$, we have $5(2^s-1) = 35$. Since Kameko's squaring operation 
$$(\widetilde{Sq}^0_*)_{(5,15)}: (\mathbb F_2 \otimes_{\mathcal A} P_5)_{35} \longrightarrow (\mathbb F_2 \otimes_{\mathcal A} P_k)_{15}$$
is an epimorphism, we have 
$(\mathbb F_2 \otimes_{\mathcal A} P_5)_{35} \cong \text{Ker}(\widetilde{Sq}^0_*)_{(5,15)} \bigoplus (\mathbb F_2 \otimes_{\mathcal A} P_5)_{15}$. Hence, we need only to compute $\text{Ker}(\widetilde{Sq}^0_*)_{(5,15)}$.

\begin{lems}\label{bdds62}
If $x$  is an admissible monomial of degree $35$ in $P_5$ and $[x] \in \text{\rm Ker}(\widetilde{Sq}^0_*)_{(5,15)}$, then
$\omega(x)$ is one of the following sequences:
\begin{align*}&\omega_{(1)} = (3,2,1,1,1),\ \omega_{(2)} = (3,2,1,3),\ \omega_{(3)} = (3,2,3,2),\\ 
&\omega_{(4)} = (3,4,2,2), \ \omega_{(5)} = (3,4,4,1).
\end{align*}
\end{lems}
\begin{proof}  Note that $z = x_1^{31}x_2^3x_3$ is the minimal spike of degree $35$ in $P_5$ and  $\omega(z) =  (3,2,1,1,1)$. Since $[x] \ne 0$, by Theorem \ref{dlsig}, either $\omega_1(x) =3$ or $\omega_1(x) =5$. If $\omega_1(x) =5$, then $x = X_\emptyset y^2$ with $y$ a monomial of degree $15$ in $P_5$. Since $x$ is admissible, by Theorem \ref{dlcb1}, $y$ is admissible. Hence, $(\widetilde{Sq}^0_*)_{(5,15)}([x]) = [y] \ne 0.$ This contradicts the fact that $[x] \in \text{Ker}(\widetilde{Sq}^0_*)_{(5,15)}$, so $\omega_1(x) =3$. Then, we have $x = x_ix_jx_\ell y_1^2$ with $y_1$ an admissible monomial of degree $16$ in $P_5$. Now, the lemma follows from Lemma \ref{bdday61}.
\end{proof}

From Lemma \ref{bdds62} and a result in \cite{su1}, we obtain
\begin{align*}
&\text{\rm Ker}(\widetilde{Sq}^0_*)_{(5,15)} =  \bigoplus_{j=1}^5 QP_5(\omega_{(j)}),\\
&QP_5(\omega_{(1)}) = (QP_5^0)_{35}\bigoplus QP_5^+(\omega_{(1)}),\\ 
&QP_5(\omega_{(j)}) = QP_5^+(\omega_{(j)}),\ j = 2,3,4,5,\\
&\dim(QP_5^0)_{35} = 460.
\end{align*}

\begin{props}\label{mdd61} There exist exactly $160$ admissible monomials in $P_5^+$ such that  their weight vectors are $\omega_{(1)}$. Consequently $\dim QP_5^+(\omega_{(1)})= 160.$ 
\end{props} 
We denote the monomials in $B_5(\omega_{(1)})$ by $a_t = a_{35,t},\ 1 \leqslant t \leqslant 160$, as given in Subsection \ref{ss65}. We need some lemmas for the proof of this proposition.

By a simple computation, one gets the following.

\begin{lems}\label{bdd21} If $(j,\ell, t,u,v)$ is a permutation of $(1,2,3,4,5)$, then the following monomials are strictly inadmissible:

\medskip
{\rm i)} $x_j^2x_\ell x_tx_u^3,\  j < \ell < t.$

{\rm ii)} $x_j^2x_\ell x_tx_ux_v^2,\ j < \ell< t < u,$
 $x_1x_2^2x_3^2x_4x_5.$
\end{lems}

The following is a corollary of a result in \cite{su1}.

\begin{lems}\label{inadb61} If $x$ is one of the following monomials then $f_i(x)$, $ 1 \leqslant i \leqslant 5$, are strictly inadmissible:
\begin{align*}
&x_1^{3}x_2^{28}x_3x_4^{3},\  x_1^{3}x_2^{28}x_3^{3}x_4,\  x_1^{3}x_2^{7}x_3^{24}x_4,\  x_1^{7}x_2^{3}x_3^{24}x_4,\  x_1^{3}x_2^{5}x_3^{25}x_4^{2},\  x_1^{3}x_2^{4}x_3^{25}x_4^{3},\\  
&x_1^{3}x_2^{5}x_3^{24}x_4^{3},\  x_1^{3}x_2^{4}x_3^{9}x_4^{19},\  x_1^{3}x_2^{4}x_3^{11}x_4^{17},\  x_1^{3}x_2^{5}x_3^{8}x_4^{19},\  x_1^{3}x_2^{5}x_3^{9}x_4^{18},\  x_1^{3}x_2^{5}x_3^{10}x_4^{17},\\  
&x_1^{3}x_2^{5}x_3^{11}x_4^{16},\  x_1^{3}x_2^{7}x_3^{8}x_4^{17},\  x_1^{7}x_2^{3}x_3^{8}x_4^{17},\  x_1^{3}x_2^{7}x_3^{9}x_4^{16},\  x_1^{7}x_2^{3}x_3^{9}x_4^{16}.
\end{align*}
\end{lems}
\begin{lems}\label{inab61} The following monomials are strictly inadmissible:
\begin{align*}
&x_1x_2^{6}x_3^{3}x_4^{8}x_5^{17}\ \  x_1x_2^{6}x_3^{3}x_4^{9}x_5^{16}\ \  x_1x_2^{6}x_3^{3}x_4^{24}x_5\ \  x_1^{3}x_2^{4}x_3x_4^{8}x_5^{19}\ \  x_1^{3}x_2^{4}x_3x_4^{9}x_5^{18}\\ 
&x_1^{3}x_2^{4}x_3x_4^{10}x_5^{17}\ \  x_1^{3}x_2^{4}x_3x_4^{11}x_5^{16}\ \  x_1^{3}x_2^{4}x_3x_4^{24}x_5^{3}\ \  x_1^{3}x_2^{4}x_3x_4^{25}x_5^{2}\ \  x_1^{3}x_2^{4}x_3^{3}x_4^{8}x_5^{17}\\ 
&x_1^{3}x_2^{4}x_3^{3}x_4^{9}x_5^{16}\ \  x_1^{3}x_2^{4}x_3^{3}x_4^{24}x_5\ \  x_1^{3}x_2^{4}x_3^{8}x_4x_5^{19}\ \  x_1^{3}x_2^{4}x_3^{8}x_4^{3}x_5^{17}\ \  x_1^{3}x_2^{4}x_3^{8}x_4^{17}x_5^{3}\\ 
&x_1^{3}x_2^{4}x_3^{8}x_4^{19}x_5\ \  x_1^{3}x_2^{4}x_3^{9}x_4x_5^{18}\ \  x_1^{3}x_2^{4}x_3^{9}x_4^{2}x_5^{17}\ \  x_1^{3}x_2^{4}x_3^{9}x_4^{3}x_5^{16}\ \  x_1^{3}x_2^{4}x_3^{9}x_4^{16}x_5^{3}\\ 
&x_1^{3}x_2^{4}x_3^{9}x_4^{17}x_5^{2}\ \  x_1^{3}x_2^{4}x_3^{9}x_4^{18}x_5\ \  x_1^{3}x_2^{4}x_3^{11}x_4x_5^{16}\ \  x_1^{3}x_2^{4}x_3^{11}x_4^{16}x_5\ \  x_1^{3}x_2^{4}x_3^{24}x_4x_5^{3}\\ 
&x_1^{3}x_2^{4}x_3^{24}x_4^{3}x_5\ \  x_1^{3}x_2^{4}x_3^{25}x_4x_5^{2}\ \  x_1^{3}x_2^{4}x_3^{25}x_4^{2}x_5\ \  x_1^{3}x_2^{5}x_3x_4^{8}x_5^{18}\ \  x_1^{3}x_2^{5}x_3x_4^{10}x_5^{16}\\ 
&x_1^{3}x_2^{5}x_3x_4^{24}x_5^{2}\ \  x_1^{3}x_2^{5}x_3^{8}x_4x_5^{18}\ \  x_1^{3}x_2^{5}x_3^{8}x_4^{2}x_5^{17}\ \  x_1^{3}x_2^{5}x_3^{8}x_4^{3}x_5^{16}\ \  x_1^{3}x_2^{5}x_3^{8}x_4^{16}x_5^{3}\\ 
&x_1^{3}x_2^{5}x_3^{8}x_4^{17}x_5^{2}\ \  x_1^{3}x_2^{5}x_3^{8}x_4^{18}x_5\ \  x_1^{3}x_2^{5}x_3^{9}x_4^{2}x_5^{16}\ \  x_1^{3}x_2^{5}x_3^{9}x_4^{16}x_5^{2}\ \  x_1^{3}x_2^{5}x_3^{10}x_4x_5^{16}\\ 
&x_1^{3}x_2^{5}x_3^{10}x_4^{16}x_5\ \  x_1^{3}x_2^{5}x_3^{24}x_4x_5^{2}\ \  x_1^{3}x_2^{5}x_3^{24}x_4^{2}x_5\ \  x_1^{3}x_2^{7}x_3^{8}x_4x_5^{16}\ \  x_1^{3}x_2^{7}x_3^{8}x_4^{16}x_5\\ 
&x_1^{3}x_2^{28}x_3x_4x_5^{2}\ \  x_1^{3}x_2^{28}x_3x_4^{2}x_5\ \  x_1^{7}x_2^{3}x_3^{8}x_4x_5^{16}\ \  x_1^{7}x_2^{3}x_3^{8}x_4^{16}x_5.
\end{align*}
\end{lems}
\begin{proof} We prove the lemma for the monomials $x = x_1x_2^{6}x_3^{3}x_4^{8}x_5^{17}$, $y= x_1^{3}x_2^{4}x_3^{3}x_4^{9}x_5^{16}$. The others can be proved by the similar computations. By a direct computation, we have
\begin{align*} x &= x_1x_2^{3}x_3^{5}x_4^{2}x_5^{24} + x_1x_2^{3}x_3^{5}x_4^{8}x_5^{18} + x_1x_2^{3}x_3^{6}x_4x_5^{24} + x_1x_2^{3}x_3^{6}x_4^{8}x_5^{17} + x_1x_2^{3}x_3^{8}x_4x_5^{22}\\
&\quad + x_1x_2^{3}x_3^{8}x_4^{2}x_5^{21} + x_1x_2^{4}x_3^{2}x_4x_5^{27} + x_1x_2^{4}x_3^{3}x_4x_5^{26} + x_1x_2^{4}x_3^{3}x_4^{2}x_5^{25} + x_1x_2^{4}x_3^{10}x_4x_5^{19}\\ 
&\quad + x_1x_2^{6}x_3^{2}x_4x_5^{25} + x_1x_2^{6}x_3^{3}x_4x_5^{24} + Sq^1(x_1^2x_2^{5}x_3^{5}x_4x_5^{21}) + Sq^2(x_1x_2^{6}x_3^{3}x_4^{2}x_5^{21}\\
&\quad + x_1x_2^{5}x_3^{5}x_4x_5^{21} + x_1x_2^{3}x_3^{6}x_4x_5^{22} + x_1x_2^{3}x_3^{6}x_4^{2}x_5^{21} + x_1x_2^{3}x_3^{5}x_4^{2}x_5^{22} + x_1x_2^{6}x_3^{3}x_4x_5^{22}\\
&\quad + x_1x_2^{6}x_3^{6}x_4x_5^{19} + x_1x_2^{6}x_3^{2}x_4x_5^{23}) + Sq^4(x_1x_2^{10}x_3^{3}x_4^{4}x_5^{13} + x_1x_2^{4}x_3^{3}x_4^{2}x_5^{21}\\
&\quad + x_1x_2^{3}x_3^{10}x_4^{4}x_5^{13} + x_1x_2^{3}x_3^{9}x_4^{4}x_5^{14} + x_1x_2^{4}x_3^{3}x_4x_5^{22} + x_1x_2^{10}x_3^{4}x_4x_5^{15}\\
&\quad + x_1x_2^{4}x_3^{2}x_4x_5^{23} + x_1x_2^{4}x_3^{6}x_4x_5^{19}) + Sq^8(x_1x_2^{6}x_3^{3}x_4^{4}x_5^{13} + x_1x_2^{3}x_3^{6}x_4^{4}x_5^{13}\\
&\quad + x_1x_2^{3}x_3^{5}x_4^{4}x_5^{14} + x_1x_2^{6}x_3^{4}x_4x_5^{15}) \ \ \text{ mod}\big(P_5^-(\omega_{(1)})\big)
\end{align*}
Hence, $x$ is strictly inadmissible. By a similar computation, we obtain
 \begin{align*} y &= x_1^{2}x_2x_3^{2}x_4^{5}x_5^{25} + x_1^{2}x_2x_3^{3}x_4^{9}x_5^{20} + x_1^{2}x_2x_3^{3}x_4^{12}x_5^{17} + x_1^{2}x_2x_3^{4}x_4^{3}x_5^{25}\\
&\quad + x_1^{2}x_2x_3^{5}x_4^{9}x_5^{18} + x_1^{2}x_2x_3^{5}x_4^{10}x_5^{17} + x_1^{2}x_2x_3^{10}x_4^{5}x_5^{17} + x_1^{2}x_2x_3^{12}x_4^{3}x_5^{17}\\
&\quad + x_1^{3}x_2x_3^{2}x_4^{4}x_5^{25} + x_1^{3}x_2x_3^{2}x_4^{5}x_5^{24} + x_1^{3}x_2x_3^{4}x_4^{3}x_5^{24} + x_1^{3}x_2x_3^{4}x_4^{10}x_5^{17}\\
&\quad + x_1^{3}x_2x_3^{5}x_4^{10}x_5^{16} + x_1^{3}x_2x_3^{8}x_4^{3}x_5^{20} + x_1^{3}x_2x_3^{8}x_4^{5}x_5^{18} + x_1^{3}x_2x_3^{10}x_4^{4}x_5^{17}\\
&\quad + x_1^{3}x_2^{2}x_3^{4}x_4^{9}x_5^{17} + x_1^{3}x_2^{2}x_3^{5}x_4^{8}x_5^{17} + x_1^{3}x_2^{2}x_3^{8}x_4^{5}x_5^{17} + x_1^{3}x_2^{4}x_3^{2}x_4^{9}x_5^{17}\\
&\quad + x_1^{3}x_2^{4}x_3^{3}x_4^{8}x_5^{17} + Sq^1(q_1) + Sq^2(q_2) + Sq^4(q_3) + Sq^8(q_4) \ \ \text{ mod}\big(P_5^-(\omega_{(1)})\big),
\end{align*}
where
\begin{align*}
q_1 &= x_1^{3}x_2x_3^{2}x_4^{3}x_5^{25} + x_1^{3}x_2x_3^{3}x_4^{9}x_5^{18} + x_1^{3}x_2x_3^{3}x_4^{10}x_5^{17}\\
&\quad  + x_1^{3}x_2x_3^{5}x_4^{5}x_5^{20} + x_1^{3}x_2x_3^{10}x_4^{3}x_5^{17} + x_1^{3}x_2^{4}x_3^{5}x_4^{5}x_5^{17},\\
q_2 &= x_1^{2}x_2x_3^{2}x_4^{3}x_5^{25} + x_1^{2}x_2x_3^{3}x_4^{9}x_5^{18} + x_1^{2}x_2x_3^{3}x_4^{10}x_5^{17} + x_1^{2}x_2x_3^{10}x_4^{3}x_5^{17} + x_1^{5}x_2x_3^{2}x_4^{3}x_5^{22}\\
&\quad + x_1^{5}x_2x_3^{6}x_4^{3}x_5^{18} + x_1^{5}x_2^{2}x_3^{2}x_4^{3}x_5^{21} + x_1^{5}x_2^{2}x_3^{3}x_4^{5}x_5^{18} + x_1^{5}x_2^{2}x_3^{3}x_4^{6}x_5^{17} + x_1^{5}x_2^{2}x_3^{6}x_4^{3}x_5^{17},\\
q_3 &= x_1^{3}x_2x_3^{2}x_4^{3}x_5^{22} + x_1^{3}x_2x_3^{6}x_4^{3}x_5^{18} + x_1^{3}x_2x_3^{8}x_4^{5}x_5^{14} + x_1^{3}x_2x_3^{9}x_4^{6}x_5^{12}\\
&\quad + x_1^{3}x_2^{2}x_3^{2}x_4^{3}x_5^{21} + x_1^{3}x_2^{2}x_3^{3}x_4^{5}x_5^{18} + x_1^{3}x_2^{2}x_3^{3}x_4^{6}x_5^{17} + x_1^{3}x_2^{2}x_3^{6}x_4^{3}x_5^{17}\\
&\quad + x_1^{3}x_2^{2}x_3^{8}x_4^{5}x_5^{13} + x_1^{3}x_2^{8}x_3^{2}x_4^{5}x_5^{13} + x_1^{3}x_2^{8}x_3^{3}x_4^{5}x_5^{12} + x_1^{3}x_2^{8}x_3^{4}x_4^{3}x_5^{13}\\
q_4 &= x_1^{3}x_2x_3^{4}x_4^{5}x_5^{14} + x_1^{3}x_2x_3^{5}x_4^{6}x_5^{12} + x_1^{3}x_2^{2}x_3^{4}x_4^{5}x_5^{13}\\
&\quad + x_1^{3}x_2^{4}x_3^{2}x_4^{5}x_5^{13} + x_1^{3}x_2^{4}x_3^{3}x_4^{5}x_5^{12} + x_1^{3}x_2^{4}x_3^{4}x_4^{3}x_5^{13}.
\end{align*}
Hence, $x$ is strictly inadmissible. 
\end{proof}
\begin{proof}[Proof of Proposition \ref{mdd61}] Let $x$ be an admissible monomial such that $\omega(x) = \omega_{(1)}$. Then, $x = x_jx_\ell x_ty^2$ with $1 \leqslant j < \ell < t \leqslant 5$ and $y \in B_5(2,1,1,1)$. 

Let $z \in B_5(2,1,1,1)$ such that $x_jx_\ell x_tz^2 \in P_5^+$.
By a direct computation using the results in Subsection \ref{subs51}, we see that if $x_jx_\ell x_tz^2 \ne a_{t}, \forall t, \ 1 \leqslant t \leqslant 160$, then there is a monomial $w$ which is given in one of Lemmas \ref{bdd21}, \ref{inadb61} and \ref{inab61} such that $x_jx_\ell x_tz^2= wz_1^{2^{u}}$ with suitable monomial $z_1 \in P_5$, and $u = \max\{j \in \mathbb Z : \omega_j(w) >0\}$. By Theorem \ref{dlcb1}, $x_jx_\ell x_tz^2$ is inadmissible. Since $x = x_jx_\ell x_ty^2$ and $x$ is admissible, one gets $x= a_{t}$ for some $t,\ 1 \leqslant t \leqslant 160$. This implies $B_5^+(\omega_{(1)}) \subset \{a_{t} : \ 1 \leqslant t \leqslant 160\}$. 

We now prove the set $\{[a_{t}] : \ 1 \leqslant t \leqslant 160\}$ is linearly independent in $(\mathbb F_2 \otimes_{\mathcal A} P_5)_{35}$. Suppose there is a linear relation
$$\mathcal S = \sum_{t = 1}^{160}\gamma_ta_{t} \equiv 0,$$ 
where $\gamma_t \in \mathbb F_2$. For $1 \leqslant i < j \leqslant 5$, we explicitly compute $p_{(i;j)}(\mathcal S)$ in terms of the admissible monomials in $P_4$ (mod$(\mathcal A^+P_4)$). By a direct computation from the relations $p_{(i;j)}(\mathcal S) \equiv 0$ with $1 \leqslant i < j \leqslant 5$, we obtain $\gamma_t = 0$ for $1 \leqslant t \leqslant 160$.
The proposition follows. 
\end{proof}
\begin{props}\label{mdd62} $QP_5(\omega_{(2)})= 0.$ 
\end{props} 

We need the following lemma.
\begin{lems}\label{inad62} All permutations of the following monomials are strictly inadmissible:
$$x_1^{3}x_2^{4}x_3^{8}x_4^{9}x_5^{11} ,\  x_1^{3}x_2^{4}x_3^{9}x_4^{9}x_5^{10} ,\  x_1^{3}x_2^{5}x_3^{8}x_4^{8}x_5^{11} ,\  x_1^{3}x_2^{5}x_3^{8}x_4^{9}x_5^{10} ,\  x_1^{3}x_2^{7}x_3^{8}x_4^{8}x_5^{9}.$$
\end{lems}
\begin{proof} We prove the lemma for the monomial $x = x_1^{3}x_2^{4}x_3^{8}x_4^{9}x_5^{11}$. The others can be proved by the similar computations. By a direct computation, we have
\begin{align*}
x &= Sq^1(x_1^{3}x_2x_3^{2}x_4^{9}x_5^{19}) + Sq^2(x_1^{5}x_2^{2}x_3^{2}x_4^{5}x_5^{19} + x_1^{5}x_2x_3^{2}x_4^{6}x_5^{19})\\
&\quad + Sq^4(x_1^{3}x_2^{8}x_3^{4}x_4^{5}x_5^{11} + x_1^{3}x_2^{2}x_3^{2}x_4^{5}x_5^{19} + x_1^{3}x_2^{8}x_3^{2}x_4^{5}x_5^{13} + x_1^{3}x_2^{2}x_3^{8}x_4^{5}x_5^{13}\\
&\quad + x_1^{3}x_2x_3^{2}x_4^{6}x_5^{19} + x_1^{3}x_2x_3^{8}x_4^{6}x_5^{13}) + Sq^8(x_1^{3}x_2^{4}x_3^{4}x_4^{5}x_5^{11} + x_1^{3}x_2^{4}x_3^{2}x_4^{5}x_5^{13}\\
&\quad + x_1^{3}x_2^{2}x_3^{4}x_4^{5}x_5^{13} + x_1^{3}x_2x_3^{4}x_4^{6}x_5^{13}) \ \ \text{ mod}\big(P_5^-(\omega_{(2)})\big)
\end{align*}
This equality shows that all permutations of $x$ are strictly inadmissible.
\end{proof}
\begin{proof}[Proof of Proposition \ref{mdd62}] Let $x$ be an admissible monomial such that $\omega(x) = \omega_{(2)}$. Then $x = x_jx_\ell x_ty^2$ with $y \in B_5(2,1,3)$. 

Let $z \in B_5(2,1,3)$ such that $x_jx_\ell x_tz^2 \in P_5^+$.
By a direct computation using the results in Subsection \ref{subs51}, we see that if $x_jx_\ell x_tz^2$ is not a permutation of one of monomials as given in Lemma \ref{inad62}, then there is a monomial $w$ which is given in Lemma \ref{bdd21} such that $x_jx_\ell x_tz^2= wz_1^{2^{u}}$ with suitable monomial $z_1 \in P_5$, and $u = \max\{j \in \mathbb Z : \omega_j(w) >0\}$. By Theorem \ref{dlcb1}, $x_jx_\ell x_tz^2$ is inadmissible. Since $x = x_jx_\ell x_ty^2$ and $x$ is admissible, $x$ is a permutation of one of monomials as given in Lemma \ref{inad62}. Now the proposition follows from Lemma \ref{inad62}. 
\end{proof}

\begin{props}\label{mdd63} $QP_5(\omega_{(3)}) = 0.$ 
\end{props} 

The following lemma is needed for the proof of the proposition.

\begin{lems}\label{inad63} The following monomials are strictly inadmissible:

{\rm i)} $x_j^3x_\ell^4 x_t^5x_u^7,\  x_j^3x_\ell^5 x_t^5x_u^6.$

{\rm ii)} $x_j^3x_\ell^4 x_tx_u^4x_v^7,\ x_j^3x_\ell^4 x_tx_u^5x_v^6, \  x_j^3x_\ell^4 x_t^2x_u^5x_v^5, \ j < \ell < t$; $x_j^3x_\ell^4 x_t^3x_u^4x_v^5$, $j< \ell,\ t > 3.$

Here $(j,\ell, t,u,v)$ is a permutation of $(1,2,3,4,5)$.

{\rm iii)} All permutations of the monomials: 
\begin{align*}
&x_1x_2^{2}x_3^{7}x_4^{12}x_5^{13}\quad x_1x_2^{3}x_3^{6}x_4^{12}x_5^{13}\quad x_1x_2^{3}x_3^{7}x_4^{12}x_5^{12}\quad x_1x_2^{6}x_3^{11}x_4^{4}x_5^{13} \\  
& x_1x_2^{7}x_3^{10}x_4^{4}x_5^{13}\quad x_1x_2^{7}x_3^{11}x_4^{4}x_5^{12}\quad x_1x_2^{6}x_3^{11}x_4^{5}x_5^{12}\quad x_1x_2^{7}x_3^{10}x_4^{5}x_5^{12} \\  
& x_1^{3}x_2^{5}x_3^{2}x_4^{12}x_5^{13}\quad x_1^{3}x_2^{3}x_3^{4}x_4^{12}x_5^{13}\quad x_1^{3}x_2^{3}x_3^{5}x_4^{12}x_5^{12}\quad x_1^{3}x_2^{4}x_3^{11}x_4^{4}x_5^{13} \\  
& x_1^{3}x_2^{5}x_3^{10}x_4^{4}x_5^{13}\quad x_1^{3}x_2^{4}x_3^{11}x_4^{5}x_5^{12}\quad x_1^{3}x_2^{7}x_3^{9}x_4^{4}x_5^{12}\quad x_1^{3}x_2^{5}x_3^{10}x_4^{5}x_5^{12}.
\end{align*}
\end{lems} 
\begin{proof} We prove the lemma for $x = x_1x_2^{2}x_3^{7}x_4^{12}x_5^{13}$ and $y = x_1x_2^{3}x_3^{6}x_4^{12}x_5^{13}$. The others are proved by the similar computations. A direct computation shows
\begin{align*} x &= Sq^1(x_1^{2}x_2x_3^{7}x_4^{5}x_5^{19} + x_1^{2}x_2x_3^{9}x_4^{3}x_5^{19} + x_1^{2}x_2x_3^{7}x_4^{3}x_5^{21})\\
&\quad + Sq^2(x_1x_2^{4}x_3^{7}x_4^{10}x_5^{11} + x_1x_2x_3^{7}x_4^{5}x_5^{19} + x_1x_2x_3^{9}x_4^{3}x_5^{19} + x_1x_2x_3^{7}x_4^{3}x_5^{21})\\
&\quad + Sq^4(x_1x_2^{2}x_3^{11}x_4^{6}x_5^{11} + x_1x_2^{2}x_3^{6}x_4^{3}x_5^{19} + x_1x_2^{2}x_3^{11}x_4^{3}x_5^{14})\\
&\quad + Sq^8(x_1x_2^{2}x_3^{7}x_4^{6}x_5^{11} + x_1x_2^{2}x_3^{7}x_4^{3}x_5^{14}) \ \ \text{ mod}\big(P_5^-(\omega_{(3)})\big)
\end{align*}
Hence, all permutations of $x$ are strictly inadmissible. We have
\begin{align*} 
y &= Sq^1\big(x_1x_2^{3}x_3^{2}x_4^{9}x_5^{19} + x_1x_2^{3}x_3^{3}x_4^{12}x_5^{15} + x_1x_2^{3}x_3^{5}x_4^{10}x_5^{15}\\
&\quad + x_1x_2^{3}x_3^{9}x_4^{2}x_5^{19} + x_1x_2^{3}x_3^{10}x_4^{5}x_5^{15} + x_1x_2^{3}x_3^{12}x_4^{3}x_5^{15} + x_1x_2^{6}x_3^{3}x_4^{3}x_5^{21}\\
&\quad + x_1x_2^{8}x_3^{3}x_4^{3}x_5^{19} + x_1x_2^{12}x_3^{3}x_4^{3}x_5^{15} + x_1^{2}x_2^{5}x_3^{3}x_4^{9}x_5^{15} + x_1^{2}x_2^{5}x_3^{9}x_4^{3}x_5^{15}\\
&\quad + x_1^{4}x_2^{3}x_3^{3}x_4^{9}x_5^{15} + x_1^{4}x_2^{3}x_3^{9}x_4^{3}x_5^{15} + x_1^{4}x_2^{5}x_3^{3}x_4^{3}x_5^{19} + x_1^{4}x_2^{9}x_3^{3}x_4^{3}x_5^{15}\big)\\
&\quad +  Sq^2\big(x_1x_2^{2}x_3^{3}x_4^{12}x_5^{15} + x_1x_2^{2}x_3^{5}x_4^{10}x_5^{15} + x_1x_2^{4}x_3^{10}x_4^{3}x_5^{15}\\
&\quad + x_1x_2^{5}x_3^{2}x_4^{10}x_5^{15} + x_1x_2^{5}x_3^{6}x_4^{6}x_5^{15} + x_1x_2^{5}x_3^{6}x_4^{10}x_5^{11} + x_1x_2^{5}x_3^{10}x_4^{2}x_5^{15}\\
&\quad + x_1x_2^{10}x_3^{2}x_4^{5}x_5^{15} + x_1x_2^{10}x_3^{5}x_4^{2}x_5^{15} + x_1^{2}x_2^{2}x_3^{5}x_4^{9}x_5^{15} + x_1^{2}x_2^{2}x_3^{9}x_4^{5}x_5^{15}\\
&\quad + x_1^{2}x_2^{3}x_3^{3}x_4^{10}x_5^{15} + x_1^{2}x_2^{3}x_3^{10}x_4^{3}x_5^{15} + x_1^{2}x_2^{5}x_3^{2}x_4^{9}x_5^{15} + x_1^{2}x_2^{5}x_3^{9}x_4^{2}x_5^{15}\\
&\quad + x_1^{2}x_2^{6}x_3^{3}x_4^{3}x_5^{19} + x_1^{2}x_2^{9}x_3^{2}x_4^{5}x_5^{15} + x_1^{2}x_2^{9}x_3^{5}x_4^{2}x_5^{15} + x_1^{2}x_2^{10}x_3^{3}x_4^{3}x_5^{15}\\
&\quad + x_1^{4}x_2^{3}x_3^{2}x_4^{9}x_5^{15} + x_1^{4}x_2^{3}x_3^{9}x_4^{2}x_5^{15} + x_1^{4}x_2^{6}x_3^{3}x_4^{5}x_5^{15} + x_1^{4}x_2^{6}x_3^{5}x_4^{3}x_5^{15}\big)\\
&\quad + Sq^4\big(x_1x_2^{3}x_3^{2}x_4^{10}x_5^{15} + x_1x_2^{3}x_3^{6}x_4^{6}x_5^{15} + x_1x_2^{3}x_3^{10}x_4^{2}x_5^{15}\\
&\quad + x_1x_2^{3}x_3^{10}x_4^{6}x_5^{11} + x_1x_2^{6}x_3^{2}x_4^{3}x_5^{19} + x_1x_2^{6}x_3^{3}x_4^{2}x_5^{19} + x_1x_2^{10}x_3^{2}x_4^{3}x_5^{15}\\
&\quad + x_1x_2^{10}x_3^{3}x_4^{2}x_5^{15} + x_1x_2^{10}x_3^{3}x_4^{3}x_5^{14} + x_1^{2}x_2^{2}x_3^{3}x_4^{5}x_5^{19} + x_1^{2}x_2^{2}x_3^{3}x_4^{9}x_5^{15}\\ 
&\quad + x_1^{2}x_2^{2}x_3^{5}x_4^{3}x_5^{19} + x_1^{2}x_2^{2}x_3^{9}x_4^{3}x_5^{15} + x_1^{2}x_2^{3}x_3^{2}x_4^{5}x_5^{19} + x_1^{2}x_2^{3}x_3^{3}x_4^{6}x_5^{17}\\
&\quad + x_1^{2}x_2^{3}x_3^{5}x_4^{2}x_5^{19} + x_1^{2}x_2^{3}x_3^{6}x_4^{3}x_5^{17} + x_1^{2}x_2^{5}x_3^{2}x_4^{3}x_5^{19} + x_1^{2}x_2^{5}x_3^{3}x_4^{2}x_5^{19}\\
&\quad + x_1^{2}x_2^{6}x_3^{3}x_4^{3}x_5^{17} + x_1^{2}x_2^{6}x_3^{3}x_4^{5}x_5^{15} + x_1^{2}x_2^{6}x_3^{5}x_4^{3}x_5^{15} + x_1^{2}x_2^{9}x_3^{2}x_4^{3}x_5^{15}\\
&\quad + x_1^{2}x_2^{9}x_3^{3}x_4^{2}x_5^{15}\big) + Sq^8\big(x_1x_2^{3}x_3^{6}x_4^{6}x_5^{11} + x_1x_2^{6}x_3^{3}x_4^{3}x_5^{14}\big) \ \text{ mod}\big(P_5^-(\omega_{(3)})\big).
\end{align*}
This equality implies that all permutations of $y$ are strictly inadmissible.
\end{proof}

\begin{proof}[Proof of Proposition \ref{mdd63}] Let $x$ be an admissible monomial such that $\omega(x) = \omega_{(3)}$. Then $x = x_jx_\ell x_ty^2$ with $1 \leqslant j < \ell < t \leqslant 5$ and $y \in B_5(2,3,2)$. 
By a direct computation using Proposition \ref{md611}, we see that  there is a monomial $w$ which is given in one of Lemmas \ref{bdd21}, \ref{inad63} such that $x_jx_\ell x_ty^2= wz_1^{2^{u}}$ with suitable monomial $z_1 \in P_5$, and $u = \max\{j \in \mathbb Z : \omega_j(w) >0\}$. By Theorem \ref{dlcb1}, $x = x_jx_\ell x_ty^2$ is inadmissible. Hence, $QP_5(\omega_{(3)}) = 0.$ 
\end{proof}

Consider the monomials $a_t = a_{35,t},\ 161 \leqslant t \leqslant 210$ as given in Subsection \ref{ss65}.

\begin{props}\label{mdd64} The $\mathbb F_2$-vector space $QP_5(\omega_{(4)})$ is an $GL_5$-module generated by the class $[a_{203}]_{\omega_{(4)}}$ and
$B_5(\omega_{(4)}) = \{a_{t} : 161 \leqslant t \leqslant 210\}.$
Consequently, $\dim QP_5(\omega_{(4)}) = 50$.
\end{props} 

We need the following lemmas.

\begin{lems}\label{inad64} The following monomials are strictly inadmissible:

{\rm i)} $x_j^2x_\ell x_t^2x_u^3x_v^3,\  j < \ell;\ x_j^2x_\ell^3 x_t^3x_u^3.$

Here $(j,\ell, t,u,v)$ is a permutation of $(1,2,3,4,5)$.

{\rm ii)} All permutations of the monomials: 
\begin{align*}
&x_1x_2^{2}x_3^{2}x_4^{15}x_5^{15}\ \  x_1x_2^{2}x_3^{3}x_4^{14}x_5^{15}\ \  x_1x_2^{2}x_3^{7}x_4^{10}x_5^{15}\ \  x_1x_2^{3}x_3^{3}x_4^{14}x_5^{14}\\ 
&x_1x_2^{3}x_3^{6}x_4^{10}x_5^{15}\ \  x_1x_2^{7}x_3^{7}x_4^{10}x_5^{10}\ \  x_1^{2}x_2^{3}x_3^{3}x_4^{13}x_5^{14}\ \  x_1^{2}x_2^{3}x_3^{5}x_4^{10}x_5^{15}\\ 
&x_1^{2}x_2^{2}x_3^{3}x_4^{13}x_5^{15}\ \  x_1^{2}x_2^{2}x_3^{7}x_4^{9}x_5^{15}\ \  x_1^{2}x_2^{7}x_3^{7}x_4^{9}x_5^{10}.
\end{align*} 
\end{lems} 
\begin{proof} We prove the lemma for $x = x_1x_2^{2}x_3^{7}x_4^{10}x_5^{15}$, $y = x_1x_2^{3}x_3^{3}x_4^{14}x_5^{14}$ and $z = x_1x_2^{7}x_3^{7}x_4^{10}x_5^{10}$. We have
\begin{align*}
x &= Sq^1(x_1^{2}x_2x_3^{7}x_4^{9}x_5^{15}) + Sq^2(x_1x_2x_3^{7}x_4^{9}x_5^{15} + x_1x_2x_3^{3}x_4^{9}x_5^{19} + x_1x_2x_3^{7}x_4^{3}x_5^{21}\\
&\quad + x_1x_2x_3^{3}x_4^{3}x_5^{25}) + Sq^4(x_1x_2x_3^{5}x_4^{9}x_5^{15} + x_1x_2x_3^{11}x_4^{5}x_5^{13} + x_1x_2x_3^{5}x_4^{3}x_5^{21}\\
&\quad + x_1x_2x_3^{3}x_4^{5}x_5^{21}) + Sq^8(x_1x_2x_3^{7}x_4^{5}x_5^{13}) \text{ mod}\big(P_5^-(\omega_{(4)})\big).
\end{align*} 
So, all permutations of $x$ are strictly inadmissible.
\begin{align*}
y &= Sq^1\big(x_1x_2^{3}x_3^{10}x_4^{13}x_5^{7} + x_1^{2}x_2^{3}x_3^{9}x_4^{13}x_5^{7} + x_1x_2^{3}x_3^{3}x_4^{14}x_5^{13} + x_1x_2^{3}x_3^{5}x_4^{14}x_5^{11}\\
&\quad + x_1x_2^{3}x_3^{3}x_4^{13}x_5^{14}\big) + Sq^2\big(x_1x_2^{5}x_3^{3}x_4^{7}x_5^{17} + x_1x_2^{5}x_3^{3}x_4^{11}x_5^{13}\\
&\quad + x_1x_2^{5}x_3^{5}x_4^{11}x_5^{11} + x_1x_2^{5}x_3^{9}x_4^{11}x_5^{7}\big) +  Sq^4\big(x_1x_2^{3}x_3^{3}x_4^{7}x_5^{17}\\
&\quad + x_1x_2^{3}x_3^{3}x_4^{11}x_5^{13} + x_1x_2^{3}x_3^{5}x_4^{11}x_5^{11} + x_1x_2^{3}x_3^{9}x_4^{11}x_5^{7} + x_1x_2^{3}x_3^{5}x_4^{13}x_5^{9}\\
&\quad + x_1^{2}x_2^{3}x_3^{6}x_4^{13}x_5^{7}\big) + Sq^8\big(x_1x_2^{3}x_3^{5}x_4^{9}x_5^{9}\big),\ \text{ mod}\big(P_5^-(\omega_{(4)})\big).
\end{align*} 
Hence, all permutations of $y$ are strictly inadmissible.
\begin{align*}
z &= Sq^1\big(x_1x_2^{7}x_3^{7}x_4^{9}x_5^{10} + x_1x_2^{7}x_3^{11}x_4^{3}x_5^{12} + x_1x_2^{7}x_3^{13}x_4^{3}x_5^{10} + x_1x_2^{11}x_3^{7}x_4^{3}x_5^{12}\\
&\quad + x_1x_2^{13}x_3^{7}x_4^{3}x_5^{10} + x_1^{2}x_2^{7}x_3^{13}x_4^{3}x_5^{9} + x_1^{2}x_2^{13}x_3^{7}x_4^{3}x_5^{9} + x_1^{4}x_2^{7}x_3^{11}x_4^{3}x_5^{9}\\
&\quad + x_1^{4}x_2^{11}x_3^{7}x_4^{3}x_5^{9}\big) + Sq^2\big(x_1x_2^{7}x_3^{11}x_4^{5}x_5^{9} + x_1x_2^{11}x_3^{7}x_4^{5}x_5^{9} + x_1^{2}x_2^{7}x_3^{11}x_4^{3}x_5^{10}\\
&\quad + x_1^{2}x_2^{11}x_3^{7}x_4^{3}x_5^{10}\big) + Sq^4\big(x_1x_2^{7}x_3^{11}x_4^{3}x_5^{9} + x_1x_2^{9}x_3^{13}x_4^{3}x_5^{5} + x_1x_2^{11}x_3^{7}x_4^{3}x_5^{9}\\
&\quad + x_1x_2^{13}x_3^{9}x_4^{3}x_5^{5} + x_1^{2}x_2^{7}x_3^{7}x_4^{5}x_5^{10}\big) +Sq^8\big( x_1x_2^{9}x_3^{9}x_4^{3}x_5^{5}\big),\ \text{ mod}\big(P_5^-(\omega_{(4)})\big).
\end{align*} 
Hence, all permutations of $z$ are strictly inadmissible.
\end{proof}
\begin{lems}\label{ina64} The following monomials are strictly inadmissible:
\begin{align*}
 &x_1x_2^{3}x_3^{6}x_4^{14}x_5^{11}\ \  x_1x_2^{3}x_3^{14}x_4^{6}x_5^{11}\ \  x_1x_2^{3}x_3^{14}x_4^{7}x_5^{10}\ \  x_1x_2^{7}x_3^{10}x_4^{3}x_5^{14}\ \  x_1^{3}x_2x_3^{6}x_4^{14}x_5^{11}\\ 
&x_1^{3}x_2x_3^{14}x_4^{6}x_5^{11}\ \  x_1^{3}x_2x_3^{14}x_4^{7}x_5^{10}\ \  x_1^{3}x_2^{5}x_3^{2}x_4^{14}x_5^{11}\ \  x_1^{3}x_2^{5}x_3^{6}x_4^{10}x_5^{11}\ \  x_1^{3}x_2^{5}x_3^{6}x_4^{11}x_5^{10}\\ 
&x_1^{3}x_2^{5}x_3^{7}x_4^{10}x_5^{10}\ \  x_1^{3}x_2^{5}x_3^{10}x_4^{3}x_5^{14}\ \  x_1^{3}x_2^{5}x_3^{10}x_4^{6}x_5^{11}\ \  x_1^{3}x_2^{5}x_3^{10}x_4^{7}x_5^{10}\ \  x_1^{3}x_2^{5}x_3^{10}x_4^{14}x_5^{3}\\ 
&x_1^{3}x_2^{5}x_3^{14}x_4^{2}x_5^{11}\ \  x_1^{3}x_2^{5}x_3^{14}x_4^{3}x_5^{10}\ \  x_1^{3}x_2^{5}x_3^{14}x_4^{10}x_5^{3}\ \  x_1^{3}x_2^{5}x_3^{14}x_4^{11}x_5^{2}\ \  x_1^{3}x_2^{13}x_3^{2}x_4^{6}x_5^{11}\\ 
&x_1^{3}x_2^{13}x_3^{2}x_4^{7}x_5^{10}\ \  x_1^{3}x_2^{13}x_3^{3}x_4^{6}x_5^{10}\ \  x_1^{3}x_2^{13}x_3^{6}x_4^{2}x_5^{11}\ \  x_1^{3}x_2^{13}x_3^{6}x_4^{3}x_5^{10}\ \  x_1^{3}x_2^{13}x_3^{6}x_4^{10}x_5^{3}\\ 
&x_1^{3}x_2^{13}x_3^{6}x_4^{11}x_5^{2}\ \  x_1^{3}x_2^{13}x_3^{7}x_4^{2}x_5^{10}\ \  x_1^{3}x_2^{13}x_3^{7}x_4^{10}x_5^{2}\ \  x_1^{7}x_2x_3^{10}x_4^{3}x_5^{14}\ \  x_1^{7}x_2^{9}x_3^{2}x_4^{3}x_5^{14}\\ 
&x_1^{7}x_2^{9}x_3^{3}x_4^{2}x_5^{14}\ \  x_1^{7}x_2^{9}x_3^{3}x_4^{6}x_5^{10}\ \  x_1^{7}x_2^{9}x_3^{3}x_4^{14}x_5^{2}.
\end{align*}
\end{lems}
\begin{proof} We prove the lemma for $x = x_1x_2^{3}x_3^{14}x_4^{7}x_5^{10}$ and $y = x_1x_2^{7}x_3^{10}x_4^{3}x_5^{14}$. The others can be obtained by the similar computations. We have
\begin{align*}
x &= x_1x_2^{3}x_3^{7}x_4^{14}x_5^{10} + Sq^1\big(x_1^{2}x_2^{3}x_3^{13}x_4^{7}x_5^{9} + x_1^{2}x_2^{3}x_3^{7}x_4^{13}x_5^{9}\big) + Sq^2\big(x_1x_2^{5}x_3^{11}x_4^{7}x_5^{9}\\
&\quad + x_1x_2^{5}x_3^{7}x_4^{11}x_5^{9}\big) + Sq^4\big(x_1x_2^{3}x_3^{11}x_4^{7}x_5^{9} + x_1x_2^{3}x_3^{7}x_4^{11}x_5^{9} + x_1x_2^{3}x_3^{13}x_4^{9}x_5^{5}\\
&\quad + x_1x_2^{3}x_3^{9}x_4^{13}x_5^{5}\big) + Sq^8\big(x_1x_2^{3}x_3^{9}x_4^{9}x_5^{5}\big),\ \text{ mod}\big(P_5^-(\omega_{(4)})\big).
\end{align*}
Hence, $x$ is strictly inadmissible.
\begin{align*}
y &= x_1x_2^{7}x_3^{3}x_4^{10}x_5^{14} + Sq^1\big(x_1x_2^{7}x_3^{9}x_4^{10}x_5^{7} + x_1x_2^{7}x_3^{10}x_4^{3}x_5^{13} + x_1x_2^{7}x_3^{10}x_4^{9}x_5^{7}\\
&\quad + x_1x_2^{7}x_3^{12}x_4^{3}x_5^{11} + x_1x_2^{11}x_3^{3}x_4^{12}x_5^{7} + x_1x_2^{13}x_3^{3}x_4^{10}x_5^{7} + x_1^{2}x_2^{7}x_3^{3}x_4^{9}x_5^{13}\\
&\quad + x_1^{2}x_2^{7}x_3^{5}x_4^{9}x_5^{11} + x_1^{2}x_2^{13}x_3^{3}x_4^{9}x_5^{7} + x_1^{4}x_2^{7}x_3^{9}x_4^{3}x_5^{11} + x_1^{4}x_2^{11}x_3^{3}x_4^{9}x_5^{7}\big)\\
&\quad + Sq^2\big(x_1x_2^{7}x_3^{5}x_4^{9}x_5^{11} + x_1x_2^{11}x_3^{5}x_4^{9}x_5^{7} + x_1^{2}x_2^{7}x_3^{10}x_4^{3}x_5^{11} + x_1^{2}x_2^{11}x_3^{3}x_4^{10}x_5^{7}\big)\\
&\quad + Sq^4\big(x_1x_2^{7}x_3^{3}x_4^{9}x_5^{11} + x_1x_2^{9}x_3^{3}x_4^{5}x_5^{13} + x_1x_2^{11}x_3^{3}x_4^{9}x_5^{7}\\
&\quad + x_1x_2^{13}x_3^{3}x_4^{5}x_5^{9} + x_1^{2}x_2^{7}x_3^{5}x_4^{10}x_5^{7} + x_1^{2}x_2^{11}x_3^{6}x_4^{5}x_5^{7}\big)\\
&\quad + Sq^8\big(x_1x_2^{9}x_3^{3}x_4^{5}x_5^{9} + x_1^{2}x_2^{7}x_3^{6}x_4^{5}x_5^{7}\big),\ \text{ mod}\big(P_5^-(\omega_{(4)})\big).
\end{align*}
This equality implies that $y$ are strictly inadmissible.
\end{proof}
\begin{lems}\label{noz1} The class  $[a_{161}]_{\omega_{(4)}}$ is non-zero in the vector space $QP_5(\omega_{(4)})$.
\end{lems}
\begin{proof} Suppose the contrary that $[a_{161}]_{\omega_{(4)}} = 0$. Then
\begin{align}\label{ctno1} a_{161} = \sum_{u=0}^4Sq^{2^u}(B_u)  \ \text{mod}(P_5^-(\omega_{(4)})),
\end{align}
where $B_u$ are suitable polynomials in $P_5$. We set
\begin{align*}
f_{4,5} &= x_1^{2}x_2^{4}x_3^{7}x_4^{14}x_5^{14} + x_1^{4}x_2^{2}x_3^{7}x_4^{14}x_5^{14},\ y_1 = x_1x_2x_3^7x_4^7x_5^{11},\\
f_{3,5} &= x_1^{2}x_2^{4}x_3^{14}x_4^{7}x_5^{14} + x_1^{4}x_2^{2}x_3^{14}x_4^{7}x_5^{14}, \ y_2 = x_1x_2x_3^7x_4^{11}x_5^{7},\\
f_{3,4} &= x_1^{2}x_2^{4}x_3^{14}x_4^{14}x_5^{7} + x_1^{4}x_2^{2}x_3^{14}x_4^{14}x_5^{7}, \  y_3= x_1x_2x_3^{11}x_4^{7}x_5^{7}.
\end{align*}
By a direct computation, we can see that there are polynomial $\bar B_u,\ u = 0,1,2$ such that
\begin{align}\label{ctno2} Sq^8(y_t)  = \sum_{u=0}^2Sq^{2^u}(\bar B_u)  \ \text{mod}(P_5^-(\omega_{(4)}) + \mathcal B)) 
\end{align}
where $\mathcal B$ is the subspace of $(P_5)_{35}$ spanned by all monomials $x$ such that $\max\{\nu_i(x): 1 \leqslant i \leqslant 5\}> 14$.
By combining (\ref{ctno1}) and (\ref{ctno1}), one gets 
\begin{align}\label{ctno3} a_{161} = \sum_{u=0}^4Sq^{2^u}(A_u) +  \ \text{mod}(P_5^-(\omega_{(4)})+\mathcal B),
\end{align}
where $A_u$ are suitable polynomials such that $y_t$ is not a term of $A_3$ with $t = 1,2,3$ (recall that a monomial $x$ in $P_k$ is called {\it a term} of a polynomial $f$ if it appears in the expression of $f$ in terms of the monomial basis of $P_k$) .

Let $(Sq^2)^3$ acts on the both sides of (\ref{ctno3}). Observe that if $x$ is a monomial in $P_5^-(\omega_{(4)})$ then, $(Sq^2)^3(x) \in P_5^-(1,4,4,2)$ and $(Sq^2)^3Sq^1 = (Sq^2)^3Sq^2 = 0$. So, we get
\begin{align*}(Sq^2)^3(a_{161}) =  \sum_{u=2}^4(Sq^2)^3(Sq^{2^u}(A_u))  \ \text{mod}(P_5^-(1,4,4,2) + (Sq^2)^3(\mathcal B)).
\end{align*}
It is easy to see that $f_{i,j} \notin P_5^-(1,4,4,2) + (Sq^2)^3(\mathcal B)$ and $f_{4,5}$ is a term of $(Sq^2)^3(a_{161})$ (recall that a polynomial $g$ in $P_k$ is called {\it a term} of a polynomial $f$ if  a monomial $x$ is a term of $g$, then it is also a term of $f$). 

By a direct computation, we see that $f_{i,j}$ is not a term of $(Sq^2)^3(Sq^{16}(A_4))$ for any $A_4 \in (P_5)_{19}$. If $f_{i,j}$ is a term of $(Sq^2)^3(Sq^8(A_3))$, then $y_t$ is a term of $A_3$ and $f_{i,j}$ is a term of $(Sq^2)^3(Sq^8(y_t))$ with some $t,\ t = 1,2,3$. However, the polynomial $f_{i,j}$ is not a term of $(Sq^2)^3(Sq^8(y_t))$. 

Since $f_{4,5}$ is a term of $(Sq^2)^3(a_{161})$, it must be a term of $(Sq^2)^3(Sq^4(A_2))$. We observe that $f_{4,5}$ is a term of $(Sq^2)^3(p_{1,j}),\ 1 \leqslant j \leqslant 8$, where   
\begin{align*}
p_{1,1} &= x_1x_2x_3^{7}x_4^{13}x_5^{13},\  p_{1,2} = x_1x_2^{2}x_3^{7}x_4^{11}x_5^{14},\  p_{1,3} = x_1x_2^{2}x_3^{7}x_4^{14}x_5^{11},\\  
p_{1,4} &= x_1^{2}x_2^{2}x_3^{7}x_4^{11}x_5^{13},\  p_{1,5} = x_1^{2}x_2^{2}x_3^{7}x_4^{13}x_5^{11},\
p_{1,6} = x_1x_2^{3}x_3^{7}x_4^{11}x_5^{13} + x_1^{3}x_2x_3^{7}x_4^{11}x_5^{13},\\  
p_{1,7} &= x_1x_2^{3}x_3^{7}x_4^{13}x_5^{11} + x_1^{3}x_2x_3^{7}x_4^{13}x_5^{11},\  
p_{1,8} = x_1x_2^{3}x_3^{7}x_4^{11}x_5^{13} + x_1^{3}x_2x_3^{7}x_4^{13}x_5^{11}.  
\end{align*}
Hence, one of the following polynomials is a term of $A_2$:
\begin{align*}
q_{1,1} &= x_1x_2x_3^{7}x_4^{11}x_5^{11},\  q_{1,2} = x_1x_2^{2}x_3^{7}x_4^{7}x_5^{14},\  q_{1,3} = x_1x_2^{2}x_3^{7}x_4^{14}x_5^{7},\\  
q_{1,4} &= x_1^{2}x_2^{2}x_3^{7}x_4^{7}x_5^{13},\  q_{1,5} = x_1^{2}x_2^{2}x_3^{7}x_4^{13}x_5^{7},\  q_{1,6} = x_1x_2^{3}x_3^{7}x_4^{7}x_5^{13} + x_1^{3}x_2x_3^{7}x_4^{7}x_5^{13},\\
q_{1,7} &= x_1x_2^{3}x_3^{7}x_4^{13}x_5^{7} + x_1^{3}x_2x_3^{7}x_4^{13}x_5^{7},\  q_{1,8} = x_1x_2^{3}x_3^{7}x_4^{13}x_5^{7} + x_1^{3}x_2x_3^{7}x_4^{7}x_5^{13},\\
q_{1,9} &= x_1x_2^{3}x_3^{7}x_4^{13}x_5^{7} + x_1^{3}x_2x_3^{7}x_4^{7}x_5^{13}.
\end{align*}
Suppose $q_{1,9}$ is a term of $A_2$. Then $\bar p = x_1^2x_2^{4}x_3^{14}x_4^{14}x_5^{7} + x_1^{4}x_2^2x_3^{14}x_4^{7}x_5^{14}$ is a term of  $(Sq^2)^3(a_{161} + Sq^4(q_{1,9}))$. This implies $q_{2,9} = x_1^3x_2x_3^{13}x_4^{7}x_5^{7} + x_1x_2^3x_3^{13}x_4^{7}x_5^{7}$  is a term of $A_2$. Then, $\tilde p = x_1^4x_2^2x_3^{14}x_4^{14}x_5^{7} + x_1^2x_2^4x_3^{14}x_4^{7}x_5^{14}$ is a term of $(Sq^2)^3(a_{161} + Sq^4(q_{1,9}+q_{2,9}))$. Hence,
$q_{3,9} = x_1x_2^3x_3^{7}x_4^{13}x_5^{7} + x_1^3x_2x_3^{7}x_4^{7}x_5^{13}$ is a term of $A_2$. Now, $f_{4,5}$ is a term of $(Sq^2)^3(a_{161} + Sq^4(q_{1,9}+q_{2,9}+ q_{3,9}))$ and $q_{i,9}, i=1,2,3,$ are not the terms of $A_2 + q_{1,9}+q_{2,9}+ q_{3,9}$. 

Now, we assume that $q_{i,9}, i=1,2,3,$ and $q_{1,1}$ are the terms of $A_2$. Then, $y^* = x_1x_2x_3^{11}x_4^{11}x_5^{11}$ is a term of $a_{161} + Sq^4(q_{1,1})$, hence $y^*$ is a term of $Sq^4(A_2 + q_{1,1}) + Sq^8(A_3) + Sq^{16}(A_4)$. Since $y_1, y_2, y_3$ are not the terms of $A_3$, $y^*$ is a term of $Sq^4(A_2)$. So, either $q_{2,1} = x_1x_2x_3^{11}x_4^{7}x_5^{11}$ or $q_{3,1} = x_1x_2x_3^{11}x_4^{11}x_5^{7}$ is a term of $A_2$. If $q_{2,1}$ is a term of $A_2$ then $f_{3,5}$ is a term of $(Sq^2)^3(a_{161} + Sq^4(q_{1,1} +  q_{2,1}))$. By an argument analogous to the previous one, we see that one of the following polynomials is a term of $A_2$:
\begin{align*}
q_{2,1} &= x_1x_2x_3^{11}x_4^{7}x_5^{11},\  q_{2,2} = x_1x_2^{2}x_3^{7}x_4^{7}x_5^{14},\  q_{2,3} = x_1x_2^{2}x_3^{14}x_4^{7}x_5^{7},\\
q_{2,4} &= x_1^{2}x_2^{2}x_3^{7}x_4^{7}x_5^{13},\  q_{2,5} = x_1^{2}x_2^{2}x_3^{13}x_4^{7}x_5^{7},\  q_{2,6} = x_1x_2^{3}x_3^{7}x_4^{7}x_5^{13} + x_1^{3}x_2x_3^{7}x_4^{7}x_5^{13},\\  
q_{2,7} &= x_1x_2^{3}x_3^{13}x_4^{7}x_5^{7} + x_1^{3}x_2x_3^{13}x_4^{7}x_5^{7},\  
q_{2,8} = x_1x_2^{3}x_3^{13}x_4^{7}x_5^{7}+ x_1^{3}x_2x_3^{7}x_4^{7}x_5^{13}.
\end{align*}
If $q_{3,1}$ is a term of $A_2$, then $f_{3,4}$ is a term of $(Sq^2)^3(a_{161} + Sq^4(q_{1,1} +  q_{3,1}))$. Hence, one of the following polynomials is a term of $A_2$:
\begin{align*}
q_{3,1} &= x_1x_2x_3^{11}x_4^{11}x_5^{7},\  q_{3,2} = x_1x_2^{2}x_3^{7}x_4^{14}x_5^{7},\  q_{3,3} = x_1x_2^{2}x_3^{14}x_4^{7}x_5^{7},\\
q_{3,4} &= x_1^{2}x_2^{2}x_3^{7}x_4^{13}x_5^{7},\  q_{3,5} = x_1^{2}x_2^{2}x_3^{13}x_4^{7}x_5^{7},\  q_{3,6} = x_1x_2^{3}x_3^{7}x_4^{13}x_5^{7} + x_1^{3}x_2x_3^{7}x_4^{13}x_5^{7},\\  
q_{3,7} &= x_1x_2^{3}x_3^{13}x_4^{7}x_5^{7} + x_1^{3}x_2x_3^{13}x_4^{7}x_5^{7},\  q_{3,8} = x_1x_2^{3}x_3^{13}x_4^{7}x_5^{7} + x_1^{3}x_2x_3^{7}x_4^{13}x_5^{7}.
\end{align*}
By repeating the above argument and after a finite number of steps, we see that $q_{i,j}$ is a term of $A_2$ for $i=1,2,3,\ 1 \leqslant j \leqslant 9$. Then, $f_{4,5}$ is a term of $(Sq^2)^3(a_{161} + \sum_{(i,j)}q_{i,j})$. However, it is not a term of 
$$(Sq^2)^3(Sq^4(A_2 + \sum_{(i,j)}q_{i,j})) +(Sq^2)^3(Sq^8(A_3)) + (Sq^2)^3(Sq^{16}(A_4)).$$
This is a contradiction. The lemma is proved. 
\end{proof}
\begin{proof}[Proof of Proposition \ref{mdd64}] Let $x$ be an admissible monomial in $P_5(\omega_{(4)})$. Then $x = x_jx_\ell x_ty^2$ with $y \in B_5(4,2,2)$. 

Let $z \in B_5(4,2,2)$.
By a direct computation using the results in Subsection \ref{subs51}, we see that if $x_jx_\ell x_tz^2 \ne a_{t}, \ 161 \leqslant t \leqslant 210$, then there is a monomial $w$ which is given in one of Lemmas \ref{inad64} and \ref{ina64} such that $x_jx_\ell x_tz^2= wz_1^{2^{u}}$ with suitable monomial $z_1 \in P_5$, and $u = \max\{j \in \mathbb Z : \omega_j(w) >0\}$. By Theorem \ref{dlcb1}, $x_jx_\ell x_tz^2$ is inadmissible. Since $x = x_jx_\ell x_ty^2$, $x$ is admissible and $y \in B_5(4,2,2)$, one gets $x= a_{t}$ for some $t,\ 161 \leqslant t \leqslant 210$. This implies $B_5(\omega_{(4)}) \subset \{a_{t} : \ 161 \leqslant t \leqslant 210\}$. 

By a direct computation, we see that there is a direct summand decomposition of the $\Sigma_5$-modules:
$$ QP_5(\omega_{(4)}) =  [\Sigma_5(a_{162})]_{\omega_{(4)}}\bigoplus  [\Sigma_5(a_{175})]_{\omega_{(4)}}\bigoplus  [\Sigma_5(a_{203})]_{\omega_{(4)}}.$$
Consider the homomorphis $\rho_i: P_5 \to P_5,\ 1 \leqslant i \leqslant 5,$ as defined in Section \ref{s2}. Let $\rho = \rho_2\rho_1\rho_2\rho_5\rho_2\rho_3\rho_1\rho_2$. Since $a_{162} \equiv_{\omega_{(4)}} g(a_{175}) + a_{175}$ and $a_{175} \equiv_{\omega_{(4)}} \rho_5(a_{203}) + a_{203}$, $QP_5(\omega_{(4)})$ is the $GL_5$-module generated by the class $[a_{203}]_{\omega_{(4)}}$.

We now prove the set $\{[a_{t}]_{\omega_{(4)}} : \ 161 \leqslant t \leqslant 210\}$ is linearly independent in $QP_5(\omega_{(4)})$. Suppose there is a linear relation
\begin{equation}\mathcal S = \sum_{t = 161}^{210}\gamma_ta_{t} \equiv_{\omega_{(4)}} 0,\label{ct541}
 \end{equation}
where $\gamma_t \in \mathbb F_2$. We prove that $\gamma_t = 0, \ \forall t,\ 161 \leqslant t \leqslant 210$.

Set $\mathcal S_1 = \rho_5(\mathcal S) + \mathcal S\equiv_{\omega_{(4)}} 0$, $\mathcal S_2 = (\rho_5\rho_2\rho_3)(\mathcal S_1) + \mathcal S_1 \equiv_{\omega_{(4)}} 0$.  A direct computation using (\ref{ct541}) shows that
\begin{equation}\mathcal S_3 := (\rho_5)(\mathcal S_2) + \mathcal S_2 \equiv_{\omega_{(4)}} \gamma_{181}a_{175} + \gamma_{183}a_{163} \equiv_{\omega_{(4)}} 0.\label{ct542}
 \end{equation}

By applying $\rho_5$ to (\ref{ct542}), we obtain
\begin{equation}(\rho_5\rho_2\rho_1)(\mathcal S_3)+ \mathcal S_3 \equiv_{\omega_{(4)}} \gamma_{181}a_{179} \equiv_{\omega_{(4)}} 0.\label{ct543}
\end{equation}
We have $a_{162} \equiv_{\omega_{(4)}} a_{179} + (\rho_3\rho_2\rho_5\rho_2\rho_3)(a_{179})$ and $a_{162} \equiv_{\omega_{(4)}} \rho_2(a_{161})$. Hence, Lemma \ref{noz1} implies that $[a_{179}]_{\omega_{(4)}} \ne 0$. So, from the relation (\ref{ct543}), we obtain $\gamma_{181} = 0$.

By a simple computation, we can see that the action of $\Sigma_5$ on $QP_5(\omega_{(4)})$ induces the one of it on the set $\{[a_t]_{\omega_{(4)}} : 171 \leqslant t \leqslant 200\}$. Furthermore, this action is transitive. Hence, from the relations $\sigma(\mathcal S) \equiv_{\omega_{(4)}} 0$ with $\sigma \in \Sigma_5$, we get $\gamma_t = 0$ for $171 \leqslant t \leqslant 200$.

Now, using the above equalities, one gets $\mathcal S_2  \equiv_{\omega_{(4)}} \gamma_{208}a_{185} \equiv_{\omega_{(4)}} 0$. This implies $\gamma_{208} = 0$. By using this and the relations $\sigma(\mathcal S) \equiv_{\omega_{(4)}} 0$ with $\sigma \in \Sigma_5$, we obtain $\gamma_t = 0$ for $201 \leqslant t \leqslant 210$. 

By computing $(\rho_5\rho_2\rho_1)(\mathcal S_1) + \mathcal S_1$, we have
$$ (\rho_5\rho_2\rho_1)(\mathcal S_1) + \mathcal S_1 \equiv_{\omega_{(4)}}  \gamma_{165}a_{162} \equiv_{\omega_{(4)}} 0.$$
Since $[a_{162}]_{\omega_{(4)}} = [\rho_2(a_{161})]_{\omega_{(4)}} \ne 0$, we get $\gamma_{165} = 0$.

We observe that the action of $\Sigma_5$ on $QP_5(\omega_{(4)})$ induces the one on the set $\{[a_t]_{\omega_{(4)}} : 161 \leqslant t \leqslant 170\}$. Since this action is transitive, we get $\gamma_t = 0$ for $161 \leqslant t \leqslant 170$.
 The proposition is proved.
\end{proof}

Consider the monomials $a_t = a_{35,t},\ 211 \leqslant t \leqslant 225$ as given in Subsection \ref{ss65}.

\begin{props}\label{mdd65} The $\mathbb F_2$-vector space $QP_5(\omega_{(5)})$ is an $GL_5$-module generated by the class $[a_{225}]_{\omega_{(5)}}$ and $B_5(\omega_{(5)}) = \{a_{t} : 211 \leqslant t \leqslant 225\}.$ Consequently, $\dim QP_5(\omega_{(5)}) = 15.$
\end{props} 
We prepare some lemmas for the proof of this proposition. The proof of the following lemma is straightforward.
\begin{lems}\label{bdd22} The following monomials are strictly inadmissible:
 $$x_j^2x_\ell x_t^2 x_u^3x_v^3,  i < j; \ \ x_j^2x_\ell^3x_t^3 x_u^3.$$
Here $(j,\ell,t,u,v) $ is a permutation of $(1,2,3,4,5)$.
\end{lems}

\begin{lems}\label{inad65} All permutations of the following monomials are strictly inadmissible:
 $$x_1x_2^{6}x_3^{6}x_4^{7}x_5^{15} ,\  x_1x_2^{6}x_3^{7}x_4^{7}x_5^{14},\ x_1^{7}x_2^{7}x_3^{9}x_4^{6}x_5^{6}.$$
\end{lems}
\begin{proof} We prove the lemma for $x = x_1x_2^{6}x_3^{6}x_4^{7}x_5^{15}$. By a direct computation, we have
\begin{align*}
x &= Sq^1\big(x_1x_2^{5}x_3^{6}x_4^{7}x_5^{15} + x_1^{4}x_2^{3}x_3^{5}x_4^{7}x_5^{15}\big) + Sq^2\big(x_1^{2}x_2^{3}x_3^{6}x_4^{7}x_5^{15}\big),\  \text{ mod}\big(P_5^-(\omega_{(5)})\big).
\end{align*}
So, all permutations of $x$ are strictly inadmissible.
\end{proof}
\begin{lems}\label{inad651} The following monomials are strictly inadmissible:
\begin{align*}
&x_1^{3}x_2^{5}x_3^{6}x_4^{14}x_5^{7} ,\  x_1^{3}x_2^{5}x_3^{14}x_4^{6}x_5^{7} ,\  x_1^{3}x_2^{5}x_3^{14}x_4^{7}x_5^{6} ,\\ 
&  x_1^{3}x_2^{13}x_3^{6}x_4^{6}x_5^{7} ,\  x_1^{3}x_2^{13}x_3^{6}x_4^{7}x_5^{6},\  x_1^{3}x_2^{13}x_3^{7}x_4^{6}x_5^{6}.
\end{align*}
\end{lems}
\begin{proof} We prove the lemma for $x = x_1^{3}x_2^{5}x_3^{14}x_4^{6}x_5^{7}$. By a direct computation, we have
\begin{align*}
x &= x_1^{3}x_2^{5}x_3^{7}x_4^{6}x_5^{14} + Sq^1\big(x_1^{3}x_2^{6}x_3^{7}x_4^{5}x_5^{13} + x_1^{3}x_2^{6}x_3^{13}x_4^{5}x_5^{7}\big) + Sq^2\big(x_1^{5}x_2^{3}x_3^{9}x_4^{3}x_5^{13}\\
&\quad + x_1^{5}x_2^{3}x_3^{13}x_4^{3}x_5^{9} + x_1^{5}x_2^{5}x_3^{7}x_4^{5}x_5^{11} + x_1^{5}x_2^{5}x_3^{7}x_4^{9}x_5^{7} + x_1^{5}x_2^{5}x_3^{11}x_4^{5}x_5^{7}\\
&\quad + x_1^{5}x_2^{9}x_3^{7}x_4^{5}x_5^{7}\big) + Sq^4\big(x_1^{3}x_2^{3}x_3^{9}x_4^{3}x_5^{13} + x_1^{3}x_2^{3}x_3^{13}x_4^{3}x_5^{9} + x_1^{3}x_2^{5}x_3^{7}x_4^{5}x_5^{11}\\
&\quad + x_1^{3}x_2^{5}x_3^{7}x_4^{9}x_5^{7} + x_1^{3}x_2^{5}x_3^{9}x_4^{5}x_5^{9} + x_1^{3}x_2^{5}x_3^{11}x_4^{5}x_5^{7} + x_1^{3}x_2^{9}x_3^{7}x_4^{5}x_5^{7}\big)\\
&\quad + Sq^8\big( x_1^{3}x_2^{3}x_3^{9}x_4^{3}x_5^{9}\big),\ \text{ mod}\big(P_5^-(\omega_{(5)})\big).
\end{align*}
This equality shows that the monomial $x$ is strictly inadmissible.
\end{proof}

\begin{lems}\label{noz2} 
The class $[a_{215}]_{\omega_{(5)}}$ is non-zero in the vector space $QP_5(\omega_{(5)})$.
\end{lems}
\begin{proof} Suppose the contrary that $[a_{215}]_{\omega_{(5)}} = 0$. Then
\begin{align}\label{ctno4} a_{215} = \sum_{u=0}^4Sq^{2^u}(C_u)  \ \text{mod}(P_5^-(\omega_{(5)})),
\end{align}
where $C_u$ are suitable polynomials in $P_5$. Let $(Sq^2)^3$ acts on the both sides of (\ref{ctno4}). Observe that if $x$ is a monomial in $P_5^-(\omega_{(5)})$ then, $(Sq^2)^3(x) \in P_5^-(1,4,4,2)$. Since $(Sq^2)^3Sq^1 = 0$ and $(Sq^2)^3Sq^2 = 0$, we get
\begin{align*}(Sq^2)^3(a_{215}) =  \sum_{u=2}^4(Sq^2)^3(Sq^{2^u}(C_u))  \ \text{mod}(P_5^-(1,4,4,2)).
\end{align*}
We denote
\begin{align*}
&g_{1,2} = x_1^{15}x_2^{3}x_3^{3}x_4^{5}x_5^{5},\  g_{1,3} = x_1^{15}x_2^{3}x_3^{5}x_4^{3}x_5^{5},\  g_{1,4} = x_1^{15}x_2^{3}x_3^{5}x_4^{5}x_5^{3},\\  
&g_{2,3} = x_1^{15}x_2^{5}x_3^{3}x_4^{3}x_5^{5},\  g_{2,4} = x_1^{15}x_2^{5}x_3^{3}x_4^{5}x_5^{3},\  g_{3,4} = x_1^{15}x_2^{5}x_3^{5}x_4^{3}x_5^{3}.\\ 
&q_1 = x_1^{15}x_2^{3}x_3^{3}x_4^{3}x_5^{9},\  q_2 = x_1^{15}x_2^{3}x_3^{3}x_4^{9}x_5^{3},\  q_3 = x_1^{15}x_2^{3}x_3^{9}x_4^{3}x_5^{3},\  q_4 = x_1^{15}x_2^{9}x_3^{3}x_4^{3}x_5^{3},\\
&r_1 = x_1^{15}x_2^{3}x_3^{3}x_4^{3}x_5^{7},\  r_2 = x_1^{15}x_2^{3}x_3^{3}x_4^{7}x_5^{3},\  r_3 = x_1^{15}x_2^{3}x_3^{7}x_4^{3}x_5^{3},\  r_t = x_1^{15}x_2^{7}x_3^{3}x_4^{3}x_5^{3},\\
&s_1 = x_1^{15}x_2^{6}x_3^{6}x_4^{6}x_5^{8},\  s_2 = x_1^{15}x_2^{6}x_3^{6}x_4^{8}x_5^{6},\  s_3 = x_1^{15}x_2^{6}x_3^{8}x_4^{6}x_5^{6},\  s_4 = x_1^{15}x_2^{8}x_3^{6}x_4^{6}x_5^{6}.
\end{align*}
Observe that the polynomial $s = \sum_{u=1}^4s_u$ is a term of $(Sq^2)^3(a_{215})$. It is not a term of $(P_5^-(1,4,4,2))$, $(Sq^2)^3(Sq^{8}(C_3))$ and $(Sq^2)^3(Sq^{16}(C_4))$ hence, it is a term of $(Sq^2)^3(Sq^{4}(C_2))$. Then, there is $(i,j)$ such that either $g_{i,j}$ or $r = \sum_{u=1}^4r_u$ is a term of $C_2$. If $r$ is a term of $C_2$, then $r_1$ is a term of $C_2$. Then, the monomial $\bar x = x_1^{15}x_2^3x_3^3x_4^3x_5^{11}$ is a term of $a_{215}+ Sq^4(r_1)$. So $\bar x$ is a term of $Sq^{4}(C_2 + r_1)$. So, $r_1$ is a term of  $C_2 + r_1$. This contradicts the fact that $r_1$ is a term of $C_2$. 

Thus, we have proved that $r$ is not a term of $C_2$, hence $g_{i,j}$ is a term of $C_2$. We can assume that $g_{1,2}$ is a term of $C_2$.

Now, $x_1^{15}x_2^3x_3^3x_4^9x_5^3$ and $x_1^{15}x_2^3x_3^3x_4^5x_5^9$ are the terms of $a_{215}+ Sq^4(g_{1,2})$. Since $r_u, u = 1,2,3,4$ is not a term of $C_2$, $q_1$ and $q_2$ are the term of $C_1$. Then,  the following monomials are the terms of $a_{215}+ Sq^4(g_{1,2}) + Sq^2(q_1 + q_2)$:
$$x_1^{15}x_2^{5}x_3^{3}x_4^{3}x_5^{9},\   x_1^{15}x_2^{3}x_3^{5}x_4^{3}x_5^{9},\   x_1^{15}x_2^{5}x_3^{3}x_4^{9}x_5^{3},\   x_1^{15}x_2^{3}x_3^{5}x_4^{9}x_5^{3}. $$
Since $r_u$ is not a term of $C_2$, the monomials $g_{1,3}, g_{1,4}, g_{2,3}, g_{2,4}$ are the terms of $C_2$. Then, $x_1^{15}x_2^{3}x_3^{9}x_4^{3}x_5^{5},\  x_1^{15}x_2^{9}x_3^{3}x_4^{5}x_5^{3}$ are the terms of 
$$a_{215}+ Sq^4(g_{1,2} + g_{1,3} + g_{1,4} + g_{2,3} + g_{2,4}) + Sq^2(q_1 + q_2).$$
So, $q_3,\ q_4$ are the terms of $C_1$. Then, $x_1^{15}x_2^{5}x_3^{9}x_4^{3}x_5^{3}$ is a term of 
$$a_{215}+ Sq^4(g_{1,2} + g_{1,3} + g_{1,4} + g_{2,3} + g_{2,4}) + Sq^2(q_1 + q_2+ q_3 + q_4).$$
Hence, $g_{3,4}$ is a term of $C_2$. Now, the monomial $x_1^{15}x_2^{6}x_3^{6}x_4^{6}x_5^{8}$ is a term of
$$(Sq^2)^3(a_{215}+ Sq^4(g_{1,2} + g_{1,3} + g_{1,4} + g_{2,3} + g_{2,4}+ g_{3,4})).$$
So, there is $(i,j)$ such that $g_{i,j}$ is a term of $C_2 + \sum_{(i,j)}g_{i,j}$. This contradicts the fact that $g_{i,j}$ is a term of $C_2$. The lemma is proved. 
\end{proof}
\begin{proof}[Proof of Proposition \ref{mdd65}] Let $x$ be an admissible monomial in $P_5$ such that $\omega(x) = \omega_{(5)}$. Then $x = x_jx_\ell x_ty^2$ with $y \in B_5(4,4,1)$. 

Let $z \in B_5(4,4,1)$.
By a direct computation using the results in Subsection \ref{subs51}, we see that if $x_jx_\ell x_tz^2 \ne a_{t}, \ 211 \leqslant t \leqslant 225$, then there is a monomial $w$ which is given in one of Lemmas \ref{bdd22} and \ref{inad65} such that $x_jx_\ell x_tz^2= wz_1^{2^{u}}$ with suitable monomial $z_1 \in P_5$, and $u = \max\{j \in \mathbb Z : \omega_j(w) >0\}$. By Theorem \ref{dlcb1}, $x_jx_\ell x_tz^2$ is inadmissible. Since $x = x_jx_\ell x_ty^2$ and $x$ is admissible, one gets $x= a_{t}$ for some $t,\ 211 \leqslant t \leqslant 225$. Hence, $B_5(\omega_{(5)}) \subset \{a_{t} : \ 211 \leqslant t \leqslant 225\}$. 

By a direct computation, we see that there is a direct summand decomposition of the $\Sigma_5$-modules:
$$ QP_5(\omega_{(5)}) =  [\Sigma_5(a_{214})]_{\omega_{(5)}}\bigoplus  [\Sigma_5(a_{225})]_{\omega_{(5)}}.$$
Since $a_{214} \equiv_{\omega_{(5)}} a_{225} + \rho_5(a_{225})$, $QP_5(\omega_{(5)})$ is the $GL_5$-module generated by the class $[a_{225}]_{\omega_{(5)}}$.

We now prove the set $\{[a_{t}]_{\omega_{(5)}} : \ 211 \leqslant t \leqslant 225\}$ is linearly independent in $QP_5(\omega_{(5)})$. Suppose there is a linear relation
\begin{equation}\mathcal S = \sum_{t = 211}^{225}\gamma_ta_{t} \equiv_{\omega_{(5)}}  0,\label{ct551}
\end{equation} 
where $\gamma_t \in \mathbb F_2$. We prove $\gamma_t = 0, \ \forall t,\ 211 \leqslant t \leqslant 225$.

Set $\mathcal S_1 = \rho_5(\mathcal S) + \mathcal S$. By a direct computation using (\ref{ct551}), we have
\begin{equation}(\rho_5\rho_2\rho_1\rho_2)(\mathcal S_1) + \mathcal S_1 \equiv_{\omega_{(5)}} \gamma_{215}a_{214} \equiv_{\omega_{(5)}}  0.\label{ct552}
\end{equation} 
From Lemma \ref{noz2}, we have $[a_{214}]_{\omega_{(5)}} = [\rho_1(a_{215}]_{\omega_{(5)}}) \ne 0$. Hence, the relation (\ref{ct552}) implies $\gamma_{215} = 0$.

Note that the action of $\Sigma_5$ on $QP_5(\omega_{(5)})$ induces the one of it on the set $\{[a_t]_{\omega_{(5)}} : 211 \leqslant t \leqslant 225\}$ and this action is transitive. So, we get $\gamma_t = 0$ for $211 \leqslant t \leqslant 225$. Then, we have 
$$\mathcal S_1 \equiv_{\omega_{(5)}} \gamma_{225}a_{214} + \gamma_{222}a_{219} +  \gamma_{223}a_{220} +  \gamma_{224}a_{221} \equiv_{\omega_{(5)}} 0.$$
The last equalities implies $\gamma_{225} = 0$. Since the action of $\Sigma_5$ on the set $\{[a_t]_{\omega_{(5)}} : 216 \leqslant t \leqslant 225\}$ is transitive, we obtain $\gamma_t = 0, \ \forall t,\ 216 \leqslant t \leqslant 225$.
The proposition is proved.
\end{proof}

\section{Proof of Theorem \ref{pthm3}}\label{s5}
\setcounter{equation}{0}

In this section, we prove Theorem \ref{pthm3} by using the results in Section \ref{s4}. From Theorem \ref{cthm}, we see that the homomorphism
$$(\widetilde{Sq}^0_*)^{s-3}: (\mathbb F_2 \otimes_{\mathcal A} P_5)_{5(2^s-1)}^{GL_5} \longrightarrow (\mathbb F_2 \otimes_{\mathcal A} P_k)_{35}^{GL_5}$$
is an isomorphism for every $s \geqslant 3$. So, we need only to prove the theorem for $s = 1,2,3$.

\medskip
\subsection{The Cases $s=1,2$}\

\medskip
For $s=1$, we have $5(2^s-1) = 5$.
\begin{props} $(\mathbb F_2 \otimes_{\mathcal A} P_5)_{5}^{GL_5} = 0$.
\end{props} 
\begin{proof} Denote by $a_t = a_{5,t},\ 1 \leqslant t \leqslant 46,$ the admissible monomials of degree 5 in $P_5$ (see Subsection \ref{ss61}). Suppose $f \in (P_5)_5$ such that $[f] \in (\mathbb F_2 \otimes_{\mathcal A} P_5)_{5}^{GL_5}$. Then, $f \equiv \sum_{t=1}^{46}\gamma_ta_t$ with $\gamma_t \in \mathbb F_2$ and $\rho_i(f) + f \equiv 0$ for $i=1,2,3,4,5.$ By computing directly from the relations $\rho_i(f) + f \equiv 0$, for $i=1,2,3,4,$ we easily obtain $\gamma_t = \gamma_1, \ 1 \leqslant t \leqslant 30,$ and $\gamma_t = 0$ for $31 \leqslant t \leqslant 45$. Now, by computing $\rho_5(f) +f$ in terms of the admissible monomials, we get
$$\rho_5(f) +f \equiv \gamma_1a_4 + \gamma_{46}a_{31} + \text{ other terms} \equiv 0.$$
The last equality implies $\gamma_1 = \gamma_{46} = 0$. The proposition follows.
\end{proof}

For $s=2$, the theorem has been proved in \cite{su5}. We have 

\begin{props}[See \cite{su5}]\label{mdds2} $(\mathbb F_2 \otimes_{\mathcal A} P_5)_{15}^{GL_5}$ is  an $\mathbb F_2$-vector space of dimension  $2$ with a basis consisting of the $2$ classes represented by the polynomials $p$ and $q$, which are determined as in Subsection $\ref{ss67}$.
\end{props} 

\subsection{The Case $s = 3$}\

\medskip
For $s = 3, \ 5(2^s-1) = 35$. From the results in Section \ref{s4}, we have 

\begin{align*}
&(\mathbb F_2 \otimes_{\mathcal A} P_5)_{35} \cong \text{Ker}(\widetilde{Sq}^0_*)_{(5,15)} \bigoplus (\mathbb F_2 \otimes_{\mathcal A} P_5)_{15},\\
&\text{\rm Ker}(\widetilde{Sq}^0_*)_{(5,15)} =  \bigoplus_{1\leqslant j \leqslant 5} QP_5(\omega_{(j)}).
\end{align*}

Since $QP_5(\omega_{(2)}) =0$ and $QP_5(\omega_{(3)}) = 0$, we need only to compute $QP_5(\omega_{(j)})^{GL_5}$ for $j = 1,4,5.$

\medskip
\subsubsection{Computation of $QP_5(\omega_{(1)})^{GL_5}$}\

\medskip
We prove the following.
\begin{props}\label{md551} $QP_5(\omega_{(1)})^{GL_5} = 0$.
\end{props}
We denote the admissible monomials in $(QP_5^0)_{35}$ by $b_t,\ 1 \leqslant t \leqslant 460$, as given in Subsection \ref{ss65}. By a direct computation from Proposition \ref{mdd61}, there is a direct summand decomposition of the $\Sigma_5$-modules: 
\begin{align*} &QP_5(\omega_{(1)}) = (QP_5^0)_{35}\bigoplus QP_5^+(\omega_{(1)}),\\
&(QP_5^0)_{35} = [\Sigma_5(b_1)]  \bigoplus [\Sigma_5(b_{61})] \bigoplus [\Sigma_5(b_{91})] \\
&\hskip3cm\bigoplus[\Sigma_5(b_{141})] \bigoplus\langle\{[b_i] : 181 \leqslant i \leqslant 460\}\rangle ,\end{align*}

We prepare some lemmas for the proof of the proposition.
\begin{lems}\label{bd551} We have
\begin{align*}
&[\Sigma_5(b_1)]^{\Sigma_5} = \langle[p_1]\rangle, \ [\Sigma_5(b_{61})]^{\Sigma_5} = \langle[p_2]\rangle,\ [\Sigma_5(b_{91})]^{\Sigma_5} = \langle[p_3]\rangle,\\ 
&[\Sigma_5(b_{141})]^{\Sigma_5} = 0,\ \langle\{[b_i] : 181 \leqslant i \leqslant 460\} \rangle^{\Sigma_5} = \langle [p_4], [p_5], [p_6]\rangle.
\end{align*}
Here, 
$p_1 =\sum_{t=1}^{60}b_t$, $p_2=\sum_{t=61}^{90}b_t$,
$p_3=\sum_{t=61}^{130}b_t$, $p_4= \sum_{t=181}^{250}b_t + \sum_{t=371}^{415}b_t$,
$p_5= \sum_{t=251}^{265}b_t + \sum_{t=271}^{395}b_t$, $p_6 = \sum_{t=266}^{375}b_t + \sum_{t=396}^{415}b_t$.
\end{lems}
\begin{proof}[Outline of the proof] The set $B_1 := \{[b_t] : 1 \leqslant t \leqslant 60\}$ is a basis of $[\Sigma_5(b_1)]$. The action of $\Sigma_5$ on $QP_5(\omega_{(1)})$ induces the one of it on $B_1$. Furthermore, this action is transitive. Hence, if $f = \sum_{t=1}^{60}\gamma_tb_t$ with $\gamma_t \in \mathbb F_2$ and $[f] \in  [\Sigma_5(b_1)]^{\Sigma_5}$, then the relations $\rho_j(f) \equiv f, j = 1,2,3,4,$ imply $\gamma_t = \gamma_1, \forall t,\ 1 \leqslant t \leqslant 60$. The first part of the lemma is proved for $i=1$. The case $[\Sigma_5(b_{61})]$ is proved by a similar argument.

The set $B_3 := \{[b_t] : 91 \leqslant t \leqslant 140\}$ is the basis of  $[\Sigma_5(b_{91})]$. Suppose $f = \sum_{t=91}^{140}\gamma_tb_t$ with $\gamma_t \in \mathbb F_2$ and $[f] \in  [\Sigma_5(b_{91})]^{\Sigma_5}$. By computing directly the relations $\rho_j(f) \equiv f, j = 1,2,3,4$, we get $\gamma_t = 0$ for $131 \leqslant t \leqslant 140$ and $\gamma_t = \gamma_{91}$ for $91 \leqslant t \leqslant 130$. So, $f \equiv p_3$.

The set $B_4 := \{[b_t] : 141 \leqslant t \leqslant 180\}$ is the basis of  $[\Sigma_5(b_{141})]$. Suppose $f = \sum_{t=141}^{180}\gamma_tb_t$ with $\gamma_t \in \mathbb F_2$ and $[f] \in  [\Sigma_5(b_{141})]^{\Sigma_5}$. By computing from the relations $\rho_j(f) \equiv f, j = 1,2,3,4$, we obtain $\gamma_t = 0, \forall t, 141 \leqslant t \leqslant 180$. Hence, $f \equiv 0$.

Suppose $f = \sum_{t=181}^{460}\gamma_tb_t$ with $\gamma_t \in \mathbb F_2$ and $[f] \in  \langle\{[b_t] : 181 \leqslant t \leqslant 460\}\rangle^{\Sigma_5}$.  By a direct computation from the relations $\rho_j(f) \equiv f, j = 1,2,3,4$, we get $f \equiv \gamma_{181}p_4 + \gamma_{251}p_5 + \gamma_{266}p_6$. The second part is proved.
\end{proof}

\begin{lems}\label{bd552} We have 
$QP_5^+(\omega_{(1)})^{\Sigma_5} = \langle [p_7], [p_8], [p_9]\rangle.$
Here, $p_7 = \sum_{t=11}^{31}a_t + \sum_{t=38}^{51}a_t + \sum_{t=66}^{85}a_t$, $p_8 = \sum_{t=1}^{10}a_t + \sum_{t=52}^{85}a_t$, $p_9 = \sum_{t=32}^{85}a_t$.
\end{lems}
\begin{proof}[Outline of the proof] Let $[f] \in  QP_5^+(\omega_{(1)})^{\Sigma_5}$. By Proposition \ref{mdd61}, $f \equiv \sum_{t=1}^{160}\gamma_ta_t$. We compute $\rho_j(f) + f$ in terms of $a_t, 1\leqslant t \leqslant 160,$ (mod($\mathcal A^+P_5)$. By computing from the relations $\rho_j(f) + f \equiv 0$ with $j = 1,2,3,4,$ we get $f \equiv \gamma_1p_7 + \gamma_4p_8 + \gamma_{13}p_9$. Hence, the lemma follows.
\end{proof}

\begin{proof}[Proof of Proposition \ref{md551}] Combining Lemmas \ref{bd551} and \ref{bd552} gives
$$QP_5(\omega_{(1)})^{\Sigma_5} = \langle\{[p_j] : 1 \leqslant j \leqslant 9\}\rangle.$$ 

Let $f \in P_5(\omega_{(1)})$ such that $[f] \in QP_5(\omega_{(1)})^{GL_5}$. Then, $f \equiv \sum_{j=1}^9\gamma_jp_j$ with $\gamma_j \in \mathbb F_2$. By a direct computation using Theorem \ref{dlsig}, we  have
\begin{align*}\rho_5(f) + f &\equiv \gamma_1b_7 + \gamma_2b_{48} + \gamma_3b_{36} + \gamma_4b_{172} + \gamma_5b_{398} + (\gamma_1 + \gamma_5 + \gamma_6)b_{21}\\
&\quad + \gamma_7b_{371} + \gamma_8a_{136} + \gamma_9b_{182} + \text{ other terms} \equiv 0.
\end{align*}
The last equality implies $\gamma_t = 0$ for $1 \leqslant t \leqslant 9$. The proposition is proved.
\end{proof}

\medskip
\subsubsection{Computation of $QP_5(\omega_{(4)})^{GL_5}$}\

\medskip
In this part, we prove the following.
\begin{props}\label{md553} $QP_5(\omega_{(4)})^{GL_5} = 0$.
\end{props}
From Proposition \ref{mdd64}, we can see that there is a direct summand decomposition of the $\Sigma_5$-modules:  
$$ QP_5(\omega_{(4)}) =  [\Sigma_5(a_{161})]_{\omega_{(4)}}\bigoplus [\Sigma_5(a_{171})]_{\omega_{(4)}}\bigoplus  [\Sigma_5(a_{201})]_{\omega_{(4)}}.$$

\begin{lems}\label{bdd555} We have $ [\Sigma_5(a_{161})]_{\omega_{(4)}}^{\Sigma_5} =  \langle [p_{10}]_{\omega_{(4)}}\rangle$ with $p_{10} = \sum_{t=161}^{170}a_t$.
\end{lems}
\begin{proof}
The set $\{[a_{t}]_{\omega_{(4)}}: 161 \leqslant t \leqslant 170\}$ is a basis of $ [\Sigma_5(a_{161})]_{\omega_{(4)}}$. If $f \in P_5(\omega_{(4)})$ such that $[f]_{\omega_{(4)}} \in  [\Sigma_5(a_{161})]_{\omega_{(4)}}^{\Sigma_5}$, then $f \equiv_{\omega_{(4)}}p_{10} = \sum_{t=161}^{170}\gamma_ta_t$ with $\gamma_t \in \mathbb F_2$ and $\rho_j(f) + f \equiv_{\omega_{(4)}} 0$, for $j=1,2,3,4.$ Since the action of $\Sigma_5$ on $QP_5(\omega_{(4)})$ induces the one of it on $\{[a_{t}]_{\omega_{(4)}}: 161 \leqslant t \leqslant 170\}$ which is transitive. Hence, from the relation $\rho_j(f) + f \equiv_{\omega_{(4)}} 0$, with $j=1,2,3,4,$ we obtain $\gamma_t = \gamma_{161}, \forall t, 161 \leqslant t \leqslant 170$. The lemma follows. 
\end{proof}

By an argument similar to the proof of Lemma \ref{bdd555}, we get the following.
\begin{lems}\label{bdd556} We have $ [\Sigma_5(a_{171})]_{\omega_{(4)}}^{\Sigma_5} =  \langle [p_{11}]_{\omega_{(4)}}\rangle$ with $p_{11} = \sum_{t=171}^{200}a_t$.
\end{lems}

\begin{lems}\label{bdd557} $ [\Sigma_5(a_{203})]_{\omega_{(4)}}^{\Sigma_5} =  0$.
\end{lems}
\begin{proof} The set $\{[a_{t}]_{\omega_{(4)}}: 201 \leqslant t \leqslant 210\}$ is a basis of $ [\Sigma_5(a_{203})]_{\omega_{(4)}}$. If $f \in P_5(\omega_{(4)})$ such that $[f]_{\omega_{(4)}} \in  [\Sigma_5(a_{204})]_{\omega_{(4)}}^{\Sigma_5}$, then $f \equiv_{\omega_{(4)}} \sum_{t=201}^{210}\gamma_ta_t$ with $\gamma_t \in \mathbb F_2$ and $\rho_j(f) + f \equiv_{\omega_{(4)}} 0$, for $j=1,2,3,4.$ By a direct computation, we get
\begin{align*}
\rho_1(f) + f &\equiv_{\omega_{(4)}} \gamma_{204}a_{201} + \gamma_{205}a_{202} + \gamma_{206}a_{203} + (\gamma_{207} + \gamma_{209})a_{207}\\
& + (\gamma_{208} + \gamma_{210})a_{208} +  (\gamma_{207} + \gamma_{209})a_{209}+ (\gamma_{208} + \gamma_{210})a_{210} \equiv_{\omega_{(4)}} 0,\\
\rho_2(f) + f &\equiv_{\omega_{(4)}} (\gamma_{201} + \gamma_{204})a_{201} + (\gamma_{202} + \gamma_{205})a_{202} + (\gamma_{203} + \gamma_{207} + \gamma_{208})a_{203}\\
& + (\gamma_{201} + \gamma_{204})a_{204} + (\gamma_{202} + \gamma_{205})a_{205} + (\gamma_{206} + \gamma_{208})a_{206}\\
& + (\gamma_{203} + \gamma_{206} + \gamma_{207})a_{207} + (\gamma_{206} + \gamma_{208})a_{208} + \gamma_{210}a_{209} \equiv_{\omega_{(4)}} 0, \\
\rho_3(f) + f &\equiv_{\omega_{(4)}} \gamma_{204}a_{201} +  (\gamma_{202} + \gamma_{203} + \gamma_{206})a_{202} +   (\gamma_{202} + \gamma_{203} + \gamma_{205})a_{203}\\
& +  (\gamma_{205} + \gamma_{206})a_{205} +  (\gamma_{205} + \gamma_{206})a_{206} + (\gamma_{207} + \gamma_{208})a_{207}\\
& +  (\gamma_{207} + \gamma_{208})a_{208} +  (\gamma_{209} + \gamma_{210})a_{209} +  (\gamma_{209} + \gamma_{210})a_{210} \equiv_{\omega_{(4)}} 0, \\ 
\rho_4(f) + f &\equiv_{\omega_{(4)}} (\gamma_{201} + \gamma_{202})a_{201} +  (\gamma_{201} + \gamma_{202})a_{202} +   \gamma_{206}a_{203}\\
& +  (\gamma_{204} + \gamma_{205})a_{204} +  (\gamma_{204} + \gamma_{205})a_{205}  +  \gamma_{208}a_{207} +  \gamma_{210}a_{209} \equiv_{\omega_{(4)}} 0.
\end{align*}
From the above equalities, we obtain $\gamma_t = 0, \forall t, 201 \leqslant t \leqslant 210$. The lemma is proved. 
\end{proof}
\begin{proof}[Proof of Proposition \ref{md553}] By combining Lemmas \ref{bdd555}-\ref{bdd557}, we have
$$QP_5(\omega_{(4)})^{\Sigma_5} = \langle [p_{10}]_{\omega_{(4)}}, [p_{11}]_{\omega_{(4)}}\rangle,$$ 
where $p_{10} = \sum_{t=161}^{170}a_t$ and $p_{11} = \sum_{t=171}^{200}a_t$.

Let $f \in P_5(\omega_{(4)})$ such that $[f] \in QP_5(\omega_{(4)})^{GL_5}$. Then, $f \equiv_{\omega_{(4)}} \gamma_1 p_{10}+ \gamma_2p_{11}$ with $\gamma_1, \gamma_2 \in \mathbb F_2$. By a direct computation, we have
\begin{align*}\rho_5(f) + f &\equiv_{\omega_{(4)}} (\gamma_1 + \gamma_2)a_{162}+ \gamma_2a_{171} + \text{other terms} \equiv_{\omega_{(4)}} 0.
\end{align*}
This equality implies $\gamma_1 = \gamma_2 = 0$. The proposition is proved.
\end{proof}

\subsubsection{Computation of $QP_5(\omega_{(5)})^{GL_5}$}\

\medskip
In this part, we prove the following.
\begin{props}\label{md554} $QP_5(\omega_{(5)})^{GL_5} = 0$.
\end{props}
From Proposition \ref{mdd65}, there is a direct summand decomposition of the $\Sigma_5$-modules:  
$$ QP_5(\omega_{(5)}) =  [\Sigma_5(a_{211})]_{\omega_{(5)}}\bigoplus  [\Sigma_5(a_{216})]_{\omega_{(5)}}.$$

By an argument similar to the proof of Lemma \ref{bdd555}, we get the following.
\begin{lems}\label{bdd558} We have $ [\Sigma_5(a_{211})]_{\omega_{(5)}}^{\Sigma_5} =  \langle [p_{12}]_{\omega_{(5)}}\rangle$ and $ [\Sigma_5(a_{216})]_{\omega_{(5)}}^{\Sigma_5} =  \langle [p_{13}]_{\omega_{(5)}}\rangle$ with $p_{12} = \sum_{t=211}^{215}a_t$ and $p_{13} = \sum_{t=216}^{225}a_t$.
\end{lems}
\begin{proof}[Proof of Proposition \ref{md554}] From Lemma \ref{bdd558}, we have
$$QP_5(\omega_{(5)})^{\Sigma_5} = \langle [p_{12}]_{\omega_{(5)}}, [p_{13}]_{\omega_{(5)}}\rangle.$$ 

Let $f \in P_5(\omega_{(5)})$ such that $[f] \in QP_5(\omega_{(5)})^{GL_5}$. Then, $f \equiv_{\omega_{(5)}} \gamma_1 p_{12} + \gamma_2p_{13}$ with $\gamma_1, \gamma_2 \in \mathbb F_2$. A direct computation shows
\begin{align*}\rho_5(f) + f &\equiv_{\omega_{(5)}} (\gamma_1 + \gamma_2)a_{214}+ \gamma_2(a_{219} + a_{220} + a_{221}) \equiv_{\omega_{(5)}} 0.
\end{align*}
This equality implies $\gamma_1 = \gamma_2 = 0$. The proposition is proved.
\end{proof}

From the above results, we easily obtain the following.
\begin{corls} $\text{\rm Ker}(\widetilde{Sq}^0_*)_{(5,15)}^{GL_5} = 0$.
\end{corls}

\subsubsection{Proof of Theorem \ref{pthm3} for $s=3$}\

\medskip
Let $f \in (P_5)_{35}$ such that $[f] \in (\mathbb F_2\otimes P_5)_{35}^{GL_5}$. Since Kameko's squaring operation 
$$(\widetilde{Sq}^0_*)_{(5,15)}: (\mathbb F_2 \otimes_{\mathcal A} P_5)_{35} \longrightarrow (\mathbb F_2 \otimes_{\mathcal A} P_k)_{15}$$
is an epimorphism of $GL_5$-modules, $(\widetilde{Sq}^0_*)_{(5,15)}([f]) \in  (\mathbb F_2\otimes P_5)_{15}^{GL_5}$. By Proposition~\ref{mdds2}, $(\widetilde{Sq}^0_*)_{(5,15)}([f]) = \lambda_1 [p]  + \lambda_2[q]$  with $\lambda_1, \lambda_2 \in \mathbb F_2$. Hence, we have
$$ f \equiv \lambda_1 \psi(p) + \lambda_2\psi(q) + \bar f,$$
where $\bar f \in (P_5)_{35}$ such that $[\bar f] \in \text{Ker}(\widetilde{Sq}^0_*)_{(5,15)}$, and $\psi : P_5 \to P_5$ is the $\mathbb F_2$-linear map determined by $\psi (y) = x_1x_2x_3x_4x_5y^2$ for any $y \in P_5$.

Now, we prove that if $[f] \ne 0$, then $\lambda_2 = 1$. 

Suppose the contrary, that $\lambda_2 = 0$. By a direct computation, we have
\begin{align}
\rho_1(\psi(p)) + \psi(p) &\equiv  0,\label{ct561}\\ 
\rho_2(\psi(p)) + \psi(p) &\equiv  a_{17} + a_{18} + a_{20} + a_{21} + a_{104} + a_{105} + a_{106} + a_{107} +  a_{108}\notag\\
&\quad  + a_{109} + a_{113} + a_{114} + a_{115} + a_{116} + a_{123} + a_{124},\label{ct562}\\
\rho_3(\psi(p)) + \psi(p) &\equiv a_{11} + a_{18} + a_{35} + a_{37} \notag\\
&\quad  + a_{52} + a_{54} + a_{91} + a_{93} + a_{95} + a_{105},\label{ct563}\\
\rho_4(\psi(p)) + \psi(p) &\equiv a_{1} + a_{4} + a_{11} + a_{12} + a_{15}\notag\\
&\quad   + a_{16} + a_{32} + a_{33} + a_{54} + a_{55} + a_{87}.\label{ct564}
\end{align}

From the relations (\ref{ct561})-(\ref{ct564}), we have $\rho_i(\bar f) + \bar f \equiv_{\omega_{(t)}} \rho_i(f) +  f \equiv_{\omega_{(t)}} 0$, for $i=1,2,3,4,5$ and $t = 4,5$. From this and Propositions \ref{md553} and \ref{md554}, we get $[\bar f]_{\omega_{(t)}} \in QP_5(\omega_{(t)})^{GL_5} = 0$. By combining this and the facts that $QP_5(\omega_{(2)}) =0$ and $QP_5(\omega_{(3)}) = 0$ we get $\bar f \equiv f' \in QP_5(\omega_{(1)})$. Now, by a direct computation from the relations $\rho_i(f) +f \equiv 0$, for $i = 1,2,3,4$, and using Proposition \ref{mdd61}, we get $\lambda_1 = 0$. Hence, by Proposition \ref{md551}, $[f] = [f'] \in QP_5(\omega_{(1)})^{GL_5} = 0$. This contradicts the hypothesis $[f] \ne 0$. Hence, $\lambda_2 = 1$ and $f \equiv \lambda_1\psi(p) + \psi(q) + \bar f$.

Suppose that $f^* = \lambda \psi(p) + \psi(q) + \bar f^*$ with $\lambda \in \mathbb F_2,\ [\bar f^*] \in \text{Ker}(\widetilde{Sq}^0_*)_{(5,15)}$ and $[f^*] \in (\mathbb F_2\otimes P_5)_{35}^{GL_5}$. Then, $(\lambda_1 + \lambda)\psi(p) + \bar f + \bar f^* = f + f^*$ and 
$$[f+f^*] = [f ] + [f^*] \in (\mathbb F_2\otimes P_5)_{35}^{GL_5}.$$ 
By the same argument as the previous one we obtain $\lambda + \lambda_1 = 0$. So, $[f + f^*] \in \text{Ker}(\widetilde{Sq}^0_*)_{(5,15)}^{GL_5} = 0$. This implies $[f] = [f^*]$.

Thus, we have proved that $\dim (\mathbb F_2\otimes P_5)_{35}^{GL_5} \leqslant 1$. 

In \cite{si1}, Singer showed that the Adams elements $h_i$ are in the image of $\varphi_1^*$. H\`a showed in \cite{ha} that the elements $d_i$ are in the image of $\varphi_4^*$. Since $\varphi^* = \{\varphi_k^*: k \geqslant 0\}$ is a homomorphism of graded algebras, we see that the element $h_2d_1$ is in the image of $\varphi_5^*$, hence $\varphi_5((h_2d_1)^*) \ne 0$. This implies $\dim (\mathbb F_2\otimes P_5)_{35}^{GL_5} \geqslant 1$. Theorem \ref{pthm3} is completely proved.

\section{Appendix}\label{s6}
\setcounter{equation}{0}

In the appendix, we list all admissible monomials of degrees 5, 7, 15, 16, 35 in $P_4$ and $P_5$. We order a set of some monomials in $P_k$ by using the order as in Definition \ref{defn3}.

\medskip 
\subsection{The admissible monomials of degree 5 in $\mathbf{P_5}$.}\label{ss61}\

\medskip 
$B_5(5)$ is the set of 46 monomials:

\medskip
 \centerline{\begin{tabular}{llll}
1. \ $x_3x_4x_5^{3} $& 2. \ $x_3x_4^{3}x_5 $& 3. \ $x_3^{3}x_4x_5 $& 4. \ $x_2x_4x_5^{3} $\cr 
5. \ $x_2x_4^{3}x_5 $& 6. \ $x_2x_3x_5^{3} $& 7. \ $x_2x_3x_4^{3} $& 8. \ $x_2x_3^{3}x_5 $\cr 
9. \ $x_2x_3^{3}x_4 $& 10. \ $x_2^{3}x_4x_5 $& 11. \ $x_2^{3}x_3x_5 $& 12. \ $x_2^{3}x_3x_4 $\cr 
13. \ $x_1x_4x_5^{3} $& 14. \ $x_1x_4^{3}x_5 $& 15. \ $x_1x_3x_5^{3} $& 16. \ $x_1x_3x_4^{3} $\cr 
17. \ $x_1x_3^{3}x_5 $& 18. \ $x_1x_3^{3}x_4 $& 19. \ $x_1x_2x_5^{3} $& 20. \ $x_1x_2x_4^{3} $\cr  
21. \ $x_1x_2x_3^{3} $& 22. \ $x_1x_2^{3}x_5 $& 23. \ $x_1x_2^{3}x_4 $& 24. \ $x_1x_2^{3}x_3 $\cr 
25. \ $x_1^{3}x_4x_5 $& 26. \ $x_1^{3}x_3x_5 $& 27. \ $x_1^{3}x_3x_4 $& 28. \ $x_1^{3}x_2x_5 $\cr 
29. \ $x_1^{3}x_2x_4 $& 30. \ $x_1^{3}x_2x_3 $& 31. \ $x_2x_3x_4x_5^{2} $& 32. \ $x_2x_3x_4^{2}x_5 $\cr 
33. \ $x_2x_3^{2}x_4x_5 $& 34. \ $x_1x_3x_4x_5^{2} $& 35. \ $x_1x_3x_4^{2}x_5 $& 36. \ $x_1x_3^{2}x_4x_5 $\cr 
37. \ $x_1x_2x_4x_5^{2} $& 38. \ $x_1x_2x_4^{2}x_5 $& 39. \ $x_1x_2x_3x_5^{2} $& 40. \ $x_1x_2x_3x_4^{2} $\cr  
41. \ $x_1x_2x_3^{2}x_5 $& 42. \ $x_1x_2x_3^{2}x_4 $& 43. \ $x_1x_2^{2}x_4x_5 $& 44. \ $x_1x_2^{2}x_3x_5 $\cr 
45. \ $x_1x_2^{2}x_3x_4 $& 46. \ $x_1x_2x_3x_4x_5 $.&
\end{tabular}}

\medskip
\subsection{The admissible monomials of degree 7 in $\mathbf {P_5}$.}\label{ss62}\

\medskip
\subsubsection{}
$B_4(7)$ is the set of 35 monomials:

\medskip
 \centerline{\begin{tabular}{lllll}
$1.\ x_4^{7}$ &  $2.\ x_3x_4^{6}$ &  $3.\ x_3^{7}$ &  $4.\ x_2x_4^{6}$ &  $5.\ x_2x_3^{2}x_4^{4}$\cr  
$6.\ x_2x_3^{6}$ &  $7.\ x_2^{7}$ &  $8.\ x_1x_4^{6}$ &  $9.\ x_1x_3^{2}x_4^{4}$ &  $10.\ x_1x_3^{6}$\cr  
$11.\ x_1x_2^{2}x_4^{4}$ &  $12.\ x_1x_2^{2}x_3^{4}$ &  $13.\ x_1x_2^{6}$ &  $14.\ x_1^{7}$ &  $15.\ x_1x_2^{2}x_3^{2}x_4^{2}$\cr  
$16.\ x_2x_3^{3}x_4^{3}$ &  $17.\ x_2^{3}x_3x_4^{3}$ &  $18.\ x_2^{3}x_3^{3}x_4$ &  $19.\ x_1x_3^{3}x_4^{3}$ &  $20.\ x_1x_2x_3^{2}x_4^{3}$\cr  
$21.\ x_1x_2x_3^{3}x_4^{2}$ &  $22.\ x_1x_2^{2}x_3x_4^{3}$ &  $23.\ x_1x_2^{2}x_3^{3}x_4$ &  $24.\ x_1x_2^{3}x_4^{3}$ &  $25.\ x_1x_2^{3}x_3x_4^{2}$\cr  
$26.\ x_1x_2^{3}x_3^{2}x_4$ &  $27.\ x_1x_2^{3}x_3^{3}$ &  $28.\ x_1^{3}x_3x_4^{3}$ &  $29.\ x_1^{3}x_3^{3}x_4$ &  $30.\ x_1^{3}x_2x_4^{3}$\cr  
$31.\ x_1^{3}x_2x_3x_4^{2}$ &  $32.\ x_1^{3}x_2x_3^{2}x_4$ &  $33.\ x_1^{3}x_2x_3^{3}$ &  $34.\ x_1^{3}x_2^{3}x_4$ &  $35.\ x_1^{3}x_2^{3}x_3$\cr 
\end{tabular}}

\medskip
\subsubsection{}
$B_5(7) = f(B_4(7))\cup B_5^+(3,2)\cup B_5(5,1)$, where

\medskip
$B_5^+(3,2) = \{x_1x_2x_3x_4^{2}x_5^{2},\  x_1x_2x_3^{2}x_4x_5^{2},\  x_1x_2x_3^{2}x_4^{2}x_5,\  x_1x_2^{2}x_3x_4x_5^{2},\  x_1x_2^{2}x_3x_4^{2}x_5\}$,

$B_5^+(5,1) = \{x_1x_2x_3x_4x_5^{3},\  x_1x_2x_3x_4^{3}x_5,\ x_1x_2x_3^{3}x_4x_5,\ x_1x_2^{3}x_3x_4x_5,\ x_1^{3}x_2x_3x_4x_5\}$.

\medskip
We have $|f(B_4(7))| = 100$. Hence, $\dim (\mathbb F_2\otimes_{\mathcal A}P_5)_7 = 110.$

\bigskip
\subsection{The admissible monomials of degree 15 in $\mathbf {P_4}$.}\label{ss63}\

 \medskip
\subsubsection{}
$B_4(15)$ is the set of  75 monomials:

\medskip
 \centerline{\begin{tabular}{lllll}
$1.\ x_4^{15}$ &  $2.\ x_3x_4^{14}$ &  $3.\ x_3^{15}$ &  $4.\ x_2x_4^{14}$ &  $5.\ x_2x_3^{2}x_4^{12}$\cr  
$6.\ x_2x_3^{14}$ &  $7.\ x_2^{15}$ &  $8.\ x_1x_4^{14}$ &  $9.\ x_1x_3^{2}x_4^{12}$ &  $10.\ x_1x_3^{14}$\cr  
$11.\ x_1x_2^{2}x_4^{12}$ &  $12.\ x_1x_2^{2}x_3^{4}x_4^{8}$ &  $13.\ x_1x_2^{2}x_3^{12}$ &  $14.\ x_1x_2^{14}$ &  $15.\ x_1^{15}$\cr  
$16.\ x_2x_3^{7}x_4^{7}$ &  $17.\ x_2^{3}x_3^{5}x_4^{7}$ &  $18.\ x_2^{3}x_3^{7}x_4^{5}$ &  $19.\ x_2^{7}x_3x_4^{7}$ &  $20.\ x_2^{7}x_3^{3}x_4^{5}$\cr  
$21.\ x_2^{7}x_3^{7}x_4$ &  $22.\ x_1x_3^{7}x_4^{7}$ &  $23.\ x_1x_2x_3^{6}x_4^{7}$ &  $24.\ x_1x_2x_3^{7}x_4^{6}$ &  $25.\ x_1x_2^{2}x_3^{5}x_4^{7}$\cr  
$26.\ x_1x_2^{2}x_3^{7}x_4^{5}$ &  $27.\ x_1x_2^{3}x_3^{4}x_4^{7}$ &  $28.\ x_1x_2^{3}x_3^{5}x_4^{6}$ &  $29.\ x_1x_2^{3}x_3^{6}x_4^{5}$ &  $30.\ x_1x_2^{3}x_3^{7}x_4^{4}$\cr  
$31.\ x_1x_2^{6}x_3x_4^{7}$ &  $32.\ x_1x_2^{6}x_3^{3}x_4^{5}$ &  $33.\ x_1x_2^{6}x_3^{7}x_4$ &  $34.\ x_1x_2^{7}x_4^{7}$ &  $35.\ x_1x_2^{7}x_3x_4^{6}$\cr  
$36.\ x_1x_2^{7}x_3^{2}x_4^{5}$ &  $37.\ x_1x_2^{7}x_3^{3}x_4^{4}$ &  $38.\ x_1x_2^{7}x_3^{6}x_4$ &  $39.\ x_1x_2^{7}x_3^{7}$ &  $40.\ x_1^{3}x_3^{5}x_4^{7}$\cr  
$41.\ x_1^{3}x_3^{7}x_4^{5}$ &  $42.\ x_1^{3}x_2x_3^{4}x_4^{7}$ &  $43.\ x_1^{3}x_2x_3^{5}x_4^{6}$ &  $44.\ x_1^{3}x_2x_3^{6}x_4^{5}$ &  $45.\ x_1^{3}x_2x_3^{7}x_4^{4}$\cr  
$46.\ x_1^{3}x_2^{3}x_3^{4}x_4^{5}$ &  $47.\ x_1^{3}x_2^{3}x_3^{5}x_4^{4}$ &  $48.\ x_1^{3}x_2^{4}x_3x_4^{7}$ &  $49.\ x_1^{3}x_2^{4}x_3^{3}x_4^{5}$ &  $50.\ x_1^{3}x_2^{4}x_3^{7}x_4$\cr  
$51.\ x_1^{3}x_2^{5}x_4^{7}$ &  $52.\ x_1^{3}x_2^{5}x_3x_4^{6}$ &  $53.\ x_1^{3}x_2^{5}x_3^{2}x_4^{5}$ &  $54.\ x_1^{3}x_2^{5}x_3^{3}x_4^{4}$ &  $55.\ x_1^{3}x_2^{5}x_3^{6}x_4$\cr  
$56.\ x_1^{3}x_2^{5}x_3^{7}$ &  $57.\ x_1^{3}x_2^{7}x_4^{5}$ &  $58.\ x_1^{3}x_2^{7}x_3x_4^{4}$ &  $59.\ x_1^{3}x_2^{7}x_3^{4}x_4$ &  $60.\ x_1^{3}x_2^{7}x_3^{5}$\cr  
$61.\ x_1^{7}x_3x_4^{7}$ &  $62.\ x_1^{7}x_3^{3}x_4^{5}$ &  $63.\ x_1^{7}x_3^{7}x_4$ &  $64.\ x_1^{7}x_2x_4^{7}$ &  $65.\ x_1^{7}x_2x_3x_4^{6}$\cr  
$66.\ x_1^{7}x_2x_3^{2}x_4^{5}$ &  $67.\ x_1^{7}x_2x_3^{3}x_4^{4}$ &  $68.\ x_1^{7}x_2x_3^{6}x_4$ &  $69.\ x_1^{7}x_2x_3^{7}$ &  $70.\ x_1^{7}x_2^{3}x_4^{5}$\cr  
$71.\ x_1^{7}x_2^{3}x_3x_4^{4}$ &  $72.\ x_1^{7}x_2^{3}x_3^{4}x_4$ &  $73.\ x_1^{7}x_2^{3}x_3^{5}$ &  $74.\ x_1^{7}x_2^{7}x_4$ &  $75.\ x_1^{7}x_2^{7}x_3$\cr   
\end{tabular}}

\medskip
$B_5(15) = f(B_4(15))\cup B_5(1,1,3)\cup B_5^+(3,2,2)\cup B_5(3,4,1)\cup \psi(B_5(5)$, where

\medskip
\subsubsection{}
$B_5(1,1,3) = \{x_1x_2^2x_3^4x_4^4x_5^4\}$;

\medskip
\subsubsection{}
$B_5^+(3,2,2)$ is the set of 75  monomials:

\medskip
 \centerline{\begin{tabular}{llll}
$1.\ x_1x_2x_3x_4^{6}x_5^{6}$ &  $2.\ x_1x_2x_3^{2}x_4^{4}x_5^{7}$ &  $3.\ x_1x_2x_3^{2}x_4^{5}x_5^{6}$ &  $4.\ x_1x_2x_3^{2}x_4^{6}x_5^{5}$\cr  
$5.\ x_1x_2x_3^{2}x_4^{7}x_5^{4}$ &  $6.\ x_1x_2x_3^{3}x_4^{4}x_5^{6}$ &  $7.\ x_1x_2x_3^{3}x_4^{6}x_5^{4}$ &  $8.\ x_1x_2x_3^{6}x_4x_5^{6}$\cr  
$9.\ x_1x_2x_3^{6}x_4^{2}x_5^{5}$ &  $10.\ x_1x_2x_3^{6}x_4^{3}x_5^{4}$ &  $11.\ x_1x_2x_3^{6}x_4^{6}x_5$ &  $12.\ x_1x_2x_3^{7}x_4^{2}x_5^{4}$\cr  
$13.\ x_1x_2^{2}x_3x_4^{4}x_5^{7}$ &  $14.\ x_1x_2^{2}x_3x_4^{5}x_5^{6}$ &  $15.\ x_1x_2^{2}x_3x_4^{6}x_5^{5}$ &  $16.\ x_1x_2^{2}x_3x_4^{7}x_5^{4}$\cr  
$17.\ x_1x_2^{2}x_3^{3}x_4^{4}x_5^{5}$ &  $18.\ x_1x_2^{2}x_3^{3}x_4^{5}x_5^{4}$ &  $19.\ x_1x_2^{2}x_3^{4}x_4x_5^{7}$ &  $20.\ x_1x_2^{2}x_3^{4}x_4^{3}x_5^{5}$\cr  
$21.\ x_1x_2^{2}x_3^{4}x_4^{7}x_5$ &  $22.\ x_1x_2^{2}x_3^{5}x_4x_5^{6}$ &  $23.\ x_1x_2^{2}x_3^{5}x_4^{2}x_5^{5}$ &  $24.\ x_1x_2^{2}x_3^{5}x_4^{3}x_5^{4}$\cr  
$25.\ x_1x_2^{2}x_3^{5}x_4^{6}x_5$ &  $26.\ x_1x_2^{2}x_3^{7}x_4x_5^{4}$ &  $27.\ x_1x_2^{2}x_3^{7}x_4^{4}x_5$ &  $28.\ x_1x_2^{3}x_3x_4^{4}x_5^{6}$\cr  
$29.\ x_1x_2^{3}x_3x_4^{6}x_5^{4}$ &  $30.\ x_1x_2^{3}x_3^{2}x_4^{4}x_5^{5}$ &  $31.\ x_1x_2^{3}x_3^{2}x_4^{5}x_5^{4}$ &  $32.\ x_1x_2^{3}x_3^{3}x_4^{4}x_5^{4}$\cr  
$33.\ x_1x_2^{3}x_3^{4}x_4x_5^{6}$ &  $34.\ x_1x_2^{3}x_3^{4}x_4^{2}x_5^{5}$ &  $35.\ x_1x_2^{3}x_3^{4}x_4^{3}x_5^{4}$ &  $36.\ x_1x_2^{3}x_3^{4}x_4^{6}x_5$\cr  
$37.\ x_1x_2^{3}x_3^{5}x_4^{2}x_5^{4}$ &  $38.\ x_1x_2^{3}x_3^{6}x_4x_5^{4}$ &  $39.\ x_1x_2^{3}x_3^{6}x_4^{4}x_5$ &  $40.\ x_1x_2^{6}x_3x_4x_5^{6}$\cr  
$41.\ x_1x_2^{6}x_3x_4^{2}x_5^{5}$ &  $42.\ x_1x_2^{6}x_3x_4^{3}x_5^{4}$ &  $43.\ x_1x_2^{6}x_3x_4^{6}x_5$ &  $44.\ x_1x_2^{6}x_3^{3}x_4x_5^{4}$\cr  
$45.\ x_1x_2^{6}x_3^{3}x_4^{4}x_5$ &  $46.\ x_1x_2^{7}x_3x_4^{2}x_5^{4}$ &  $47.\ x_1x_2^{7}x_3^{2}x_4x_5^{4}$ &  $48.\ x_1x_2^{7}x_3^{2}x_4^{4}x_5$\cr  
$49.\ x_1^{3}x_2x_3x_4^{4}x_5^{6}$ &  $50.\ x_1^{3}x_2x_3x_4^{6}x_5^{4}$ &  $51.\ x_1^{3}x_2x_3^{2}x_4^{4}x_5^{5}$ &  $52.\ x_1^{3}x_2x_3^{2}x_4^{5}x_5^{4}$\cr  
$53.\ x_1^{3}x_2x_3^{3}x_4^{4}x_5^{4}$ &  $54.\ x_1^{3}x_2x_3^{4}x_4x_5^{6}$ &  $55.\ x_1^{3}x_2x_3^{4}x_4^{2}x_5^{5}$ &  $56.\ x_1^{3}x_2x_3^{4}x_4^{3}x_5^{4}$\cr  
$57.\ x_1^{3}x_2x_3^{4}x_4^{6}x_5$ &  $58.\ x_1^{3}x_2x_3^{5}x_4^{2}x_5^{4}$ &  $59.\ x_1^{3}x_2x_3^{6}x_4x_5^{4}$ &  $60.\ x_1^{3}x_2x_3^{6}x_4^{4}x_5$\cr  
$61.\ x_1^{3}x_2^{3}x_3x_4^{4}x_5^{4}$ &  $62.\ x_1^{3}x_2^{3}x_3^{4}x_4x_5^{4}$ &  $63.\ x_1^{3}x_2^{3}x_3^{4}x_4^{4}x_5$ &  $64.\ x_1^{3}x_2^{4}x_3x_4x_5^{6}$\cr  
$65.\ x_1^{3}x_2^{4}x_3x_4^{2}x_5^{5}$ &  $66.\ x_1^{3}x_2^{4}x_3x_4^{3}x_5^{4}$ &  $67.\ x_1^{3}x_2^{4}x_3x_4^{6}x_5$ &  $68.\ x_1^{3}x_2^{4}x_3^{3}x_4x_5^{4}$\cr  
$69.\ x_1^{3}x_2^{4}x_3^{3}x_4^{4}x_5$ &  $70.\ x_1^{3}x_2^{5}x_3x_4^{2}x_5^{4}$ &  $71.\ x_1^{3}x_2^{5}x_3^{2}x_4x_5^{4}$ &  $72.\ x_1^{3}x_2^{5}x_3^{2}x_4^{4}x_5$\cr  
$73.\ x_1^{7}x_2x_3x_4^{2}x_5^{4}$ &  $74.\ x_1^{7}x_2x_3^{2}x_4x_5^{4}$ &  $75.\ x_1^{7}x_2x_3^{2}x_4^{4}x_5$ &  
\end{tabular}}

\medskip
\subsubsection{}
$B_5(3,4,1)$ is the set of 40 monomials:

\medskip
 \centerline{\begin{tabular}{llll}
$1.\ x_1x_2^{2}x_3^{2}x_4^{3}x_5^{7}$ &  $2.\ x_1x_2^{2}x_3^{2}x_4^{7}x_5^{3}$ &  $3.\ x_1x_2^{2}x_3^{3}x_4^{2}x_5^{7}$ &  $4.\ x_1x_2^{2}x_3^{3}x_4^{3}x_5^{6}$\cr  
$5.\ x_1x_2^{2}x_3^{3}x_4^{6}x_5^{3}$ &  $6.\ x_1x_2^{2}x_3^{3}x_4^{7}x_5^{2}$ &  $7.\ x_1x_2^{2}x_3^{7}x_4^{2}x_5^{3}$ &  $8.\ x_1x_2^{2}x_3^{7}x_4^{3}x_5^{2}$\cr  
$9.\ x_1x_2^{3}x_3^{2}x_4^{2}x_5^{7}$ &  $10.\ x_1x_2^{3}x_3^{2}x_4^{3}x_5^{6}$ &  $11.\ x_1x_2^{3}x_3^{2}x_4^{6}x_5^{3}$ &  $12.\ x_1x_2^{3}x_3^{2}x_4^{7}x_5^{2}$\cr  
$13.\ x_1x_2^{3}x_3^{3}x_4^{2}x_5^{6}$ &  $14.\ x_1x_2^{3}x_3^{3}x_4^{6}x_5^{2}$ &  $15.\ x_1x_2^{3}x_3^{6}x_4^{2}x_5^{3}$ &  $16.\ x_1x_2^{3}x_3^{6}x_4^{3}x_5^{2}$\cr  
$17.\ x_1x_2^{3}x_3^{7}x_4^{2}x_5^{2}$ &  $18.\ x_1x_2^{7}x_3^{2}x_4^{2}x_5^{3}$ &  $19.\ x_1x_2^{7}x_3^{2}x_4^{3}x_5^{2}$ &  $20.\ x_1x_2^{7}x_3^{3}x_4^{2}x_5^{2}$\cr  
$21.\ x_1^{3}x_2x_3^{2}x_4^{2}x_5^{7}$ &  $22.\ x_1^{3}x_2x_3^{2}x_4^{3}x_5^{6}$ &  $23.\ x_1^{3}x_2x_3^{2}x_4^{6}x_5^{3}$ &  $24.\ x_1^{3}x_2x_3^{2}x_4^{7}x_5^{2}$\cr  
$25.\ x_1^{3}x_2x_3^{3}x_4^{2}x_5^{6}$ &  $26.\ x_1^{3}x_2x_3^{3}x_4^{6}x_5^{2}$ &  $27.\ x_1^{3}x_2x_3^{6}x_4^{2}x_5^{3}$ &  $28.\ x_1^{3}x_2x_3^{6}x_4^{3}x_5^{2}$\cr  
$29.\ x_1^{3}x_2x_3^{7}x_4^{2}x_5^{2}$ &  $30.\ x_1^{3}x_2^{3}x_3x_4^{2}x_5^{6}$ &  $31.\ x_1^{3}x_2^{3}x_3x_4^{6}x_5^{2}$ &  $32.\ x_1^{3}x_2^{3}x_3^{5}x_4^{2}x_5^{2}$\cr  
$33.\ x_1^{3}x_2^{5}x_3^{2}x_4^{2}x_5^{3}$ &  $34.\ x_1^{3}x_2^{5}x_3^{2}x_4^{3}x_5^{2}$ &  $35.\ x_1^{3}x_2^{5}x_3^{3}x_4^{2}x_5^{2}$ &  $36.\ x_1^{3}x_2^{7}x_3x_4^{2}x_5^{2}$\cr  
$37.\ x_1^{7}x_2x_3^{2}x_4^{2}x_5^{3}$ &  $38.\ x_1^{7}x_2x_3^{2}x_4^{3}x_5^{2}$ &  $39.\ x_1^{7}x_2x_3^{3}x_4^{2}x_5^{2}$ &  $40.\ x_1^{7}x_2^{3}x_3x_4^{2}x_5^{2}$\cr
\end{tabular}}

\medskip
We have $|f(B_4(15)| = 270$. Hence, $\dim (\mathbb F_2\otimes_{\mathcal A}P_5)_{15} = 432.$

\medskip

\subsection{The admissible monomials of degree 16 in $\mathbf {P_5}$.}\label{ss64}\

\subsubsection{} 
$B_4(16)$ is the set of 73 monomials:

\medskip
 \centerline{\begin{tabular}{lllll}
$1.\ \   x_3x_4^{15}$& $2.\ \   x_3^{3}x_4^{13}$& $3.\ \   x_3^{15}x_4$& $4.\ \   x_2x_4^{15}$& $5.\ \   x_2x_3x_4^{14}$\cr 
$6.\ \   x_2x_3^{2}x_4^{13}$& $7.\ \   x_2x_3^{3}x_4^{12}$& $8.\ \   x_2x_3^{14}x_4$& $9.\ \   x_2x_3^{15}$& $10.\ \   x_2^{3}x_4^{13}$\cr 
$11.\ \   x_2^{3}x_3x_4^{12}$& $12.\ \   x_2^{3}x_3^{13}$& $13.\ \   x_2^{15}x_4$& $14.\ \   x_2^{15}x_3$& $15.\ \   x_1x_4^{15}$\cr 
$16.\ \   x_1x_3x_4^{14}$& $17.\ \   x_1x_3^{2}x_4^{13}$& $18.\ \   x_1x_3^{3}x_4^{12}$& $19.\ \   x_1x_3^{14}x_4$& $20.\ \   x_1x_3^{15}$\cr 
$21.\ \   x_1x_2x_4^{14}$& $22.\ \   x_1x_2x_3^{2}x_4^{12}$& $23.\ \   x_1x_2x_3^{14}$& $24.\ \   x_1x_2^{2}x_4^{13}$& $25.\ \   x_1x_2^{2}x_3x_4^{12}$\cr 
$26.\ \   x_1x_2^{2}x_3^{4}x_4^{9}$& $27.\ \   x_1x_2^{2}x_3^{5}x_4^{8}$& $28.\ \   x_1x_2^{2}x_3^{12}x_4$& $29.\ \   x_1x_2^{2}x_3^{13}$& $30.\ \   x_1x_2^{3}x_4^{12}$\cr 
$31.\ \   x_1x_2^{3}x_3^{4}x_4^{8}$& $32.\ \   x_1x_2^{3}x_3^{12}$& $33.\ \   x_1x_2^{14}x_4$& $34.\ \   x_1x_2^{14}x_3$& $35.\ \   x_1x_2^{15}$\cr 
$36.\ \   x_1^{3}x_4^{13}$& $37.\ \   x_1^{3}x_3x_4^{12}$& $38.\ \   x_1^{3}x_3^{13}$& $39.\ \   x_1^{3}x_2x_4^{12}$& $40.\ \   x_1^{3}x_2x_3^{4}x_4^{8}$\cr 
$41.\ \   x_1^{3}x_2x_3^{12}$& $42.\ \   x_1^{3}x_2^{13}$& $43.\ \   x_1^{15}x_4$& $44.\ \   x_1^{15}x_3$& $45.\ \   x_1^{15}x_2$\cr 
\end{tabular}}
\centerline{\begin{tabular}{lllll}
$46.\ \   x_1x_2^{3}x_3^{6}x_4^{6}$& $47.\ \   x_1^{3}x_2x_3^{6}x_4^{6}$& $48.\ \   x_1^{3}x_2^{5}x_3^{2}x_4^{6}$& $49.\ \   x_1^{3}x_2^{5}x_3^{6}x_4^{2}$& $50.\ \   x_1x_2x_3^{7}x_4^{7}$\cr 
$51.\ \   x_1x_2^{3}x_3^{5}x_4^{7}$& $52.\ \   x_1x_2^{3}x_3^{7}x_4^{5}$& $53.\ \   x_1x_2^{7}x_3x_4^{7}$& $54.\ \   x_1x_2^{7}x_3^{3}x_4^{5}$& $55.\ \   x_1x_2^{7}x_3^{7}x_4$\cr 
$56.\ \   x_1^{3}x_2x_3^{5}x_4^{7}$& $57.\ \   x_1^{3}x_2x_3^{7}x_4^{5}$& $58.\ \   x_1^{3}x_2^{3}x_3^{5}x_4^{5}$& $59.\ \   x_1^{3}x_2^{5}x_3x_4^{7}$& $60.\ \   x_1^{3}x_2^{5}x_3^{3}x_4^{5}$\cr 
$61.\ \   x_1^{3}x_2^{5}x_3^{7}x_4$& $62.\ \   x_1^{3}x_2^{7}x_3x_4^{5}$& $63.\ \   x_1^{3}x_2^{7}x_3^{5}x_4$& $64.\ \   x_1^{7}x_2x_3x_4^{7}$& $65.\ \   x_1^{7}x_2x_3^{3}x_4^{5}$\cr 
$66.\ \   x_1^{7}x_2x_3^{7}x_4$& $67.\ \   x_1^{7}x_2^{3}x_3x_4^{5}$& $68.\ \   x_1^{7}x_2^{3}x_3^{5}x_4$& $69.\ \   x_1^{7}x_2^{7}x_3x_4$& $70.\ \   x_1^{3}x_2^{3}x_3^{3}x_4^{7}$\cr 
$71.\ \   x_1^{3}x_2^{3}x_3^{7}x_4^{3}$& $72.\ \   x_1^{3}x_2^{7}x_3^{3}x_4^{3}$& $73.\ \   x_1^{7}x_2^{3}x_3^{3}x_4^{3}.$& 
\end{tabular}}

\bigskip 
$B_5(16) = B_5^0(16) \cup B_5^+(16)$, where $B_5^0(156) = f(B_4(16))$, $|B_5^0(16)| =255$ and 
$$B_5^+(16) = B_5^+(2,1,1,1)\cup B_5^+(2,1,3) \cup B_5^+(2,3,2) \cup B_5^+(4,2,2) \cup B_5^+(4,4,1).$$ 

\subsubsection{}
$B_5^+(2,1,1,1)$ is the set of 4 monomials $a_t = a_{16,t}, \ 1 \leqslant t \leqslant 4$:

\medskip
 \centerline{\begin{tabular}{llll}
$1.\ \   x_1x_2x_3^{2}x_4^{4}x_5^{8}$& $2.\ \   x_1x_2^{2}x_3x_4^{4}x_5^{8}$& $3.\ \   x_1x_2^{2}x_3^{4}x_4x_5^{8}$& $4.\ \   x_1x_2^{2}x_3^{4}x_4^{8}x_5$.\cr 
\end{tabular}}

\subsubsection{}
$B_5^+(2,1,3)$ is the set of 5 monomials $a_t = a_{16,t}, \ 5 \leqslant t \leqslant 9$:

\medskip
 \centerline{\begin{tabular}{llll}
$5.\ \   x_1x_2^{2}x_3^{4}x_4^{4}x_5^{5}$& $6.\ \   x_1x_2^{2}x_3^{4}x_4^{5}x_5^{4}$& $7.\ \   x_1x_2^{2}x_3^{5}x_4^{4}x_5^{4}$& $8.\ \   x_1x_2^{3}x_3^{4}x_4^{4}x_5^{4}$\cr 
$9.\ \   x_1^{3}x_2x_3^{4}x_4^{4}x_5^{4}$.& 
\end{tabular}}

\medskip
\subsubsection{}
$B_5^+(2,3,2)$ is the set of 20 monomials $a_t = a_{16,t}, \ 10 \leqslant t \leqslant 29$:

\medskip
 \centerline{\begin{tabular}{llll}
$10.\ \   x_1x_2x_3^{2}x_4^{6}x_5^{6}$& $11.\ \   x_1x_2x_3^{6}x_4^{2}x_5^{6}$& $12.\ \   x_1x_2x_3^{6}x_4^{6}x_5^{2}$& $13.\ \   x_1x_2^{2}x_3x_4^{6}x_5^{6}$\cr 
$14.\ \   x_1x_2^{2}x_3^{5}x_4^{2}x_5^{6}$& $15.\ \   x_1x_2^{2}x_3^{5}x_4^{6}x_5^{2}$& $16.\ \   x_1x_2^{3}x_3^{2}x_4^{4}x_5^{6}$& $17.\ \   x_1x_2^{3}x_3^{2}x_4^{6}x_5^{4}$\cr 
$18.\ \   x_1x_2^{3}x_3^{4}x_4^{2}x_5^{6}$& $19.\ \   x_1x_2^{3}x_3^{4}x_4^{6}x_5^{2}$& $20.\ \   x_1x_2^{3}x_3^{6}x_4^{2}x_5^{4}$& $21.\ \   x_1x_2^{3}x_3^{6}x_4^{4}x_5^{2}$\cr 
$22.\ \   x_1^{3}x_2x_3^{2}x_4^{4}x_5^{6}$& $23.\ \   x_1^{3}x_2x_3^{2}x_4^{6}x_5^{4}$& $24.\ \   x_1^{3}x_2x_3^{4}x_4^{2}x_5^{6}$& $25.\ \   x_1^{3}x_2x_3^{4}x_4^{6}x_5^{2}$\cr 
$26.\ \   x_1^{3}x_2x_3^{6}x_4^{2}x_5^{4}$& $27.\ \   x_1^{3}x_2x_3^{6}x_4^{4}x_5^{2}$& $28.\ \   x_1^{3}x_2^{5}x_3^{2}x_4^{2}x_5^{4}$& $29.\ \   x_1^{3}x_2^{5}x_3^{2}x_4^{4}x_5^{2}$.\cr 
\end{tabular}}

\medskip
\subsubsection{}
$B_5^+(4,2,2)$ is the set of 110 monomials $a_t = a_{16,t}, \ 30 \leqslant t \leqslant 139$:

\medskip
 \centerline{\begin{tabular}{llll}
$30.\ \   x_1x_2x_3x_4^{6}x_5^{7}$& $31.\ \   x_1x_2x_3x_4^{7}x_5^{6}$& $32.\ \   x_1x_2x_3^{6}x_4x_5^{7}$& $33.\ \   x_1x_2x_3^{6}x_4^{7}x_5$\cr 
$34.\ \   x_1x_2x_3^{7}x_4x_5^{6}$& $35.\ \   x_1x_2x_3^{7}x_4^{6}x_5$& $36.\ \   x_1x_2^{6}x_3x_4x_5^{7}$& $37.\ \   x_1x_2^{6}x_3x_4^{7}x_5$\cr 
$38.\ \   x_1x_2^{6}x_3^{7}x_4x_5$& $39.\ \   x_1x_2^{7}x_3x_4x_5^{6}$& $40.\ \   x_1x_2^{7}x_3x_4^{6}x_5$& $41.\ \   x_1x_2^{7}x_3^{6}x_4x_5$\cr 
$42.\ \   x_1^{7}x_2x_3x_4x_5^{6}$& $43.\ \   x_1^{7}x_2x_3x_4^{6}x_5$& $44.\ \   x_1^{7}x_2x_3^{6}x_4x_5$& $45.\ \   x_1x_2x_3^{2}x_4^{5}x_5^{7}$\cr 
$46.\ \   x_1x_2x_3^{2}x_4^{7}x_5^{5}$& $47.\ \   x_1x_2x_3^{7}x_4^{2}x_5^{5}$& $48.\ \   x_1x_2^{2}x_3x_4^{5}x_5^{7}$& $49.\ \   x_1x_2^{2}x_3x_4^{7}x_5^{5}$\cr 
$50.\ \   x_1x_2^{2}x_3^{5}x_4x_5^{7}$& $51.\ \   x_1x_2^{2}x_3^{5}x_4^{7}x_5$& $52.\ \   x_1x_2^{2}x_3^{7}x_4x_5^{5}$& $53.\ \   x_1x_2^{2}x_3^{7}x_4^{5}x_5$\cr 
$54.\ \   x_1x_2^{7}x_3x_4^{2}x_5^{5}$& $55.\ \   x_1x_2^{7}x_3^{2}x_4x_5^{5}$& $56.\ \   x_1x_2^{7}x_3^{2}x_4^{5}x_5$& $57.\ \   x_1^{7}x_2x_3x_4^{2}x_5^{5}$\cr 
$58.\ \   x_1^{7}x_2x_3^{2}x_4x_5^{5}$& $59.\ \   x_1^{7}x_2x_3^{2}x_4^{5}x_5$& $60.\ \   x_1x_2x_3^{3}x_4^{4}x_5^{7}$& $61.\ \   x_1x_2x_3^{3}x_4^{7}x_5^{4}$\cr 
$62.\ \   x_1x_2x_3^{7}x_4^{3}x_5^{4}$& $63.\ \   x_1x_2^{3}x_3x_4^{4}x_5^{7}$& $64.\ \   x_1x_2^{3}x_3x_4^{7}x_5^{4}$& $65.\ \   x_1x_2^{3}x_3^{4}x_4x_5^{7}$\cr 
$66.\ \   x_1x_2^{3}x_3^{4}x_4^{7}x_5$& $67.\ \   x_1x_2^{3}x_3^{7}x_4x_5^{4}$& $68.\ \   x_1x_2^{3}x_3^{7}x_4^{4}x_5$& $69.\ \   x_1x_2^{7}x_3x_4^{3}x_5^{4}$\cr 
$70.\ \   x_1x_2^{7}x_3^{3}x_4x_5^{4}$& $71.\ \   x_1x_2^{7}x_3^{3}x_4^{4}x_5$& $72.\ \   x_1^{3}x_2x_3x_4^{4}x_5^{7}$& $73.\ \   x_1^{3}x_2x_3x_4^{7}x_5^{4}$\cr 
$74.\ \   x_1^{3}x_2x_3^{4}x_4x_5^{7}$& $75.\ \   x_1^{3}x_2x_3^{4}x_4^{7}x_5$& $76.\ \   x_1^{3}x_2x_3^{7}x_4x_5^{4}$& $77.\ \   x_1^{3}x_2x_3^{7}x_4^{4}x_5$\cr 
$78.\ \   x_1^{3}x_2^{4}x_3x_4x_5^{7}$& $79.\ \   x_1^{3}x_2^{4}x_3x_4^{7}x_5$& $80.\ \   x_1^{3}x_2^{4}x_3^{7}x_4x_5$& $81.\ \   x_1^{3}x_2^{7}x_3x_4x_5^{4}$\cr 
$82.\ \   x_1^{3}x_2^{7}x_3x_4^{4}x_5$& $83.\ \   x_1^{3}x_2^{7}x_3^{4}x_4x_5$& $84.\ \   x_1^{7}x_2x_3x_4^{3}x_5^{4}$& $85.\ \   x_1^{7}x_2x_3^{3}x_4x_5^{4}$\cr 
$86.\ \   x_1^{7}x_2x_3^{3}x_4^{4}x_5$& $87.\ \   x_1^{7}x_2^{3}x_3x_4x_5^{4}$& $88.\ \   x_1^{7}x_2^{3}x_3x_4^{4}x_5$& $89.\ \   x_1^{7}x_2^{3}x_3^{4}x_4x_5$\cr 
$90.\ \   x_1x_2x_3^{3}x_4^{5}x_5^{6}$& $91.\ \   x_1x_2x_3^{3}x_4^{6}x_5^{5}$& $92.\ \   x_1x_2x_3^{6}x_4^{3}x_5^{5}$& $93.\ \   x_1x_2^{3}x_3x_4^{5}x_5^{6}$\cr 
$94.\ \   x_1x_2^{3}x_3x_4^{6}x_5^{5}$& $95.\ \   x_1x_2^{3}x_3^{5}x_4x_5^{6}$& $96.\ \   x_1x_2^{3}x_3^{5}x_4^{6}x_5$& $97.\ \   x_1x_2^{3}x_3^{6}x_4x_5^{5}$\cr 
$98.\ \   x_1x_2^{3}x_3^{6}x_4^{5}x_5$& $99.\ \   x_1x_2^{6}x_3x_4^{3}x_5^{5}$& $100.\ \   x_1x_2^{6}x_3^{3}x_4x_5^{5}$& $101.\ \   x_1x_2^{6}x_3^{3}x_4^{5}x_5$\cr 
$102.\ \   x_1^{3}x_2x_3x_4^{5}x_5^{6}$& $103.\ \   x_1^{3}x_2x_3x_4^{6}x_5^{5}$& $104.\ \   x_1^{3}x_2x_3^{5}x_4x_5^{6}$& $105.\ \   x_1^{3}x_2x_3^{5}x_4^{6}x_5$\cr 
$106.\ \   x_1^{3}x_2x_3^{6}x_4x_5^{5}$& $107.\ \   x_1^{3}x_2x_3^{6}x_4^{5}x_5$& $108.\ \   x_1^{3}x_2^{5}x_3x_4x_5^{6}$& $109.\ \   x_1^{3}x_2^{5}x_3x_4^{6}x_5$\cr 
\end{tabular}}
\centerline{\begin{tabular}{llll}
$110.\ \   x_1^{3}x_2^{5}x_3^{6}x_4x_5$& $111.\ \   x_1x_2^{2}x_3^{3}x_4^{5}x_5^{5}$& $112.\ \   x_1x_2^{2}x_3^{5}x_4^{3}x_5^{5}$& $113.\ \   x_1x_2^{3}x_3^{2}x_4^{5}x_5^{5}$\cr 
$114.\ \   x_1x_2^{3}x_3^{5}x_4^{2}x_5^{5}$& $115.\ \   x_1^{3}x_2x_3^{2}x_4^{5}x_5^{5}$& $116.\ \   x_1^{3}x_2x_3^{5}x_4^{2}x_5^{5}$& $117.\ \   x_1^{3}x_2^{5}x_3x_4^{2}x_5^{5}$\cr 
$118.\ \   x_1^{3}x_2^{5}x_3^{2}x_4x_5^{5}$& $119.\ \   x_1^{3}x_2^{5}x_3^{2}x_4^{5}x_5$& $120.\ \   x_1x_2^{3}x_3^{3}x_4^{4}x_5^{5}$& $121.\ \   x_1x_2^{3}x_3^{3}x_4^{5}x_5^{4}$\cr 
$122.\ \   x_1x_2^{3}x_3^{4}x_4^{3}x_5^{5}$& $123.\ \   x_1x_2^{3}x_3^{5}x_4^{3}x_5^{4}$& $124.\ \   x_1^{3}x_2x_3^{3}x_4^{4}x_5^{5}$& $125.\ \   x_1^{3}x_2x_3^{3}x_4^{5}x_5^{4}$\cr 
$126.\ \   x_1^{3}x_2x_3^{4}x_4^{3}x_5^{5}$& $127.\ \   x_1^{3}x_2x_3^{5}x_4^{3}x_5^{4}$& $128.\ \   x_1^{3}x_2^{3}x_3x_4^{4}x_5^{5}$& $129.\ \   x_1^{3}x_2^{3}x_3x_4^{5}x_5^{4}$\cr 
$130.\ \   x_1^{3}x_2^{3}x_3^{4}x_4x_5^{5}$& $131.\ \   x_1^{3}x_2^{3}x_3^{4}x_4^{5}x_5$& $132.\ \   x_1^{3}x_2^{3}x_3^{5}x_4x_5^{4}$& $133.\ \   x_1^{3}x_2^{3}x_3^{5}x_4^{4}x_5$\cr 
$134.\ \   x_1^{3}x_2^{4}x_3x_4^{3}x_5^{5}$& $135.\ \   x_1^{3}x_2^{4}x_3^{3}x_4x_5^{5}$& $136.\ \   x_1^{3}x_2^{4}x_3^{3}x_4^{5}x_5$& $137.\ \   x_1^{3}x_2^{5}x_3x_4^{3}x_5^{4}$\cr 
$138.\ \   x_1^{3}x_2^{5}x_3^{3}x_4x_5^{4}$& $139.\ \   x_1^{3}x_2^{5}x_3^{3}x_4^{4}x_5$.& 
\end{tabular}}

\medskip
\subsubsection{}
$B_5^+(4,4,1)$ is the set of 49 monomials $a_t = a_{16,t}, \ 140 \leqslant t \leqslant 188$:

\medskip
 \centerline{\begin{tabular}{llll}
$140.\ \   x_1x_2^{2}x_3^{3}x_4^{3}x_5^{7}$& $141.\ \   x_1x_2^{2}x_3^{3}x_4^{7}x_5^{3}$& $142.\ \   x_1x_2^{2}x_3^{7}x_4^{3}x_5^{3}$& $143.\ \   x_1x_2^{3}x_3^{2}x_4^{3}x_5^{7}$\cr 
$144.\ \   x_1x_2^{3}x_3^{2}x_4^{7}x_5^{3}$& $145.\ \   x_1x_2^{3}x_3^{3}x_4^{2}x_5^{7}$& $146.\ \   x_1x_2^{3}x_3^{3}x_4^{7}x_5^{2}$& $147.\ \   x_1x_2^{3}x_3^{7}x_4^{2}x_5^{3}$\cr 
$148.\ \   x_1x_2^{3}x_3^{7}x_4^{3}x_5^{2}$& $149.\ \   x_1x_2^{7}x_3^{2}x_4^{3}x_5^{3}$& $150.\ \   x_1x_2^{7}x_3^{3}x_4^{2}x_5^{3}$& $151.\ \   x_1x_2^{7}x_3^{3}x_4^{3}x_5^{2}$\cr 
$152.\ \   x_1^{3}x_2x_3^{2}x_4^{3}x_5^{7}$& $153.\ \   x_1^{3}x_2x_3^{2}x_4^{7}x_5^{3}$& $154.\ \   x_1^{3}x_2x_3^{3}x_4^{2}x_5^{7}$& $155.\ \   x_1^{3}x_2x_3^{3}x_4^{7}x_5^{2}$\cr 
$156.\ \   x_1^{3}x_2x_3^{7}x_4^{2}x_5^{3}$& $157.\ \   x_1^{3}x_2x_3^{7}x_4^{3}x_5^{2}$& $158.\ \   x_1^{3}x_2^{3}x_3x_4^{2}x_5^{7}$& $159.\ \   x_1^{3}x_2^{3}x_3x_4^{7}x_5^{2}$\cr 
$160.\ \   x_1^{3}x_2^{3}x_3^{7}x_4x_5^{2}$& $161.\ \   x_1^{3}x_2^{7}x_3x_4^{2}x_5^{3}$& $162.\ \   x_1^{3}x_2^{7}x_3x_4^{3}x_5^{2}$& $163.\ \   x_1^{3}x_2^{7}x_3^{3}x_4x_5^{2}$\cr 
$164.\ \   x_1^{7}x_2x_3^{2}x_4^{3}x_5^{3}$& $165.\ \   x_1^{7}x_2x_3^{3}x_4^{2}x_5^{3}$& $166.\ \   x_1^{7}x_2x_3^{3}x_4^{3}x_5^{2}$& $167.\ \   x_1^{7}x_2^{3}x_3x_4^{2}x_5^{3}$\cr 
$168.\ \   x_1^{7}x_2^{3}x_3x_4^{3}x_5^{2}$& $169.\ \   x_1^{7}x_2^{3}x_3^{3}x_4x_5^{2}$& $170.\ \   x_1x_2^{3}x_3^{3}x_4^{3}x_5^{6}$& $171.\ \   x_1x_2^{3}x_3^{3}x_4^{6}x_5^{3}$\cr 
$172.\ \   x_1x_2^{3}x_3^{6}x_4^{3}x_5^{3}$& $173.\ \   x_1x_2^{6}x_3^{3}x_4^{3}x_5^{3}$& $174.\ \   x_1^{3}x_2x_3^{3}x_4^{3}x_5^{6}$& $175.\ \   x_1^{3}x_2x_3^{3}x_4^{6}x_5^{3}$\cr 
\end{tabular}}
\centerline{\begin{tabular}{lllll}
$176.\ \   x_1^{3}x_2x_3^{6}x_4^{3}x_5^{3}$& $177.\ \   x_1^{3}x_2^{3}x_3x_4^{3}x_5^{6}$& $178.\ \   x_1^{3}x_2^{3}x_3x_4^{6}x_5^{3}$& $179.\ \   x_1^{3}x_2^{3}x_3^{3}x_4x_5^{6}$\cr 
$180.\ \   x_1^{3}x_2^{3}x_3^{3}x_4^{5}x_5^{2}$& $181.\ \   x_1^{3}x_2^{3}x_3^{5}x_4^{2}x_5^{3}$& $182.\ \   x_1^{3}x_2^{3}x_3^{5}x_4^{3}x_5^{2}$& $183.\ \   x_1^{3}x_2^{5}x_3^{2}x_4^{3}x_5^{3}$\cr 
$184.\ \   x_1^{3}x_2^{5}x_3^{3}x_4^{2}x_5^{3}$& $185.\ \   x_1^{3}x_2^{5}x_3^{3}x_4^{3}x_5^{2}$& $186.\ \   x_1^{3}x_2^{3}x_3^{3}x_4^{3}x_5^{4}$& $187.\ \   x_1^{3}x_2^{3}x_3^{3}x_4^{4}x_5^{3}$\cr 
$188.\ \   x_1^{3}x_2^{3}x_3^{4}x_4^{3}x_5^{3}$.& 
\end{tabular}}

\subsection{The admissible monomials of degree 35 in $\mathbf {P_5^0}$.}\label{ss65}

\medskip 
\subsubsection{}
$B_5^0(35)$ is the set of 460 monomials, $b_t = b_{35,t}, \ 1 \leqslant t \leqslant 460$, determined as follows:

\medskip
 \centerline{\begin{tabular}{llll}
1. \ $x_3x_4^{3}x_5^{31} $& 2. \ $x_3x_4^{31}x_5^{3} $& 3. \ $x_3^{3}x_4x_5^{31} $& 4. \ $x_3^{3}x_4^{31}x_5 $\cr  
5. \ $x_3^{31}x_4x_5^{3} $& 6. \ $x_3^{31}x_4^{3}x_5 $& 7. \ $x_2x_4^{3}x_5^{31} $& 8. \ $x_2x_4^{31}x_5^{3} $\cr  
9. \ $x_2x_3^{3}x_5^{31} $& 10. \ $x_2x_3^{3}x_4^{31} $& 11. \ $x_2x_3^{31}x_5^{3} $& 12. \ $x_2x_3^{31}x_4^{3} $\cr  
13. \ $x_2^{3}x_4x_5^{31} $& 14. \ $x_2^{3}x_4^{31}x_5 $& 15. \ $x_2^{3}x_3x_5^{31} $& 16. \ $x_2^{3}x_3x_4^{31} $\cr  
17. \ $x_2^{3}x_3^{31}x_5 $& 18. \ $x_2^{3}x_3^{31}x_4 $& 19. \ $x_2^{31}x_4x_5^{3} $& 20. \ $x_2^{31}x_4^{3}x_5 $\cr  
21. \ $x_2^{31}x_3x_5^{3} $& 22. \ $x_2^{31}x_3x_4^{3} $& 23. \ $x_2^{31}x_3^{3}x_5 $& 24. \ $x_2^{31}x_3^{3}x_4 $\cr  
25. \ $x_1x_4^{3}x_5^{31} $& 26. \ $x_1x_4^{31}x_5^{3} $& 27. \ $x_1x_3^{3}x_5^{31} $& 28. \ $x_1x_3^{3}x_4^{31} $\cr  
29. \ $x_1x_3^{31}x_5^{3} $& 30. \ $x_1x_3^{31}x_4^{3} $& 31. \ $x_1x_2^{3}x_5^{31} $& 32. \ $x_1x_2^{3}x_4^{31} $\cr  
33. \ $x_1x_2^{3}x_3^{31} $& 34. \ $x_1x_2^{31}x_5^{3} $& 35. \ $x_1x_2^{31}x_4^{3} $& 36. \ $x_1x_2^{31}x_3^{3} $\cr  
37. \ $x_1^{3}x_4x_5^{31} $& 38. \ $x_1^{3}x_4^{31}x_5 $& 39. \ $x_1^{3}x_3x_5^{31} $& 40. \ $x_1^{3}x_3x_4^{31} $\cr  
41. \ $x_1^{3}x_3^{31}x_5 $& 42. \ $x_1^{3}x_3^{31}x_4 $& 43. \ $x_1^{3}x_2x_5^{31} $& 44. \ $x_1^{3}x_2x_4^{31} $\cr  
45. \ $x_1^{3}x_2x_3^{31} $& 46. \ $x_1^{3}x_2^{31}x_5 $& 47. \ $x_1^{3}x_2^{31}x_4 $& 48. \ $x_1^{3}x_2^{31}x_3 $\cr  
49. \ $x_1^{31}x_4x_5^{3} $& 50. \ $x_1^{31}x_4^{3}x_5 $& 51. \ $x_1^{31}x_3x_5^{3} $& 52. \ $x_1^{31}x_3x_4^{3} $\cr  
53. \ $x_1^{31}x_3^{3}x_5 $& 54. \ $x_1^{31}x_3^{3}x_4 $& 55. \ $x_1^{31}x_2x_5^{3} $& 56. \ $x_1^{31}x_2x_4^{3} $\cr  
57. \ $x_1^{31}x_2x_3^{3} $& 58. \ $x_1^{31}x_2^{3}x_5 $& 59. \ $x_1^{31}x_2^{3}x_4 $& 60. \ $x_1^{31}x_2^{3}x_3 $\cr  
61. \ $x_3x_4^{7}x_5^{27} $& 62. \ $x_3^{7}x_4x_5^{27} $& 63. \ $x_3^{7}x_4^{27}x_5 $& 64. \ $x_2x_4^{7}x_5^{27} $\cr  
65. \ $x_2x_3^{7}x_5^{27} $& 66. \ $x_2x_3^{7}x_4^{27} $& 67. \ $x_2^{7}x_4x_5^{27} $& 68. \ $x_2^{7}x_4^{27}x_5 $\cr  
69. \ $x_2^{7}x_3x_5^{27} $& 70. \ $x_2^{7}x_3x_4^{27} $& 71. \ $x_2^{7}x_3^{27}x_5 $& 72. \ $x_2^{7}x_3^{27}x_4 $\cr  
73. \ $x_1x_4^{7}x_5^{27} $& 74. \ $x_1x_3^{7}x_5^{27} $& 75. \ $x_1x_3^{7}x_4^{27} $& 76. \ $x_1x_2^{7}x_5^{27} $\cr  
77. \ $x_1x_2^{7}x_4^{27} $& 78. \ $x_1x_2^{7}x_3^{27} $& 79. \ $x_1^{7}x_4x_5^{27} $& 80. \ $x_1^{7}x_4^{27}x_5 $\cr  
81. \ $x_1^{7}x_3x_5^{27} $& 82. \ $x_1^{7}x_3x_4^{27} $& 83. \ $x_1^{7}x_3^{27}x_5 $& 84. \ $x_1^{7}x_3^{27}x_4 $\cr  
85. \ $x_1^{7}x_2x_5^{27} $& 86. \ $x_1^{7}x_2x_4^{27} $& 87. \ $x_1^{7}x_2x_3^{27} $& 88. \ $x_1^{7}x_2^{27}x_5 $\cr  
89. \ $x_1^{7}x_2^{27}x_4 $& 90. \ $x_1^{7}x_2^{27}x_3 $& $91.\  x_3^{3}x_4^{29}x_5^{3}$ &  $92.\  x_2^{3}x_4^{29}x_5^{3}$\cr  
$93.\  x_2^{3}x_3^{29}x_5^{3}$ &  $94.\  x_2^{3}x_3^{29}x_4^{3}$ &  $95.\  x_1^{3}x_4^{29}x_5^{3}$ &  $96.\  x_1^{3}x_3^{29}x_5^{3}$\cr  
$97.\  x_1^{3}x_3^{29}x_4^{3}$ &  $98.\  x_1^{3}x_2^{29}x_5^{3}$ &  $99.\  x_1^{3}x_2^{29}x_4^{3}$ &  $100.\  x_1^{3}x_2^{29}x_3^{3}$\cr  
$101.\  x_3^{3}x_4^{5}x_5^{27}$ &  $102.\  x_2^{3}x_4^{5}x_5^{27}$ &  $103.\  x_2^{3}x_3^{5}x_5^{27}$ &  $104.\  x_2^{3}x_3^{5}x_4^{27}$\cr  
$105.\  x_1^{3}x_4^{5}x_5^{27}$ &  $106.\  x_1^{3}x_3^{5}x_5^{27}$ &  $107.\  x_1^{3}x_3^{5}x_4^{27}$ &  $108.\  x_1^{3}x_2^{5}x_5^{27}$\cr  
$109.\  x_1^{3}x_2^{5}x_4^{27}$ &  $110.\  x_1^{3}x_2^{5}x_3^{27}$ &  $111.\  x_3^{3}x_4^{7}x_5^{25}$ &  $112.\  x_3^{7}x_4^{3}x_5^{25}$\cr  
$113.\  x_2^{3}x_4^{7}x_5^{25}$ &  $114.\  x_2^{3}x_3^{7}x_5^{25}$ &  $115.\  x_2^{3}x_3^{7}x_4^{25}$ &  $116.\  x_2^{7}x_4^{3}x_5^{25}$\cr  
$117.\  x_2^{7}x_3^{3}x_5^{25}$ &  $118.\  x_2^{7}x_3^{3}x_4^{25}$ &  $119.\  x_1^{3}x_4^{7}x_5^{25}$ &  $120.\  x_1^{3}x_3^{7}x_5^{25}$\cr  
$121.\  x_1^{3}x_3^{7}x_4^{25}$ &  $122.\  x_1^{3}x_2^{7}x_5^{25}$ &  $123.\  x_1^{3}x_2^{7}x_4^{25}$ &  $124.\  x_1^{3}x_2^{7}x_3^{25}$\cr  
$125.\  x_1^{7}x_4^{3}x_5^{25}$ &  $126.\  x_1^{7}x_3^{3}x_5^{25}$ &  $127.\  x_1^{7}x_3^{3}x_4^{25}$ &  $128.\  x_1^{7}x_2^{3}x_5^{25}$\cr  
$129.\  x_1^{7}x_2^{3}x_4^{25}$ &  $130.\  x_1^{7}x_2^{3}x_3^{25}$ &  $131.\  x_3^{3}x_4^{3}x_5^{29}$ &  $132.\  x_2^{3}x_4^{3}x_5^{29}$\cr  
$133.\  x_2^{3}x_3^{3}x_5^{29}$ &  $134.\  x_2^{3}x_3^{3}x_4^{29}$ &  $135.\  x_1^{3}x_4^{3}x_5^{29}$ &  $136.\  x_1^{3}x_3^{3}x_5^{29}$\cr  
$137.\  x_1^{3}x_3^{3}x_4^{29}$ &  $138.\  x_1^{3}x_2^{3}x_5^{29}$ &  $139.\  x_1^{3}x_2^{3}x_4^{29}$ &  $140.\  x_1^{3}x_2^{3}x_3^{29}$\cr 
141. \ $x_2x_3x_4^{2}x_5^{31} $& 142. \ $x_2x_3x_4^{31}x_5^{2} $& 143. \ $x_2x_3^{2}x_4x_5^{31} $& 144. \ $x_2x_3^{2}x_4^{31}x_5 $\cr  
145. \ $x_2x_3^{31}x_4x_5^{2} $& 146. \ $x_2x_3^{31}x_4^{2}x_5 $& 147. \ $x_2^{31}x_3x_4x_5^{2} $& 148. \ $x_2^{31}x_3x_4^{2}x_5 $\cr  
149. \ $x_1x_3x_4^{2}x_5^{31} $& 150. \ $x_1x_3x_4^{31}x_5^{2} $& 151. \ $x_1x_3^{2}x_4x_5^{31} $& 152. \ $x_1x_3^{2}x_4^{31}x_5 $\cr  
153. \ $x_1x_3^{31}x_4x_5^{2} $& 154. \ $x_1x_3^{31}x_4^{2}x_5 $& 155. \ $x_1x_2x_4^{2}x_5^{31} $& 156. \ $x_1x_2x_4^{31}x_5^{2} $\cr  
157. \ $x_1x_2x_3^{2}x_5^{31} $& 158. \ $x_1x_2x_3^{2}x_4^{31} $& 159. \ $x_1x_2x_3^{31}x_5^{2} $& 160. \ $x_1x_2x_3^{31}x_4^{2} $\cr  
\end{tabular}}
\centerline{\begin{tabular}{llll}
161. \ $x_1x_2^{2}x_4x_5^{31} $& 162. \ $x_1x_2^{2}x_4^{31}x_5 $& 163. \ $x_1x_2^{2}x_3x_5^{31} $& 164. \ $x_1x_2^{2}x_3x_4^{31} $\cr  
165. \ $x_1x_2^{2}x_3^{31}x_5 $& 166. \ $x_1x_2^{2}x_3^{31}x_4 $& 167. \ $x_1x_2^{31}x_4x_5^{2} $& 168. \ $x_1x_2^{31}x_4^{2}x_5 $\cr  
169. \ $x_1x_2^{31}x_3x_5^{2} $& 170. \ $x_1x_2^{31}x_3x_4^{2} $& 171. \ $x_1x_2^{31}x_3^{2}x_5 $& 172. \ $x_1x_2^{31}x_3^{2}x_4 $\cr  
173. \ $x_1^{31}x_3x_4x_5^{2} $& 174. \ $x_1^{31}x_3x_4^{2}x_5 $& 175. \ $x_1^{31}x_2x_4x_5^{2} $& 176. \ $x_1^{31}x_2x_4^{2}x_5 $\cr  
177. \ $x_1^{31}x_2x_3x_5^{2} $& 178. \ $x_1^{31}x_2x_3x_4^{2} $& 179. \ $x_1^{31}x_2x_3^{2}x_5 $& 180. \ $x_1^{31}x_2x_3^{2}x_4 $\cr  
$181.\  x_2x_3x_4^{3}x_5^{30}$ &  $182.\  x_2^{3}x_3x_4x_5^{30}$ &  $183.\  x_1x_3x_4^{3}x_5^{30}$ &  $184.\  x_1x_2x_4^{3}x_5^{30}$\cr  
$185.\  x_1x_2x_3^{3}x_5^{30}$ &  $186.\  x_1x_2x_3^{3}x_4^{30}$ &  $187.\  x_1^{3}x_3x_4x_5^{30}$ &  $188.\  x_1^{3}x_2x_4x_5^{30}$\cr  
$189.\  x_1^{3}x_2x_3x_5^{30}$ &  $190.\  x_1^{3}x_2x_3x_4^{30}$ &  $191.\  x_2^{3}x_3^{29}x_4x_5^{2}$ &  $192.\  x_2^{3}x_3^{29}x_4^{2}x_5$\cr  
$193.\  x_1^{3}x_3^{29}x_4x_5^{2}$ &  $194.\  x_1^{3}x_3^{29}x_4^{2}x_5$ &  $195.\  x_1^{3}x_2^{29}x_4x_5^{2}$ &  $196.\  x_1^{3}x_2^{29}x_4^{2}x_5$\cr  
$197.\  x_1^{3}x_2^{29}x_3x_5^{2}$ &  $198.\  x_1^{3}x_2^{29}x_3x_4^{2}$ &  $199.\  x_1^{3}x_2^{29}x_3^{2}x_5$ &  $200.\  x_1^{3}x_2^{29}x_3^{2}x_4$\cr  
$201.\  x_2^{3}x_3x_4^{3}x_5^{28}$ &  $202.\  x_1^{3}x_3x_4^{3}x_5^{28}$ &  $203.\  x_1^{3}x_2x_4^{3}x_5^{28}$ &  $204.\  x_1^{3}x_2x_3^{3}x_5^{28}$\cr  
$205.\  x_1^{3}x_2x_3^{3}x_4^{28}$ &  $206.\  x_2x_3^{3}x_4^{5}x_5^{26}$ &  $207.\  x_2^{3}x_3x_4^{5}x_5^{26}$ &  $208.\  x_2^{3}x_3^{5}x_4x_5^{26}$\cr  
$209.\  x_2^{3}x_3^{5}x_4^{26}x_5$ &  $210.\  x_1x_3^{3}x_4^{5}x_5^{26}$ &  $211.\  x_1x_2^{3}x_4^{5}x_5^{26}$ &  $212.\  x_1x_2^{3}x_3^{5}x_5^{26}$\cr  
$213.\  x_1x_2^{3}x_3^{5}x_4^{26}$ &  $214.\  x_1^{3}x_3x_4^{5}x_5^{26}$ &  $215.\  x_1^{3}x_3^{5}x_4x_5^{26}$ &  $216.\  x_1^{3}x_3^{5}x_4^{26}x_5$\cr  
$217.\  x_1^{3}x_2x_4^{5}x_5^{26}$ &  $218.\  x_1^{3}x_2x_3^{5}x_5^{26}$ &  $219.\  x_1^{3}x_2x_3^{5}x_4^{26}$ &  $220.\  x_1^{3}x_2^{5}x_4x_5^{26}$\cr  
$221.\  x_1^{3}x_2^{5}x_4^{26}x_5$ &  $222.\  x_1^{3}x_2^{5}x_3x_5^{26}$ &  $223.\  x_1^{3}x_2^{5}x_3x_4^{26}$ &  $224.\  x_1^{3}x_2^{5}x_3^{26}x_5$\cr  
$225.\  x_1^{3}x_2^{5}x_3^{26}x_4$ &  $226.\  x_2x_3^{3}x_4^{6}x_5^{25}$ &  $227.\  x_2x_3^{6}x_4^{3}x_5^{25}$ &  $228.\  x_1x_3^{3}x_4^{6}x_5^{25}$\cr  
$229.\  x_1x_3^{6}x_4^{3}x_5^{25}$ &  $230.\  x_1x_2^{3}x_4^{6}x_5^{25}$ &  $231.\  x_1x_2^{3}x_3^{6}x_5^{25}$ &  $232.\  x_1x_2^{3}x_3^{6}x_4^{25}$\cr  
$233.\  x_1x_2^{6}x_4^{3}x_5^{25}$ &  $234.\  x_1x_2^{6}x_3^{3}x_5^{25}$ &  $235.\  x_1x_2^{6}x_3^{3}x_4^{25}$ &  $236.\  x_2^{3}x_3^{3}x_4^{4}x_5^{25}$\cr  
$237.\  x_2^{3}x_3^{4}x_4^{3}x_5^{25}$ &  $238.\  x_1^{3}x_3^{3}x_4^{4}x_5^{25}$ &  $239.\  x_1^{3}x_3^{4}x_4^{3}x_5^{25}$ &  $240.\  x_1^{3}x_2^{3}x_4^{4}x_5^{25}$\cr  
$241.\  x_1^{3}x_2^{3}x_3^{4}x_5^{25}$ &  $242.\  x_1^{3}x_2^{3}x_3^{4}x_4^{25}$ &  $243.\  x_1^{3}x_2^{4}x_4^{3}x_5^{25}$ &  $244.\  x_1^{3}x_2^{4}x_3^{3}x_5^{25}$\cr  
$245.\  x_1^{3}x_2^{4}x_3^{3}x_4^{25}$ &  $246.\  x_2^{3}x_3^{3}x_4^{5}x_5^{24}$ &  $247.\  x_1^{3}x_3^{3}x_4^{5}x_5^{24}$ &  $248.\  x_1^{3}x_2^{3}x_4^{5}x_5^{24}$\cr  
$249.\  x_1^{3}x_2^{3}x_3^{5}x_5^{24}$ &  $250.\  x_1^{3}x_2^{3}x_3^{5}x_4^{24}$ &  $251.\  x_2x_3^{3}x_4x_5^{30}$ &  $252.\  x_1x_3^{3}x_4x_5^{30}$\cr  
$253.\  x_1x_2^{3}x_4x_5^{30}$ &  $254.\  x_1x_2^{3}x_3x_5^{30}$ &  $255.\  x_1x_2^{3}x_3x_4^{30}$ &  $256.\  x_2^{3}x_3^{3}x_4x_5^{28}$\cr  
$257.\  x_2^{3}x_3^{3}x_4^{28}x_5$ &  $258.\  x_1^{3}x_3^{3}x_4x_5^{28}$ &  $259.\  x_1^{3}x_3^{3}x_4^{28}x_5$ &  $260.\  x_1^{3}x_2^{3}x_4x_5^{28}$\cr  
$261.\  x_1^{3}x_2^{3}x_4^{28}x_5$ &  $262.\  x_1^{3}x_2^{3}x_3x_5^{28}$ &  $263.\  x_1^{3}x_2^{3}x_3x_4^{28}$ &  $264.\  x_1^{3}x_2^{3}x_3^{28}x_5$\cr  
$265.\  x_1^{3}x_2^{3}x_3^{28}x_4$ &  $266.\  x_2x_3^{3}x_4^{30}x_5$ &  $267.\  x_1x_3^{3}x_4^{30}x_5$ &  $268.\  x_1x_2^{3}x_4^{30}x_5$\cr  
$269.\  x_1x_2^{3}x_3^{30}x_5$ &  $270.\  x_1x_2^{3}x_3^{30}x_4$ &  $271.\  x_2x_3^{30}x_4x_5^{3}$ &  $272.\  x_2x_3^{30}x_4^{3}x_5$\cr  
$273.\  x_1x_3^{30}x_4x_5^{3}$ &  $274.\  x_1x_3^{30}x_4^{3}x_5$ &  $275.\  x_1x_2^{30}x_4x_5^{3}$ &  $276.\  x_1x_2^{30}x_4^{3}x_5$\cr  
$277.\  x_1x_2^{30}x_3x_5^{3}$ &  $278.\  x_1x_2^{30}x_3x_4^{3}$ &  $279.\  x_1x_2^{30}x_3^{3}x_5$ &  $280.\  x_1x_2^{30}x_3^{3}x_4$\cr  
$281.\  x_2x_3^{7}x_4^{26}x_5$ &  $282.\  x_2^{7}x_3x_4^{26}x_5$ &  $283.\  x_1x_3^{7}x_4^{26}x_5$ &  $284.\  x_1x_2^{7}x_4^{26}x_5$\cr  
$285.\  x_1x_2^{7}x_3^{26}x_5$ &  $286.\  x_1x_2^{7}x_3^{26}x_4$ &  $287.\  x_1^{7}x_3x_4^{26}x_5$ &  $288.\  x_1^{7}x_2x_4^{26}x_5$\cr  
$289.\  x_1^{7}x_2x_3^{26}x_5$ &  $290.\  x_1^{7}x_2x_3^{26}x_4$ &  $291.\  x_2x_3^{2}x_4^{3}x_5^{29}$ &  $292.\  x_2x_3^{2}x_4^{29}x_5^{3}$\cr  
$293.\  x_2x_3^{3}x_4^{2}x_5^{29}$ &  $294.\  x_2^{3}x_3x_4^{2}x_5^{29}$ &  $295.\  x_1x_3^{2}x_4^{3}x_5^{29}$ &  $296.\  x_1x_3^{2}x_4^{29}x_5^{3}$\cr  
$297.\  x_1x_3^{3}x_4^{2}x_5^{29}$ &  $298.\  x_1x_2^{2}x_4^{3}x_5^{29}$ &  $299.\  x_1x_2^{2}x_4^{29}x_5^{3}$ &  $300.\  x_1x_2^{2}x_3^{3}x_5^{29}$\cr  
$301.\  x_1x_2^{2}x_3^{3}x_4^{29}$ &  $302.\  x_1x_2^{2}x_3^{29}x_5^{3}$ &  $303.\  x_1x_2^{2}x_3^{29}x_4^{3}$ &  $304.\  x_1x_2^{3}x_4^{2}x_5^{29}$\cr  
$305.\  x_1x_2^{3}x_3^{2}x_5^{29}$ &  $306.\  x_1x_2^{3}x_3^{2}x_4^{29}$ &  $307.\  x_1^{3}x_3x_4^{2}x_5^{29}$ &  $308.\  x_1^{3}x_2x_4^{2}x_5^{29}$\cr  
$309.\  x_1^{3}x_2x_3^{2}x_5^{29}$ &  $310.\  x_1^{3}x_2x_3^{2}x_4^{29}$ &  $311.\  x_2x_3^{2}x_4^{5}x_5^{27}$ &  $312.\  x_1x_3^{2}x_4^{5}x_5^{27}$\cr  
$313.\  x_1x_2^{2}x_4^{5}x_5^{27}$ &  $314.\  x_1x_2^{2}x_3^{5}x_5^{27}$ &  $315.\  x_1x_2^{2}x_3^{5}x_4^{27}$ &  $316.\  x_2x_3^{2}x_4^{7}x_5^{25}$\cr  
$317.\  x_2x_3^{7}x_4^{2}x_5^{25}$ &  $318.\  x_2^{7}x_3x_4^{2}x_5^{25}$ &  $319.\  x_1x_3^{2}x_4^{7}x_5^{25}$ &  $320.\  x_1x_3^{7}x_4^{2}x_5^{25}$\cr  
$321.\  x_1x_2^{2}x_4^{7}x_5^{25}$ &  $322.\  x_1x_2^{2}x_3^{7}x_5^{25}$ &  $323.\  x_1x_2^{2}x_3^{7}x_4^{25}$ &  $324.\  x_1x_2^{7}x_4^{2}x_5^{25}$\cr  
$325.\  x_1x_2^{7}x_3^{2}x_5^{25}$ &  $326.\  x_1x_2^{7}x_3^{2}x_4^{25}$ &  $327.\  x_1^{7}x_3x_4^{2}x_5^{25}$ &  $328.\  x_1^{7}x_2x_4^{2}x_5^{25}$\cr  
$329.\  x_1^{7}x_2x_3^{2}x_5^{25}$ &  $330.\  x_1^{7}x_2x_3^{2}x_4^{25}$ &  $331.\  x_2x_3^{3}x_4^{28}x_5^{3}$ &  $332.\  x_2^{3}x_3x_4^{28}x_5^{3}$\cr  
$333.\  x_1x_3^{3}x_4^{28}x_5^{3}$ &  $334.\  x_1x_2^{3}x_4^{28}x_5^{3}$ &  $335.\  x_1x_2^{3}x_3^{28}x_5^{3}$ &  $336.\  x_1x_2^{3}x_3^{28}x_4^{3}$\cr  
$337.\  x_1^{3}x_3x_4^{28}x_5^{3}$ &  $338.\  x_1^{3}x_2x_4^{28}x_5^{3}$ &  $339.\  x_1^{3}x_2x_3^{28}x_5^{3}$ &  $340.\  x_1^{3}x_2x_3^{28}x_4^{3}$\cr  
$341.\  x_2x_3^{3}x_4^{4}x_5^{27}$ &  $342.\  x_1x_3^{3}x_4^{4}x_5^{27}$ &  $343.\  x_1x_2^{3}x_4^{4}x_5^{27}$ &  $344.\  x_1x_2^{3}x_3^{4}x_5^{27}$\cr  
$345.\  x_1x_2^{3}x_3^{4}x_4^{27}$ &  $346.\  x_2x_3^{3}x_4^{7}x_5^{24}$ &  $347.\  x_2x_3^{7}x_4^{3}x_5^{24}$ &  $348.\  x_2^{3}x_3^{7}x_4x_5^{24}$\cr  
$349.\  x_2^{7}x_3x_4^{3}x_5^{24}$ &  $350.\  x_2^{7}x_3^{3}x_4x_5^{24}$ &  $351.\  x_1x_3^{3}x_4^{7}x_5^{24}$ &  $352.\  x_1x_3^{7}x_4^{3}x_5^{24}$\cr  
\end{tabular}}
\centerline{\begin{tabular}{llll}
$353.\  x_1x_2^{3}x_4^{7}x_5^{24}$ &  $354.\  x_1x_2^{3}x_3^{7}x_5^{24}$ &  $355.\  x_1x_2^{3}x_3^{7}x_4^{24}$ &  $356.\  x_1x_2^{7}x_4^{3}x_5^{24}$\cr  
$357.\  x_1x_2^{7}x_3^{3}x_5^{24}$ &  $358.\  x_1x_2^{7}x_3^{3}x_4^{24}$ &  $359.\  x_1^{3}x_3^{7}x_4x_5^{24}$ &  $360.\  x_1^{3}x_2^{7}x_4x_5^{24}$\cr  
$361.\  x_1^{3}x_2^{7}x_3x_5^{24}$ &  $362.\  x_1^{3}x_2^{7}x_3x_4^{24}$ &  $363.\  x_1^{7}x_3x_4^{3}x_5^{24}$ &  $364.\  x_1^{7}x_3^{3}x_4x_5^{24}$\cr  
$365.\  x_1^{7}x_2x_4^{3}x_5^{24}$ &  $366.\  x_1^{7}x_2x_3^{3}x_5^{24}$ &  $367.\  x_1^{7}x_2x_3^{3}x_4^{24}$ &  $368.\  x_1^{7}x_2^{3}x_4x_5^{24}$\cr  
$369.\  x_1^{7}x_2^{3}x_3x_5^{24}$ &  $370.\  x_1^{7}x_2^{3}x_3x_4^{24}$ &  $371.\  x_2^{3}x_3x_4^{30}x_5$ &  $372.\  x_1^{3}x_3x_4^{30}x_5$\cr  
$373.\  x_1^{3}x_2x_4^{30}x_5$ &  $374.\  x_1^{3}x_2x_3^{30}x_5$ &  $375.\  x_1^{3}x_2x_3^{30}x_4$ &  $376.\  x_2x_3x_4^{6}x_5^{27}$\cr  
$377.\  x_1x_3x_4^{6}x_5^{27}$ &  $378.\  x_1x_2x_4^{6}x_5^{27}$ &  $379.\  x_1x_2x_3^{6}x_5^{27}$ &  $380.\  x_1x_2x_3^{6}x_4^{27}$\cr  
$381.\  x_2x_3x_4^{7}x_5^{26}$ &  $382.\  x_1x_3x_4^{7}x_5^{26}$ &  $383.\  x_1x_2x_4^{7}x_5^{26}$ &  $384.\  x_1x_2x_3^{7}x_5^{26}$\cr  
$385.\  x_1x_2x_3^{7}x_4^{26}$ &  $386.\  x_2^{3}x_3^{4}x_4x_5^{27}$ &  $387.\  x_2^{3}x_3^{4}x_4^{27}x_5$ &  $388.\  x_1^{3}x_3^{4}x_4x_5^{27}$\cr  
$389.\  x_1^{3}x_3^{4}x_4^{27}x_5$ &  $390.\  x_1^{3}x_2^{4}x_4x_5^{27}$ &  $391.\  x_1^{3}x_2^{4}x_4^{27}x_5$ &  $392.\  x_1^{3}x_2^{4}x_3x_5^{27}$\cr  
$393.\  x_1^{3}x_2^{4}x_3x_4^{27}$ &  $394.\  x_1^{3}x_2^{4}x_3^{27}x_5$ &  $395.\  x_1^{3}x_2^{4}x_3^{27}x_4$ &  $396.\  x_2x_3^{6}x_4x_5^{27}$\cr  
$397.\  x_2x_3^{6}x_4^{27}x_5$ &  $398.\  x_1x_3^{6}x_4x_5^{27}$ &  $399.\  x_1x_3^{6}x_4^{27}x_5$ &  $400.\  x_1x_2^{6}x_4x_5^{27}$\cr  
$401.\  x_1x_2^{6}x_4^{27}x_5$ &  $402.\  x_1x_2^{6}x_3x_5^{27}$ &  $403.\  x_1x_2^{6}x_3x_4^{27}$ &  $404.\  x_1x_2^{6}x_3^{27}x_5$\cr  
$405.\  x_1x_2^{6}x_3^{27}x_4$ &  $406.\  x_2^{3}x_3x_4^{4}x_5^{27}$ &  $407.\  x_1^{3}x_3x_4^{4}x_5^{27}$ &  $408.\  x_1^{3}x_2x_4^{4}x_5^{27}$\cr  
$409.\  x_1^{3}x_2x_3^{4}x_5^{27}$ &  $410.\  x_1^{3}x_2x_3^{4}x_4^{27}$ &  $411.\  x_2^{3}x_3x_4^{7}x_5^{24}$ &  $412.\  x_1^{3}x_3x_4^{7}x_5^{24}$\cr  
$413.\  x_1^{3}x_2x_4^{7}x_5^{24}$ &  $414.\  x_1^{3}x_2x_3^{7}x_5^{24}$ &  $415.\  x_1^{3}x_2x_3^{7}x_4^{24}$ &  $416.\  x_2x_3x_4^{30}x_5^{3}$\cr  
$417.\  x_1x_3x_4^{30}x_5^{3}$ &  $418.\  x_1x_2x_4^{30}x_5^{3}$ &  $419.\  x_1x_2x_3^{30}x_5^{3}$ &  $420.\  x_1x_2x_3^{30}x_4^{3}$\cr  
$421.\  x_2x_3^{7}x_4x_5^{26}$ &  $422.\  x_2^{7}x_3x_4x_5^{26}$ &  $423.\  x_1x_3^{7}x_4x_5^{26}$ &  $424.\  x_1x_2^{7}x_4x_5^{26}$\cr  
$425.\  x_1x_2^{7}x_3x_5^{26}$ &  $426.\  x_1x_2^{7}x_3x_4^{26}$ &  $427.\  x_1^{7}x_3x_4x_5^{26}$ &  $428.\  x_1^{7}x_2x_4x_5^{26}$\cr  
$429.\  x_1^{7}x_2x_3x_5^{26}$ &  $430.\  x_1^{7}x_2x_3x_4^{26}$ &  $431.\  x_2x_3^{3}x_4^{29}x_5^{2}$ &  $432.\  x_2^{3}x_3x_4^{29}x_5^{2}$\cr  
$433.\  x_1x_3^{3}x_4^{29}x_5^{2}$ &  $434.\  x_1x_2^{3}x_4^{29}x_5^{2}$ &  $435.\  x_1x_2^{3}x_3^{29}x_5^{2}$ &  $436.\  x_1x_2^{3}x_3^{29}x_4^{2}$\cr  
$437.\  x_1^{3}x_3x_4^{29}x_5^{2}$ &  $438.\  x_1^{3}x_2x_4^{29}x_5^{2}$ &  $439.\  x_1^{3}x_2x_3^{29}x_5^{2}$ &  $440.\  x_1^{3}x_2x_3^{29}x_4^{2}$\cr  
$441.\  x_2x_3^{3}x_4^{3}x_5^{28}$ &  $442.\  x_1x_3^{3}x_4^{3}x_5^{28}$ &  $443.\  x_1x_2^{3}x_4^{3}x_5^{28}$ &  $444.\  x_1x_2^{3}x_3^{3}x_5^{28}$\cr  
$445.\  x_1x_2^{3}x_3^{3}x_4^{28}$ &  $446.\  x_2^{3}x_3x_4^{6}x_5^{25}$ &  $447.\  x_1^{3}x_3x_4^{6}x_5^{25}$ &  $448.\  x_1^{3}x_2x_4^{6}x_5^{25}$\cr  
$449.\  x_1^{3}x_2x_3^{6}x_5^{25}$ &  $450.\  x_1^{3}x_2x_3^{6}x_4^{25}$ &  $451.\  x_2^{3}x_3^{5}x_4^{2}x_5^{25}$ &  $452.\  x_1^{3}x_3^{5}x_4^{2}x_5^{25}$\cr  
$453.\  x_1^{3}x_2^{5}x_4^{2}x_5^{25}$ &  $454.\  x_1^{3}x_2^{5}x_3^{2}x_5^{25}$ &  $455.\  x_1^{3}x_2^{5}x_3^{2}x_4^{25}$ &  $456.\  x_2^{3}x_3^{5}x_4^{3}x_5^{24}$\cr  
$457.\  x_1^{3}x_3^{5}x_4^{3}x_5^{24}$ &  $458.\  x_1^{3}x_2^{5}x_4^{3}x_5^{24}$ &  $459.\  x_1^{3}x_2^{5}x_3^{3}x_5^{24}$ &  $460.\  x_1^{3}x_2^{5}x_3^{3}x_4^{24}$\cr  
\end{tabular}}

\bigskip 
We have $B_5(35) = B_5^0(35) \cup B_5^+(\omega_{(1)})\cup B_5(\omega_{(4)})\cup B_5(\omega_{(5)})\cup\psi(B_5(15))$.

\subsubsection{}
$B_5^+(\omega_{(1)})$ is the set of 160 monomials $a_t = a_{35,t}, \ 1 \leqslant t \leqslant 160$:

\medskip
 \centerline{\begin{tabular}{lll}
$1.\  x_1x_2x_3^{2}x_4^{3}x_5^{28}$ &  $2.\  x_1x_2x_3^{2}x_4^{4}x_5^{27}$ &  $3.\  x_1x_2x_3^{2}x_4^{7}x_5^{24}$\cr  
$4.\  x_1x_2x_3^{2}x_4^{28}x_5^{3}$ &  $5.\  x_1x_2x_3^{3}x_4^{2}x_5^{28}$ &  $6.\  x_1x_2x_3^{7}x_4^{2}x_5^{24}$\cr  
$7.\  x_1x_2^{3}x_3x_4^{2}x_5^{28}$ &  $8.\  x_1x_2^{7}x_3x_4^{2}x_5^{24}$ &  $9.\  x_1^{3}x_2x_3x_4^{2}x_5^{28}$\cr  
$10.\  x_1^{7}x_2x_3x_4^{2}x_5^{24}$ &  $11.\  x_1x_2x_3^{2}x_4^{5}x_5^{26}$ &  $12.\  x_1x_2x_3^{2}x_4^{6}x_5^{25}$\cr  
$13.\  x_1x_2x_3^{6}x_4x_5^{26}$ &  $14.\  x_1x_2x_3^{6}x_4^{26}x_5$ &  $15.\  x_1x_2^{2}x_3x_4^{5}x_5^{26}$\cr  
$16.\  x_1x_2^{2}x_3x_4^{6}x_5^{25}$ &  $17.\  x_1x_2^{2}x_3^{5}x_4^{2}x_5^{25}$ &  $18.\  x_1x_2^{2}x_3^{5}x_4^{25}x_5^{2}$\cr  
$19.\  x_1x_2^{3}x_3^{3}x_4^{4}x_5^{24}$ &  $20.\  x_1x_2^{3}x_3^{4}x_4^{2}x_5^{25}$ &  $21.\  x_1x_2^{3}x_3^{4}x_4^{25}x_5^{2}$\cr  
$22.\  x_1x_2^{3}x_3^{5}x_4^{2}x_5^{24}$ &  $23.\  x_1x_2^{3}x_3^{5}x_4^{24}x_5^{2}$ &  $24.\  x_1^{3}x_2x_3^{3}x_4^{4}x_5^{24}$\cr  
$25.\  x_1^{3}x_2x_3^{4}x_4^{2}x_5^{25}$ &  $26.\  x_1^{3}x_2x_3^{4}x_4^{25}x_5^{2}$ &  $27.\  x_1^{3}x_2x_3^{5}x_4^{2}x_5^{24}$\cr  
$28.\  x_1^{3}x_2x_3^{5}x_4^{24}x_5^{2}$ &  $29.\  x_1^{3}x_2^{3}x_3x_4^{4}x_5^{24}$ &  $30.\  x_1^{3}x_2^{3}x_3^{4}x_4x_5^{24}$\cr  
$31.\  x_1^{3}x_2^{3}x_3^{4}x_4^{24}x_5$ &  $32.\  x_1x_2^{2}x_3x_4^{3}x_5^{28}$ &  $33.\  x_1x_2^{2}x_3x_4^{28}x_5^{3}$\cr  
$34.\  x_1x_2^{2}x_3^{3}x_4x_5^{28}$ &  $35.\  x_1x_2^{2}x_3^{3}x_4^{28}x_5$ &  $36.\  x_1x_2^{2}x_3^{28}x_4x_5^{3}$\cr  $37.\  x_1x_2^{2}x_3^{28}x_4^{3}x_5$ &  $38.\  x_1x_2^{2}x_3x_4^{4}x_5^{27}$ &  $39.\  x_1x_2^{2}x_3x_4^{7}x_5^{24}$\cr  
$40.\  x_1x_2^{2}x_3^{4}x_4x_5^{27}$ &  $41.\  x_1x_2^{2}x_3^{4}x_4^{27}x_5$ &  $42.\  x_1x_2^{2}x_3^{7}x_4x_5^{24}$\cr  
$43.\  x_1x_2^{2}x_3^{7}x_4^{24}x_5$ &  $44.\  x_1x_2^{3}x_3^{2}x_4x_5^{28}$ &  $45.\  x_1x_2^{3}x_3^{2}x_4^{28}x_5$\cr  
$46.\  x_1x_2^{7}x_3^{2}x_4x_5^{24}$ &  $47.\  x_1x_2^{7}x_3^{2}x_4^{24}x_5$ &  $48.\  x_1^{3}x_2x_3^{2}x_4x_5^{28}$\cr  
$49.\  x_1^{3}x_2x_3^{2}x_4^{28}x_5$ &  $50.\  x_1^{7}x_2x_3^{2}x_4x_5^{24}$ &  $51.\  x_1^{7}x_2x_3^{2}x_4^{24}x_5$\cr  
$52.\  x_1x_2^{2}x_3^{3}x_4^{4}x_5^{25}$ &  $53.\  x_1x_2^{2}x_3^{3}x_4^{5}x_5^{24}$ &  $54.\  x_1x_2^{2}x_3^{4}x_4^{3}x_5^{25}$\cr  
\end{tabular}}
\centerline{\begin{tabular}{lll}
$55.\  x_1x_2^{2}x_3^{4}x_4^{25}x_5^{3}$ &  $56.\  x_1x_2^{2}x_3^{5}x_4^{3}x_5^{24}$ &  $57.\  x_1x_2^{2}x_3^{5}x_4^{24}x_5^{3}$\cr  
$58.\  x_1x_2^{3}x_3^{2}x_4^{4}x_5^{25}$ &  $59.\  x_1x_2^{3}x_3^{2}x_4^{5}x_5^{24}$ &  $60.\  x_1x_2^{3}x_3^{4}x_4^{3}x_5^{24}$\cr  
$61.\  x_1x_2^{3}x_3^{4}x_4^{24}x_5^{3}$ &  $62.\  x_1^{3}x_2x_3^{2}x_4^{4}x_5^{25}$ &  $63.\  x_1^{3}x_2x_3^{2}x_4^{5}x_5^{24}$\cr  
$64.\  x_1^{3}x_2x_3^{4}x_4^{3}x_5^{24}$ &  $65.\  x_1^{3}x_2x_3^{4}x_4^{24}x_5^{3}$ &  $66.\  x_1x_2^{2}x_3^{4}x_4^{9}x_5^{19}$\cr  
$67.\  x_1x_2^{2}x_3^{4}x_4^{11}x_5^{17}$ &  $68.\  x_1x_2^{2}x_3^{5}x_4^{8}x_5^{19}$ &  $69.\  x_1x_2^{2}x_3^{5}x_4^{11}x_5^{16}$\cr  
$70.\  x_1x_2^{2}x_3^{7}x_4^{8}x_5^{17}$ &  $71.\  x_1x_2^{2}x_3^{7}x_4^{9}x_5^{16}$ &  $72.\  x_1x_2^{3}x_3^{4}x_4^{8}x_5^{19}$\cr  
$73.\  x_1x_2^{3}x_3^{4}x_4^{11}x_5^{16}$ &  $74.\  x_1x_2^{3}x_3^{7}x_4^{8}x_5^{16}$ &  $75.\  x_1x_2^{7}x_3^{2}x_4^{8}x_5^{17}$\cr  
$76.\  x_1x_2^{7}x_3^{2}x_4^{9}x_5^{16}$ &  $77.\  x_1x_2^{7}x_3^{3}x_4^{8}x_5^{16}$ &  $78.\  x_1^{3}x_2x_3^{4}x_4^{8}x_5^{19}$\cr  
$79.\  x_1^{3}x_2x_3^{4}x_4^{11}x_5^{16}$ &  $80.\  x_1^{3}x_2x_3^{7}x_4^{8}x_5^{16}$ &  $81.\  x_1^{3}x_2^{7}x_3x_4^{8}x_5^{16}$\cr  
$82.\  x_1^{7}x_2x_3^{2}x_4^{8}x_5^{17}$ &  $83.\  x_1^{7}x_2x_3^{2}x_4^{9}x_5^{16}$ &  $84.\  x_1^{7}x_2x_3^{3}x_4^{8}x_5^{16}$\cr  
$85.\  x_1^{7}x_2^{3}x_3x_4^{8}x_5^{16}$ &  $86.\  x_1x_2x_3x_4^{2}x_5^{30}$ &  $87.\  x_1x_2x_3x_4^{6}x_5^{26}$\cr  
$88.\  x_1x_2x_3x_4^{30}x_5^{2}$ &  $89.\  x_1x_2x_3^{2}x_4x_5^{30}$ &  $90.\  x_1x_2x_3^{2}x_4^{2}x_5^{29}$\cr  $91.\  x_1x_2x_3^{2}x_4^{29}x_5^{2}$ &  $92.\  x_1x_2x_3^{2}x_4^{30}x_5$ &  $93.\  x_1x_2x_3^{3}x_4^{4}x_5^{26}$\cr  
$94.\  x_1x_2x_3^{3}x_4^{6}x_5^{24}$ &  $95.\  x_1x_2x_3^{3}x_4^{28}x_5^{2}$ &  $96.\  x_1x_2x_3^{6}x_4^{2}x_5^{25}$\cr  
$97.\  x_1x_2x_3^{6}x_4^{3}x_5^{24}$ &  $98.\  x_1x_2x_3^{30}x_4x_5^{2}$ &  $99.\  x_1x_2x_3^{30}x_4^{2}x_5$\cr  
$100.\  x_1x_2^{2}x_3x_4x_5^{30}$ &  $101.\  x_1x_2^{2}x_3x_4^{2}x_5^{29}$ &  $102.\  x_1x_2^{2}x_3x_4^{29}x_5^{2}$\cr  
$103.\  x_1x_2^{2}x_3x_4^{30}x_5$ &  $104.\  x_1x_2^{2}x_3^{5}x_4x_5^{26}$ &  $105.\  x_1x_2^{2}x_3^{5}x_4^{9}x_5^{18}$\cr  
$106.\  x_1x_2^{2}x_3^{5}x_4^{10}x_5^{17}$ &  $107.\  x_1x_2^{2}x_3^{5}x_4^{26}x_5$ &  $108.\  x_1x_2^{2}x_3^{29}x_4x_5^{2}$\cr  
$109.\  x_1x_2^{2}x_3^{29}x_4^{2}x_5$ &  $110.\  x_1x_2^{3}x_3x_4^{4}x_5^{26}$ &  $111.\  x_1x_2^{3}x_3x_4^{6}x_5^{24}$\cr  
$112.\  x_1x_2^{3}x_3x_4^{28}x_5^{2}$ &  $113.\  x_1x_2^{3}x_3^{4}x_4x_5^{26}$ &  $114.\  x_1x_2^{3}x_3^{4}x_4^{9}x_5^{18}$\cr  
$115.\  x_1x_2^{3}x_3^{4}x_4^{10}x_5^{17}$ &  $116.\  x_1x_2^{3}x_3^{4}x_4^{26}x_5$ &  $117.\  x_1x_2^{3}x_3^{5}x_4^{8}x_5^{18}$\cr  
$118.\  x_1x_2^{3}x_3^{5}x_4^{10}x_5^{16}$ &  $119.\  x_1x_2^{3}x_3^{6}x_4x_5^{24}$ &  $120.\  x_1x_2^{3}x_3^{6}x_4^{8}x_5^{17}$\cr  
$121.\  x_1x_2^{3}x_3^{6}x_4^{9}x_5^{16}$ &  $122.\  x_1x_2^{3}x_3^{6}x_4^{24}x_5$ &  $123.\  x_1x_2^{3}x_3^{28}x_4x_5^{2}$\cr  
$124.\  x_1x_2^{3}x_3^{28}x_4^{2}x_5$ &  $125.\  x_1x_2^{6}x_3x_4x_5^{26}$ &  $126.\  x_1x_2^{6}x_3x_4^{2}x_5^{25}$\cr  
$127.\  x_1x_2^{6}x_3x_4^{3}x_5^{24}$ &  $128.\  x_1x_2^{6}x_3x_4^{26}x_5$ &  $129.\  x_1x_2^{6}x_3^{3}x_4x_5^{24}$\cr  
$130.\  x_1x_2^{30}x_3x_4x_5^{2}$ &  $131.\  x_1x_2^{30}x_3x_4^{2}x_5$ &  $132.\  x_1^{3}x_2x_3x_4^{4}x_5^{26}$\cr  
$133.\  x_1^{3}x_2x_3x_4^{6}x_5^{24}$ &  $134.\  x_1^{3}x_2x_3x_4^{28}x_5^{2}$ &  $135.\  x_1^{3}x_2x_3^{4}x_4x_5^{26}$\cr  
$136.\  x_1^{3}x_2x_3^{4}x_4^{9}x_5^{18}$ &  $137.\  x_1^{3}x_2x_3^{4}x_4^{10}x_5^{17}$ &  $138.\  x_1^{3}x_2x_3^{4}x_4^{26}x_5$\cr  
$139.\  x_1^{3}x_2x_3^{5}x_4^{8}x_5^{18}$ &  $140.\  x_1^{3}x_2x_3^{5}x_4^{10}x_5^{16}$ &  $141.\  x_1^{3}x_2x_3^{6}x_4x_5^{24}$\cr  
$142.\  x_1^{3}x_2x_3^{6}x_4^{8}x_5^{17}$ &  $143.\  x_1^{3}x_2x_3^{6}x_4^{9}x_5^{16}$ &  $144.\  x_1^{3}x_2x_3^{6}x_4^{24}x_5$\cr  
$145.\  x_1^{3}x_2x_3^{28}x_4x_5^{2}$ &  $146.\  x_1^{3}x_2x_3^{28}x_4^{2}x_5$ &  $147.\  x_1^{3}x_2^{3}x_3^{4}x_4^{8}x_5^{17}$\cr  
$148.\  x_1^{3}x_2^{3}x_3^{4}x_4^{9}x_5^{16}$ &  $149.\  x_1^{3}x_2^{3}x_3^{5}x_4^{8}x_5^{16}$ &  $150.\  x_1^{3}x_2^{4}x_3x_4x_5^{26}$\cr  
$151.\  x_1^{3}x_2^{4}x_3x_4^{2}x_5^{25}$ &  $152.\  x_1^{3}x_2^{4}x_3x_4^{3}x_5^{24}$ &  $153.\  x_1^{3}x_2^{4}x_3x_4^{26}x_5$\cr  
$154.\  x_1^{3}x_2^{4}x_3^{3}x_4x_5^{24}$ &  $155.\  x_1^{3}x_2^{5}x_3x_4^{2}x_5^{24}$ &  $156.\  x_1^{3}x_2^{5}x_3^{2}x_4x_5^{24}$\cr  
$157.\  x_1^{3}x_2^{5}x_3^{2}x_4^{8}x_5^{17}$ &  $158.\  x_1^{3}x_2^{5}x_3^{2}x_4^{9}x_5^{16}$ &  $159.\  x_1^{3}x_2^{5}x_3^{2}x_4^{24}x_5$\cr  
$160.\  x_1^{3}x_2^{5}x_3^{3}x_4^{8}x_5^{16}$ &  
\end{tabular}}

\medskip 
\subsubsection{}
$B_5^+(\omega_{(4)})$  is the set of 50 monomials $$a_t = a_{35,t},\ 161 \leqslant t \leqslant 210$$

\medskip
\centerline{\begin{tabular}{lll}
161. $x_1x_2^{2}x_3^{7}x_4^{11}x_5^{14} $& 162. $x_1x_2^{7}x_3^{2}x_4^{11}x_5^{14} $& 163. $x_1x_2^{7}x_3^{11}x_4^{2}x_5^{14} $\cr  
164. $x_1x_2^{7}x_3^{11}x_4^{14}x_5^{2} $& 165. $x_1^{7}x_2x_3^{2}x_4^{11}x_5^{14} $& 166. $x_1^{7}x_2x_3^{11}x_4^{2}x_5^{14} $\cr  
167. $x_1^{7}x_2x_3^{11}x_4^{14}x_5^{2} $& 168. $x_1^{7}x_2^{11}x_3x_4^{2}x_5^{14} $& 169. $x_1^{7}x_2^{11}x_3x_4^{14}x_5^{2} $\cr  
170. $x_1^{7}x_2^{11}x_3^{13}x_4^{2}x_5^{2} $& 171. $x_1x_2^{3}x_3^{6}x_4^{11}x_5^{14} $& 172. $x_1^{3}x_2x_3^{6}x_4^{11}x_5^{14} $\cr  
173. $x_1x_2^{3}x_3^{7}x_4^{10}x_5^{14} $& 174. $x_1x_2^{3}x_3^{7}x_4^{14}x_5^{10} $& 175. $x_1x_2^{7}x_3^{3}x_4^{10}x_5^{14} $\cr  
176. $x_1x_2^{7}x_3^{3}x_4^{14}x_5^{10} $& 177. $x_1^{3}x_2x_3^{7}x_4^{10}x_5^{14} $& 178. $x_1^{3}x_2x_3^{7}x_4^{14}x_5^{10} $\cr  
179. $x_1^{3}x_2^{7}x_3x_4^{10}x_5^{14} $& 180. $x_1^{3}x_2^{7}x_3x_4^{14}x_5^{10} $& 181. $x_1^{7}x_2x_3^{3}x_4^{10}x_5^{14} $\cr  
182. $x_1^{7}x_2x_3^{3}x_4^{14}x_5^{10} $& 183. $x_1^{7}x_2^{3}x_3x_4^{10}x_5^{14} $& 184. $x_1^{7}x_2^{3}x_3x_4^{14}x_5^{10} $\cr  
\end{tabular}}
\centerline{\begin{tabular}{lll}
185. $x_1x_2^{7}x_3^{11}x_4^{6}x_5^{10} $& 186. $x_1^{7}x_2x_3^{11}x_4^{6}x_5^{10} $& 187. $x_1^{7}x_2^{11}x_3x_4^{6}x_5^{10} $\cr  
188. $x_1^{3}x_2^{5}x_3^{2}x_4^{11}x_5^{14} $& 189. $x_1^{3}x_2^{5}x_3^{11}x_4^{2}x_5^{14} $& 190. $x_1^{3}x_2^{5}x_3^{11}x_4^{14}x_5^{2} $\cr  
191. $x_1^{3}x_2^{7}x_3^{9}x_4^{2}x_5^{14} $& 192. $x_1^{3}x_2^{7}x_3^{9}x_4^{14}x_5^{2} $& 193. $x_1^{7}x_2^{3}x_3^{9}x_4^{2}x_5^{14} $\cr  
194. $x_1^{7}x_2^{3}x_3^{9}x_4^{14}x_5^{2} $& 195. $x_1^{3}x_2^{7}x_3^{13}x_4^{2}x_5^{10} $& 196. $x_1^{3}x_2^{7}x_3^{13}x_4^{10}x_5^{2} $\cr  
197. $x_1^{7}x_2^{3}x_3^{13}x_4^{2}x_5^{10} $& 198. $x_1^{7}x_2^{3}x_3^{13}x_4^{10}x_5^{2} $& 199. $x_1^{7}x_2^{11}x_3^{5}x_4^{2}x_5^{10} $\cr  
200. $x_1^{7}x_2^{11}x_3^{5}x_4^{10}x_5^{2} $& 201. $x_1^{3}x_2^{3}x_3^{5}x_4^{10}x_5^{14} $& 202. $x_1^{3}x_2^{3}x_3^{5}x_4^{14}x_5^{10} $\cr  
203. $x_1^{3}x_2^{5}x_3^{3}x_4^{10}x_5^{14} $& 204. $x_1^{3}x_2^{5}x_3^{3}x_4^{14}x_5^{10} $& 205. $x_1^{3}x_2^{3}x_3^{13}x_4^{6}x_5^{10} $\cr  
206. $x_1^{3}x_2^{5}x_3^{11}x_4^{6}x_5^{10} $& 207. $x_1^{3}x_2^{7}x_3^{5}x_4^{10}x_5^{10} $& 208. $x_1^{7}x_2^{3}x_3^{5}x_4^{10}x_5^{10} $\cr  
209. $x_1^{3}x_2^{7}x_3^{9}x_4^{6}x_5^{10} $& 210. $x_1^{7}x_2^{3}x_3^{9}x_4^{6}x_5^{10} $&\cr
\end{tabular}}

\medskip 
\subsubsection{}
$B_5^+(\omega_{(5)})$  is the set of 15 monomials $$a_t = a_{35,t},\ 211 \leqslant t \leqslant 225$$

\medskip
\centerline{\begin{tabular}{lll}
211. $x_1^{3}x_2^{5}x_3^{6}x_4^{6}x_5^{15} $& 212. $x_1^{3}x_2^{5}x_3^{6}x_4^{15}x_5^{6} $& 213. $x_1^{3}x_2^{5}x_3^{15}x_4^{6}x_5^{6} $\cr  
214. $x_1^{3}x_2^{15}x_3^{5}x_4^{6}x_5^{6} $& 215. $x_1^{15}x_2^{3}x_3^{5}x_4^{6}x_5^{6} $& 216. $x_1^{3}x_2^{5}x_3^{6}x_4^{7}x_5^{14} $\cr  
217. $x_1^{3}x_2^{5}x_3^{7}x_4^{6}x_5^{14} $& 218. $x_1^{3}x_2^{5}x_3^{7}x_4^{14}x_5^{6} $& 219. $x_1^{3}x_2^{7}x_3^{5}x_4^{6}x_5^{14} $\cr  
220. $x_1^{3}x_2^{7}x_3^{5}x_4^{14}x_5^{6} $& 221. $x_1^{3}x_2^{7}x_3^{13}x_4^{6}x_5^{6} $& 222. $x_1^{7}x_2^{3}x_3^{5}x_4^{6}x_5^{14} $\cr  
223. $x_1^{7}x_2^{3}x_3^{5}x_4^{14}x_5^{6} $& 224. $x_1^{7}x_2^{3}x_3^{13}x_4^{6}x_5^{6} $& 225. $x_1^{7}x_2^{11}x_3^{5}x_4^{6}x_5^{6} $\cr 
\end{tabular}}

\medskip
We have $f(B_4(35)) = QP_5^0(\omega_{(1)})$ and $|f(B_4(35))| = 460$. Hence, 
$$\dim QP_5(\omega_{(1)}) = 620.$$

\subsection{$GL_5$-invariants of $\mathbb F_2\otimes_{\mathcal A}P_5$ in degree $15$}\label{ss67}\

\medskip
$(\mathbb F_2\otimes_{\mathcal A}P_5)_{15}^{GL_5} = \langle [p], [q]\rangle$, where
\begin{align*}p & = x_1^{15} + x_2^{15} + x_3^{15} + x_4^{15} + x_5^{15} + x_1x_2^{14} + x_1x_3^{14} + x_1x_4^{14} + x_1x_5^{14}\\
&\quad + x_2x_3^{14} + x_2x_4^{14} + x_2x_5^{14} + x_3x_4^{14} + x_3x_5^{14} + x_4x_5^{14} + x_1x_2^{2}x_3^{12}\\
&\quad + x_1x_2^{2}x_4^{12} + x_1x_2^{2}x_5^{12} + x_1x_3^{2}x_4^{12} + x_1x_3^{2}x_5^{12} + x_1x_4^{2}x_5^{12} + x_2x_3^{2}x_4^{12}\\
&\quad + x_2x_3^{2}x_5^{12} + x_2x_4^{2}x_5^{12} + x_3x_4^{2}x_5^{12} + x_1x_2^{2}x_3^{4}x_4^{8} + x_1x_2^{2}x_3^{4}x_5^{8}\\
&\quad + x_1x_2^{2}x_4^{4}x_5^{8} + x_1x_3^{2}x_4^{4}x_5^{8} + x_2x_3^{2}x_4^{4}x_5^{8} + x_1x_2^{2}x_3^{4}x_4^{4}x_5^{4}, \\
q&= x_1x_2x_3x_4^{6}x_5^{6} + x_1x_2x_3^{6}x_4x_5^{6} + x_1x_2x_3^{6}x_4^{6}x_5 + x_1x_2^{6}x_3x_4x_5^{6}\\
&\quad + x_1x_2^{6}x_3x_4^{6}x_5 + x_1x_2^{3}x_3^{6}x_4x_5^{4} + x_1x_2^{3}x_3^{6}x_4^{4}x_5 + x_1x_2^{6}x_3^{3}x_4x_5^{4}\\
&\quad + x_1x_2^{6}x_3^{3}x_4^{4}x_5 + x_1^{3}x_2x_3x_4^{4}x_5^{6} + x_1^{3}x_2x_3x_4^{6}x_5^{4} + x_1^{3}x_2x_3^{4}x_4x_5^{6}\\
&\quad + x_1^{3}x_2x_3^{4}x_4^{6}x_5 + x_1^{3}x_2^{4}x_3x_4x_5^{6} + x_1^{3}x_2^{4}x_3x_4^{6}x_5 + x_1x_2^{3}x_3^{3}x_4^{4}x_5^{4}\\
&\quad + x_1^{3}x_2x_3^{3}x_4^{4}x_5^{4} + x_1^{3}x_2^{3}x_3x_4^{4}x_5^{4} + x_1^{3}x_2^{3}x_3^{4}x_4x_5^{4} + x_1^{3}x_2^{3}x_3^{4}x_4^{4}x_5\\
&\quad + x_1^{3}x_2^{4}x_3^{3}x_4x_5^{4} + x_1^{3}x_2^{4}x_3^{3}x_4^{4}x_5.
\end{align*}

\section*{Acknowledgment} 
The first version of this paper was written when the author was visiting Viet Nam Institute for Advanced Study in Mathematics (VIASM). He would like to thank the VIASM for supporting the visit and kind hospitality. 

The author was supported in part by the National Foundation for Science and Technology Development (NAFOSTED) of Viet Nam under the grant number 101.04-2017.05.

\bigskip
{}

\end{document}